\documentclass[10pt, a4paper, english]{smfart}

\usepackage[utf8]{inputenc}
\usepackage[english]{babel}
\usepackage[T1]{fontenc}
\usepackage{lmodern}

\usepackage[hidelinks,pdfencoding=unicode]{hyperref}
\usepackage[shortcuts]{extdash}
\usepackage{xpatch}

\usepackage{amsmath}
\usepackage{amssymb}
\usepackage{stmaryrd}
\usepackage{mathtools} 

\usepackage[all,2cell]{xy}
\UseAllTwocells
\SilentMatrices
\SelectTips{cm}{10}

\usepackage{enumitem}
\setlist{nosep}
\setlist[enumerate, 1]{label={\rm (}\emph{\alph*}{\rm )}}
\setlist[enumerate, 2]{label={\rm (}\emph{\alph{enumi}.\arabic*}{\rm )}}

\makeatletter

\patchcmd\@begintheorem
  {\ignorespaces}
  {\hypertarget{\csname @currentHref\endcsname}{}\ignorespaces}
  {}{}

\patchcmd{\smf@captionsenglish}{2000}{2020}{}{}

\makeatother

\usepackage{xspace}

\newif\iffr

\newcommand\fren[2]{\iffr #1\else #2\fi}

\newcommand\journal\emph

\renewcommand\and{\fren{et}{and}\xspace}

\theoremstyle{plain}
\newtheorem{theorem}{Theorem}[section]
\newtheorem*{theorem*}{Theorem}
\newtheorem{proposition}[theorem]{Proposition}
\newtheorem{lemma}[theorem]{Lemma}
\newtheorem{corollary}[theorem]{Corollary}
\theoremstyle{remark}
\newtheorem{remark}[theorem]{Remark}

\newtheorem{example}[theorem]{Example}
\newtheorem{examples}[theorem]{Examples}
\theoremstyle{definition}

\let\paragraph\undefined
\newtheorem{paragraph}[theorem]{}

\theoremstyle{plain}

\numberwithin{equation}{theorem}

\makeatletter
\newif\ifsection
\preto\section{\sectiontrue}
\preto\subsection{\sectionfalse}
\xapptocmd\@sect{%
  \ifsection
    \numberwithin{theorem}{section}
  \else
    \numberwithin{theorem}{subsection}
  \fi
  \setcounter{theorem}{0}\relax}
  {}{}
\makeatother

\let\forlang\emph
\let\ndef\emph
\let\nbd\nobreakdash
\newcommand\noemph\emph

\def\xpoint{\futurelet\@let@token\@xpoint}
\def\@xpoint{%
  \ifx\@let@token.\else
    .%
  \fi
  \xspace}

\newcommand\resp[1]{{\rm (resp.}\ #1{\rm )}}
\newcommand\respd[2]{{\rm (resp.}\ #1, {\rm resp.}\ #2{\rm )}}

\newcommand\zbox[1]{\makebox[0pt][l]{#1}}
\newcommand\pbox[1]{\zbox{\quad#1}}

\newcommand\quadtext[1]{\quad\text{#1}\quad}

\newcommand\quadand{\quadtext{and}}
\newcommand\quador{\quadtext{or}}

\newcommand\quadiff{\quadtext{if and only if}}

\renewcommand\le\leqslant
\renewcommand\ge\geqslant
\renewcommand\epsilon\varepsilon
\newcommand\e\epsilon
\renewcommand\phi\varphi

\let\hookto\hookrightarrow
\let\xto\xrightarrow
\let\xot\xleftarrow
\newcommand\tod\Rightarrow
\newcommand\tot\Rrightarrow
\makeatletter
\providecommand*{\twoheadrightarrowfill@}{%
  \arrowfill@\relbar\relbar\twoheadrightarrow
}
\providecommand*{\twoheadleftarrowfill@}{%
  \arrowfill@\twoheadleftarrow\relbar\relbar
}
\providecommand*{\xtwoheadrightarrow}[2][]{%
  \ext@arrow 0579\twoheadrightarrowfill@{#1}{#2}%
}
\providecommand*{\xtwoheadleftarrow}[2][]{%
  \ext@arrow 5097\twoheadleftarrowfill@{#1}{#2}%
}
\makeatother
\newcommand\trivfib{\xtwoheadrightarrow{\raisebox{-1.5pt}[0pt][0pt]{$\scriptstyle
\sim$}}}

\newcommand\oo[1]{$\omega$\=/}
\newcommand\mn[3]{$(#1,#2)$\=/}
\newcommand\oon[2]{\mn{\omega}{#1}{#2}}

\newcommand\pdfmn[2]{$(#1, #2)$}

\newcommand{\Homi}{\operatorname{\kern.5truept\underline{\kern-.5truept\mathsf{Hom}\kern-.5truept}\kern1truept}}
\newcommand\Homir{\Homi^{\mathrm{r}}}
\newcommand\Homil{\Homi^{\mathrm{l}}}
\newcommand\boxprod{\otimes'}

\newcommand\boxjoin{\join'}
\newcommand\Lotimes{\otimes^{\mathbb{L}}}
\newcommand\Ljoin{\join^{\mathbb{L}}}
\newcommand\RHomil{\mathbb{R}\!\Homi^{\mathrm{l}}}
\newcommand\RHomir{\mathbb{R}\!\Homi^{\mathrm{r}}}

\newcommand{\Cat}{{\mathcal{C}\mspace{-2.mu}\it{at}}}
\newcommand{\Gpd}{{\mathcal{G}\mspace{-2.mu}\it{pd}}}
\newcommand{\nCat}[1]{%
\mathchoice
  {\hbox{$#1$-\kern1pt$\Cat$}}
  {\hbox{$#1$-\kern1pt$\Cat$}}
  {\hbox{$\scriptstyle #1$\raisebox{-0.5pt}{-}$\scriptstyle \Cat$}}
  {TODO}%
}
\newcommand{\nGpd}[1]{%
\mathchoice
  {\hbox{$#1$-\kern1pt$\Gpd$}}
  {\hbox{$#1$-\kern1pt$\Gpd$}}
  {\hbox{$\scriptstyle #1$\raisebox{-0.5pt}{-}$\scriptstyle \Gpd$}}
  {TODO}%
}
\newcommand{\ooCat}{\nCat{\omega}}
\newcommand{\mnCat}[2]{\nCat{(#1,#2)}}

\newcommand{\ooGpd}{\nGpd{\omega}}
\renewcommand\C{\mathcal{C}}
\newcommand\D{\mathcal{D}}
\newcommand\M{\mathcal{M}}
\newcommand\W{\mathcal{W}}
\newcommand\Ct{\widetilde{\mathcal{C}}}
\newcommand\Wt{\widetilde{\mathcal{W}}}
\newcommand\Kt{\widetilde{K}}
\newcommand\Ft{\widetilde{F}}
\newcommand\twAr{\operatorname{\mathrm{Tw}}}
\newcommand\comma\downarrow

\newcommand\N{\mathbb{N}}
\newcommand\Z{\mathbb{Z}}

\newcommand\ADC{\mathcal{C}_{\mathrm{ad}}}
\newcommand\atom[1]{\langle #1 \rangle}
\newcommand\Stf{{\mathcal{S}t}_{\mathrm{f}}}

\newcommand\restrict[2]{{#1}_{|#2}}

\newcommand{\Hom}{\operatorname{\mathsf{Hom}}}
\newcommand\limind\varinjlim
\newcommand\limproj\varprojlim
\newcommand\op{\mathrm{op}}
\newcommand\id[1]{1_{#1}}

\newcommand\co{\mathrm{co}}
\newcommand\HomOpLax{\Homi_{\mathrm{oplax}}}
\newcommand\HomLax{\Homi_{\mathrm{lax}}}
\newcommand\join\star
\newcommand\rotimes[2]{\otimes_{#1,#2}}
\newcommand\rjoin[1]{\join_{#1}}
\renewcommand\emptyset\varnothing
\newcommand{\cotr}[2]{\mathchoice
  {\raise -1.8pt\vbox{\hbox{$#2\backslash$}}#1}
  {\raise -1.8pt\vbox{\hbox{$#2\backslash$}}#1}
  {\raise -1.8pt\vbox{\hbox{$\scriptstyle#2\backslash$}}#1}
  {\raise -1.8pt\vbox{\hbox{$\scriptscriptstyle#2\backslash$}}#1}}
\newcommand{\tr}[2]{\mathchoice
  {#1\raise -1.8pt\vbox{\hbox{$\kern -.8pt/#2$}}}
  {#1\raise -1.8pt\vbox{\hbox{$\kern -.8pt/#2$}}\kern .8pt}
  {#1\raise -1.8pt\vbox{\hbox{$\scriptstyle\kern -.8pt /#2$}}}
  {#1\raise -1.8pt\vbox{\hbox{$\scriptscriptstyle\kern -.8pt /#2$}}}}
\newcommand{\cotrp}[2]{\mathchoice
  {\raise -1.8pt\vbox{\hbox{$#2\backslash\mkern-3mu'\mkern1mu$}}#1}
  {\raise -1.8pt\vbox{\hbox{$#2\backslash'$}}#1}
  {\raise -1.8pt\vbox{\hbox{$\scriptstyle#2\backslash'$}}#1}
  {\raise -1.8pt\vbox{\hbox{$\scriptscriptstyle#2\backslash'$}}#1}}
\newcommand{\trp}[2]{\mathchoice
  {#1\raise -1.8pt\vbox{\hbox{$\kern -.8pt/'\mkern-2mu#2$}}}
  {#1\raise -1.8pt\vbox{\hbox{$\kern -.8pt/'#2$}}\kern .8pt}
  {#1\raise -1.8pt\vbox{\hbox{$\scriptstyle\kern -.8pt /'#2$}}}
  {#1\raise -1.8pt\vbox{\hbox{$\scriptscriptstyle\kern -.8pt /'#2$}}}}
\newcommand{\trm}[2]{\mathchoice
  {#1\raise -1.8pt\vbox{\hbox{$\kern -.8pt\!\stackrel{\,\,\,\rm co}{/}\!\!#2$}}}
  {#1\raise -1.8pt\vbox{\hbox{$\kern -.8pt\!\stackrel{\,\,\,\rm co}{/}\!\!#2$}}\kern .8pt}
  {#1\raise -1.8pt\vbox{\hbox{$\scriptstyle\kern -.8pt\!\!\stackrel{\,\,\rm co}{/}\!\!#2$}}\kern .8pt}
  {TODO}}
\newcommand\Cyl{\operatorname{\Gamma}}
\newcommand\CylRev{\operatorname{\Gamma_{\!\mathrm{rev}}}}
\newcommand\CylInv{\operatorname{\Gamma_{\!\mathrm{inv}}}}
\newcommand\Dn[1]{\mathrm{D}_{#1}}
\newcommand\bDn[1]{\partial\Dn{#1}}
\newcommand\Deltan[1]{\varDelta_{#1}}
\newcommand\Rn[1]{\operatorname{R}_{#1}}
\newcommand\Jn[1]{\operatorname{J}_{#1}}
\newcommand\In[1]{\operatorname{I}_{#1}}

\newcommand\J{\Jn{1}}
\newcommand\winv{\bar}
\newcommand\comp\ast
\newcommand\compv{\comp_v}
\newcommand\reflect{r}

\newcommand\Ho{\operatorname{\mathrm{Ho}}}

\newcommand\funcell[1]{\langle #1 \rangle}

\frfalse

\author{Dimitri Ara}
\address{Aix~Marseille~Univ,~CNRS,~Centrale~Marseille,~I2M,~Marseille,~France}
\email{dimitri.ara@univ-amu.fr}
\urladdr{\href{http://www.i2m.univ-amu.fr/perso/dimitri.ara/}{http://www.i2m.univ-amu.fr/perso/dimitri.ara/}}

\author{Maxime Lucas}
\address{Galinette Team, Inria Rennes, LS2N, Nantes, France}
\email{maxime.lucas@inria.fr}
\urladdr{\href{http://lucas-webpage.gforge.inria.fr/}{http://lucas-webpage.gforge.inria.fr/}}

\title[The folk model category structure on
\oo-categories is monoidal]{The folk model category structure on\\ strict
\oo-categories is monoidal}

\begin{document}

\frontmatter

\begin{abstract}
  We prove that the folk model category structure on the category of strict
  \oo-categories, introduced by Lafont, Métayer and Worytkiewicz, is
  monoidal, first, for the Gray tensor product and, second, for the join of
  \oo-categories, introduced by the first author and Maltsiniotis. We
  moreover show that the Gray tensor product induces, by adjunction, a
  tensor product of strict $(m,n)$-categories and that this tensor product
  is also compatible with the folk model category structure. In particular,
  we get a monoidal model category structure on the category of strict
  \oo-groupoids. We prove that this monoidal model category structure
  satisfies the monoid axiom, so that the category of Gray monoids, studied
  by the second author, bears a natural model category structure.
\end{abstract}


\maketitle

\mainmatter

\section*{Introduction}

The category $\ooCat$ of strict \oo-categories, that we shall simply call
\oo-categories in this paper, is endowed with a model category structure,
introduced by Lafont, \hbox{Métayer} and Worytkiewicz \cite{Folk}, known as
the \emph{folk} model category structure. The weak equivalences of this
structure are the equivalences of \oo-categories, higher dimensional
generalization of the equivalences of categories or of $2$-categories; the
cofibrant objects are the \oo-categories that are free in the sense of
polygraphs~\cite{MetCof}. This model category structure, which is also called
the canonical model category structure, is in some sense intrinsic to the
notion of \oo-categories.

On the other hand, the category $\ooCat$ is endowed with two non-trivial
monoidal category structures. The first one is the Gray tensor product
$\otimes$, sometimes called the lax Gray tensor product, first introduced by
Al-Agl and Steiner \cite{AlAglSteiner}, and then studied by Crans
\cite{CransThese}. This tensor product generalizes the tensor product of
$2$-categories introduced by Gray in \cite{GrayFCT}, hence its name. It is
somehow a lax version of the cartesian product. For instance, one has
  \[
    \Dn{1} \otimes \Dn{1} =
    \raisebox{1.75pc}{
    \xymatrix{
      \bullet \ar[r] \ar[d] &
      \bullet \ar[d] \\
      \bullet \ar[r] & \bullet
      \ar@{}[u];[l]_(.30){}="s"
      \ar@{}[u];[l]_(.70){}="t"
      \ar@2"s";"t"
      \pbox{,}
    }}
   \qquad
   \qquad
    \Dn{1} \otimes \Deltan{2} =
    \raisebox{1.75pc}{
     \xymatrix{
       \bullet
       \ar[r]^{}
       \ar[d]_{}
       &
       \bullet
       \ar[r]^{}
       \ar[d]_{}
       \ar@{}[ld]|-{}
       &
       \bullet
       \ar[d]^{}
       \\
       \bullet
       \ar[r]_{}
       &
       \bullet
       \ar[r]_{}
       &
       \bullet \pbox{,}
       \ar@{}[lu];[ll]_(0.30){}="s1"^(0.70){}="t1"
       \ar@2"s1";"t1"_{}
       \ar@{}[u];[l]_(0.30){}="s2"^(0.70){}="t2"
       \ar@2"s2";"t2"_{}
     }}
   \]
   \[
    \Dn{1} \otimes \Dn{2} =
    \raisebox{2.25pc}{
    \xymatrix@C=3pc@R=3pc{
      \bullet
      \ar@/^2ex/[r]^(0.70){}_{}="0"
      \ar@/_2ex/[r]_(0.70){}_{}="1"
      \ar[d]_{}="f"
      \ar@2"0";"1"
      &
      \bullet
      \ar[d] \\
      \bullet
      \ar@{.>}@/^2ex/[r]^(0.30){}_{}="0"
      \ar@/_2ex/[r]_(0.30){}_{}="1"
      \ar@{:>}"0";"1"_{}
      &
      \bullet
      \ar@{}[u];[l]_(.40){}="x"
      \ar@{}[u];[l]_(.60){}="y"
      \ar@<-1.5ex>@/_1ex/@{:>}"x";"y"_(0.60){}_{}="0"
      \ar@<1.5ex>@/^1ex/@2"x";"y"^(0.40){}_{}="1"
      \ar@{}"1";"0"_(.05){}="z"
      \ar@{}"1";"0"_(.95){}="t"
      \ar@3{>}"z";"t"_{}
      \pbox{,}
    }}
   \]
   where
   \[
   \Dn{1} = \xymatrix{\bullet \ar[r] & \bullet},
   \quad
   \Deltan{2} = \xymatrix{\bullet \ar[r] & \bullet \ar[r] & \bullet}
   \quadand
    \Dn{2} = \xymatrix@C=3pc@R=3pc{\bullet \ar@/^2.5ex/[r]_{}="0"
      \ar@/_2.5ex/[r]_{}="1" \ar@2"0";"1" & \bullet \pbox{.}}
   \]
In general, by iterating $n$ times the Gray tensor product with $\Dn{1}$
starting from $\Dn{0}$, one gets a lax cube of dimension $n$.
This (non-symmetric) tensor product defines a biclosed monoidal category
structure and the two associated internal $\Hom$ are related to higher lax
and oplax transformations. The second monoidal category structure is given
by the join of \oo-categories $\join$, introduced by the first author and
Maltsiniotis in~\cite{AraMaltsiJoint} to study slice \oo-categories in a
similar way as Joyal did for quasi-categories (see the introduction of
\cite{AraMaltsiJoint} for more details). This operation, inspired by the
topological join, is a higher dimensional lax version of the classical join
of categories. For instance, one has
  \[
      \Dn{0} \join \Dn{0} = \xymatrix{\bullet \ar[r] & \bullet} = \Dn{1}
      \pbox{,}
      \qquad
      \qquad
      \Dn{0} \join \Dn{1} =
      \raisebox{2.0pc}{
      \xymatrix@R=1pc{
        & \bullet \ar[dd] \\
        \bullet \ar[ru] \ar[rd]_(0.50){}="s" \\
        & \bullet \pbox{,}
        \ar@{}"s";[uu]_(0.15){}="ss"_(0.65){}="tt"
        \ar@2"ss";"tt"
      }}
  \]
  \[
      \Dn{1} \join \Dn{1} =
    \raisebox{1.5pc}{
    \xymatrix{
      \bullet \ar[r]_(0.60){}="03" \ar[d] \ar[dr]_{}="02"_(0.60){}="02'" &
      \bullet
      &
      \bullet \ar[r]_(0.40){}="03'" \ar[d]_{}="t3"
        &
      \bullet
      \\
      \bullet \ar[r] & \bullet \ar[u]_{}="s3"
      &
      \bullet \ar[r] \ar[ur]_{}="13"_(0.40){}="13'" & \bullet \ar[u] \pbox{.}
      \ar@{}"s3";"t3"_(0.20){}="ss3"_(0.80){}="tt3"
      \ar@3"ss3";"tt3"
      \ar@{}"13";[]_(0.05){}="s123"_(0.85){}="t123"
      \ar@2"s123";"t123"
      \ar@{}"02";[lll]_(0.05){}="s012"_(0.85){}="t012"
      \ar@2"s012";"t012"
      \ar@{}"03'";"13'"_(0.05){}="s013"_(0.85){}="t013"
      \ar@2"s013";"t013"
      \ar@{}"03";"02'"_(0.05){}="s023"_(0.85){}="t023"
      \ar@2"s023";"t023"
      \\
    }}
  \]
More generally, by iterating $n$ times the join with $\Dn{0}$ starting from
$\Dn{0}$, one gets Street's $n$-th oriental $\mathcal{O}_n$
\cite{StreetOrient}. The join only admits ``local
internal $\Hom$'', in some appropriate sense, that are given by
``generalized slice \oo-categories''. The Gray tensor product and the join
are two fundamental structures of the theory of \oo-categories.

\medskip

The main goal of this paper is to prove that both the Gray tensor product
and the join interact well with the folk model category structure or, more
precisely, that they both define a monoidal model category structure on
$\ooCat$ endowed with the folk model category structure. Concretely, this
means that if $i : X \to Y$ and $j : Z \to T$ are two folk cofibrations,
then their pushout-product, that is, the \oo-functor
\[
  i \boxprod j : Y \otimes Z \amalg_{X \otimes Z} X \otimes T
  \to Y \otimes T
\]
induced by the commutative square
\[
  \xymatrix{
    X \otimes Z \ar[d]_{i \otimes Z} \ar[r]^{X \otimes j}
    &
    X \otimes T \ar[d]^{i \otimes T}
    \\
    Y \otimes Z \ar[r]_{Y \otimes j}
    &
    Y \otimes T \pbox{,}
  }
\]
is a folk cofibration, and a folk trivial cofibration if moreover either $i$
or $j$ is a folk trivial cofibration; and likewise for the join. This
implies in particular that the Gray tensor product and the join can be
left-derived as functors of two variables.

Note that the fact that the pushout-product, for the Gray tensor product, of
two folk cofibrations is a folk cofibration was already proved by the second
author in~\cite{LucasThesis} by means of cubical \oo-categories. Moreover,
the particular case saying that the Gray tensor product of two cofibrant
\oo-categories is cofibrant was also established by \hbox{Hadzihasanovic}
in~\cite{HadziThesis}. Our proof, which is based on Steiner's theory of
augmented directed complexes~\cite{Steiner} and results of the first author
and Maltsiniotis about pushouts of these~\cite[Chapter 3]{AraMaltsiJoint},
is completely different and has the advantage to adapt easily to the case of
the join. The hard part in showing the compatibility of the
Gray tensor product and the join with the folk model category structure is
then to prove that the pushout-product of a folk cofibration and a folk trivial
cofibration is a folk trivial cofibration.

In the case of the Gray tensor product, we prove a more general result. Let
$\mnCat{m}{n}$, for $0 \le n \le m \le \omega$, be the category of $(m,
n)$-categories, that is, the category of (strict) $m$-categories whose
$k$-cells are strictly invertible as soon as~$k > n$. Denote by $r : \ooCat
\to \mnCat{m}{n}$ the left adjoint to the inclusion functor
$\mnCat{m}{n} \hookto \ooCat$. It follows from \cite{Folk} and
\cite{AraMetGpd} that the folk model category structure can be transferred
along this adjunction to $\mnCat{m}{n}$. We prove, first, that the Gray tensor
product induces, using $r$, a biclosed monoidal product for $(m,
n)$\nbd-categories and, second, that this Gray tensor product of $(m,
n)$-categories is compatible with the transferred model category structure
on $\mnCat{m}{n}$. In particular, in the case $n = 0$, we get a monoidal
model category structure on the category of (strict) $m$-groupoids. We prove
that this structure is symmetric and satisfies the so-called monoid axiom of
Schwede and Shipley \cite{SchwedeShipley}. This implies that the category of
Gray monoids, that is, of monoid objects in the category of \oo-groupoids
endowed with the Gray tensor product, bears a canonical model category
structure. This result was one of the motivating starting point of this
work, as the second author showed in~\cite{LucasThesis} that Gray monoids
provide a good framework for higher dimensional rewriting.

On our way to show these results, we prove several properties of independent
interest related to the Gray tensor product:
\begin{itemize}
  \item We prove that if $x$ is an $m$-cell of an \oo-category $X$ and $y$
    is an $n$-cell of an \oo-category $Y$, then the associated $(m+n)$-cell
    $x \otimes y$ is reversible (that is, weakly invertible) if either $x$
    or $y$ is reversible.
  \item We show that the analogous statement for strictly invertible cells
    holds. This implies that the tensor product of two \oo-groupoids is an
    \oo-groupoid.
  \item We prove that if $X$ is a cofibrant \oo-category, then $\J \otimes
    X$, where $\J$ is the \oo-category obtained by factorizing the
    codiagonal of the terminal \oo-category into a folk cofibration followed
    by a folk trivial fibration, is a cylinder object for~$X$ in the folk
    model category structure.
  \item We show that the invertible cells of the \oo-category $\HomOpLax(X,
    Y)$, defined by the adjunction
    \[
      \Hom_{\ooCat}(T \otimes X, Y) \simeq \Hom_{\ooCat}(T, \HomOpLax(X,
      Y)),
    \]
    are precisely the component-wise invertible higher oplax
    transformations. This implies that if $Y$ is an $(m, n)$-category, then
    so is $\HomOpLax(X, Y)$.
    \item We construct an \oo-functor
    \[ X \otimes (Y \join Z) \to (X \otimes Y) \join Z, \]
    natural in $X$, $Y$ and $Z$ in $\ooCat$, defining a tensorial strength
    on the functor ${-} \join Z$ for the Gray tensor product.
\end{itemize}

Finally, in an appendix, we prove that the ``local internal $\Homi$'' of the
join, the so-called generalized slices, can be right-derived as functors of
two variables. By ``local internal $\Homi$'', we mean the right adjoints of
the functors
\[
  \begin{split}
    \ooCat & \to \cotr{\ooCat}{X} \\
    Y & \mapsto (X \join Y, X \hookto X \join Y)
  \end{split}
  \qquad
  \begin{split}
    \ooCat & \to \cotr{\ooCat}{Y} \\
    X & \mapsto (X \join Y, Y \hookto X \join Y),
  \end{split}
\]
by opposition to the right adjoints of the functors
\[
  \begin{split}
    \ooCat & \to \ooCat \\
    Y & \mapsto X \join Y
  \end{split}
  \qquad \qquad
  \begin{split}
    \ooCat & \to \ooCat \\
    X & \mapsto X \join Y,
  \end{split}
\]
which do not exist in the case of the join. These two local internal
$\Homi$, like classical internal $\Homi$, can be promoted to functors of two
variables, but in the local case, we get functors from the twisted arrow
category:
\[
  \begin{split}
    \twAr(\ooCat) & \to \ooCat \\
    X \xto{u} Z & \mapsto \cotr{Z}{u}
  \end{split}
  \qquad \qquad
  \begin{split}
    \twAr(\ooCat) & \to \ooCat \\
    Y \xto{v} Z & \mapsto \trm{Z}{v}.
  \end{split}
\]
We prove, in the general setting of locally biclosed monoidal category
introduced in~\cite{AraMaltsiJoint}, that these functors can be
right-derived. This requires the use the theory of right simplicial
derivability structures of Kahn and Maltsiniotis \cite{KahnMaltsi} as the
twisted arrow category of a complete and cocomplete category is neither
finitely cocomplete nor finitely complete in general.

\medbreak
\goodbreak

We were unable to answer the following obvious question: is the tensor
product of two folk weak equivalences a folk weak equivalence? Of course, a
similar question can be asked for the join. We leave these two questions as
open problems.

\medskip

Our paper is organized as follows. In the first section, we recall the
definitions related to the folk model category structure on the category
$\ooCat$ of (strict) \oo-categories. In particular, we define reversible
cells (that is, weakly invertible cells). Using the Gray tensor product,
whose definition is recalled in the next section, we introduce oplax
transformations and reversible transformations. We recall the definition of
the \oo-category of cylinders and we end the section by introducing some
classical dualities of $\ooCat$.

The purpose of the second section is to recall the definition of the Gray
tensor product. We start by a summary of Steiner's theory of augmented
directed complexes~\cite{Steiner} and we use this theory to introduce,
following \cite{Steiner} and \cite{AraMaltsiJoint}, the Gray tensor
product and its associated internal $\Homi$, namely $\HomOpLax$ and
$\HomLax$.

The aim of the third section is to prove that the pushout-product, for
the Gray tensor product, of two folk cofibrations is a folk cofibration. We
start by recalling the notion of a rigid monomorphism of augmented directed
complexes with basis and some results from~\cite{AraMaltsiJoint} of
compatibility between pushouts of augmented directed complexes and pushouts
of \oo-categories. We then prove that the pushout-product, for the tensor
product of augmented directed complexes, of two rigid monomorphisms is a
rigid monomorphism. We then deduce the analogous result for \oo-categories
and folk cofibrations.

In the fourth section, we prove that if $X$ is a folk cofibrant \oo-category,
then $\J \otimes X$, where $\J$ is the \oo-category obtained by
factorizing the codiagonal of the terminal \oo-category into a cofibration
followed by a trivial fibration, is a cylinder object for $X$ in the folk
model category. On our way to do so, we prove that the tensor
product of a reversible (resp.~invertible) $m$-cell by any other $n$-cell is
reversible (resp.~invertible). We start by proving the case $m = 1$
providing explicit formulas and then prove the general case by induction.

In the fifth section, we end the proof of the fact that the Gray tensor
product makes of $\ooCat$ endowed with the folk model category structure a
monoidal model category. Our strategy is abstracted in a general lemma whose
main hypothesis, besides the fact that the pushout-product of two generating
cofibrations is a cofibration, is the fact that the tensor product of a
generating trivial cofibration and an object is a weak equivalence. We
prove this hypothesis for the Gray tensor product using results from the
previous section. We then prove some additional properties of the resulting
monoidal model category.

In the sixth section, we introduce the category $\mnCat{m}{n}$ of (strict)
$(m, n)$\=/categories and we study the interactions between the Gray tensor
product and these $(m, n)$\nbd-categories. We prove that the invertible
cells of the \oo-category $\HomOpLax(X, Y)$, defined by the adjunction
$\Hom_{\ooCat}(T \otimes X, Y) \simeq \Hom_{\ooCat}(T, \HomOpLax(X, Y))$,
are precisely the component-wise invertible higher oplax transformations. As
a consequence, we obtain that if $Y$ is an $(m, n)$-category, then so is
$\HomOpLax(X, Y)$. This implies by a result of Day \cite{DayRefl} that the
Gray tensor product induces, using the reflection functor $r : \ooCat \to
\mnCat{m}{n}$, a biclosed monoidal category structure on $\mnCat{m}{n}$. In
the case of $m$-groupoids (that is, the case $n = 0$), we show that the
resulting monoidal product is symmetric. We introduce the folk model
category structure on $\mnCat{m}{n}$, that is, the model category structure
obtained by transferring along $r$ the folk model category structure on
$\ooCat$, and we prove that the monoidal product on $\mnCat{m}{n}$ induced
by the Gray tensor product is compatible with this structure. In the case of
$m$-groupoids, we prove that the resulting monoidal model category
satisfies the monoid axiom of Schwede and Shipley \cite{SchwedeShipley}. As
a consequence, by results of Harper~\cite{Harper} and Muro~\cite{Muro}, we
obtain model category structures on the categories of algebras in $\ooGpd$
over a given non-symmetric operad in $\ooGpd$.

In the seventh section, we recall the definition of the join of
\oo-categories, introduced in \cite{AraMaltsiJoint}, and its associated
local internal $\Homi$, the generalized slices. We prove that the join makes of
$\ooCat$ endowed with the folk model category structure a monoidal model
category. The proof, that we only sketch, is very similar to the one for the
Gray tensor product, and only requires one additional tool: the existence of
an \oo-functor $X \otimes (Y \join Z) \to (X \otimes Y) \join Z$ defining a
tensorial strength on the functor ${-} \join Z$ for the Gray tensor product.

Finally, in an appendix, we recall the definition of a monoidal model
category and how in this setting the monoidal tensor and, in the biclosed
setting, the associated internal $\Homi$ can be derived as functors of two
variables. We then adapt this last result to the case of locally biclosed
monoidal products, introduced in~\cite{AraMaltsiJoint}, our example of
interest being the join of \oo-categories. More precisely, we prove that the
local internal $\Homi$ of such a monoidal product can be right-derived as
functors from the twisted arrow category. To do so, we endow the twisted
arrow category of a model category with a right simplicial derivability
structure in the sense of Kahn and Maltsiniotis~\cite{KahnMaltsi}, proving
that right simplicial derivability structures can be lifted along discrete
opfibrations.

\section{Preliminaries on the folk model category structure}

We will now describe the so-called ``folk'' model category structure
on $\ooCat$ introduced by Lafont, Métayer and Worytkiewicz in
\cite{Folk}. We start by describing the weak equivalences of this structure:
the equivalences of \oo-categories.

\begin{paragraph}
  We will denote by $\ooCat$ the category of strict \oo-categories and
  strict \oo-functors. As all the \oo-categories and \oo-functors in
  this paper will be strict, we will drop the adjective ``strict''
  from now on.  We will say that a cell of an \oo-category is
  \emph{trivial} if it is the identity on a cell of lower dimension. If
  $x$ is an $n$-cell of an \oo-category with~$n \ge 1$, we will denote
  by $1_x$ the identity on $x$, by $sx$ its source $(n-1)$-cell and by $tx$
  its target $(n-1)$-cell. If $x$ is an $n$-cell with $n \ge 0$, for $0 \le
  k \le n$, we will denote by $s_k(x)$ its iterated source $k$-cell and by
  $t_k(x)$ its iterated target $k$-cell.
\end{paragraph}

\begin{paragraph}\label{paragr:def_coind}
  Let $X$ be an \oo-category. By a \ndef{structure of reversibility} on
  $X$, we mean a set~$R$ of cells of $X$ such that, if $u : x \to y$ is an
  $n$-cell in $R$, then there exists an $n$-cell $\winv{u} : y \to x$
  and $(n+1)$-cells $\winv{r} \comp_n r \to \id{x}$ and $r \comp_n
  \winv{r} \to \id{y}$ all three in $R$. We say that an $n$-cell $u$ of $X$
  is \ndef{reversible} if $n \ge 1$ and if there exists a structure of
  reversibility $R$ on $X$ containing~$u$. A cell $\winv{u}$ in $R$ as in
  the definition of a structure of reversibility is then called \emph{a}
  \ndef{reverse} of $u$.

  If $C$ is a set of cells of $X$, to prove that every cell of $C$ is
  reversible, it suffices to produce, for every $n$-cell $u$ of $C$,
  a formula giving a reverse of $u$ \emph{assuming} that the $(n+1)$-cells
  of $C$ are reversible. Indeed, one can then consider the subcategory~$R$
  of~$X$ generated by the reversible cells of~$X$, the cells in $C$ and the
  cells given by the formulas, and show that the cells of $R$ form a
  structure of reversibility. This is sometimes called \ndef{reasoning by
  coinduction}.
\end{paragraph}

\begin{paragraph}
  An \oo-functor $f : X \to Y$ is an \ndef{equivalence of
  \oo-categories} or \ndef{folk weak equivalence} if:
  \begin{itemize}
  \item for every $0$-cell $y$ of $Y$, there exists a $0$-cell $x$ of
    $X$ and a reversible $1$-cell $f(x) \to y$ of $Y$,
  \item for every $n \ge 1$, every pair of parallel $(n-1)$-cells
    $x, x'$ of $X$ and every $n$-cell $v : f(x) \to f(x')$ of $Y$,
    there exists an $n$-cell $u : x \to x'$ of $X$ and a reversible
    $(n+1)$-cell $f(u) \to v$ of $Y$.
  \end{itemize}
\end{paragraph}

We now move on to the description of generating cofibrations and trivial
cofibrations of the folk model category structure.

\begin{paragraph}
  For every $n \ge 0$, we will denote by $\Dn{n}$ the free-standing
  $n$-cell in $\ooCat$. In other words, the \oo-category $\Dn{n}$
  corepresents the functor sending an \oo-category to its set of $n$-cells.
  This \oo-category $\Dn{n}$ is actually an $n$-category. It has a unique
  non-trivial $n$-cell that we will call its \ndef{principal
  cell}. Here are pictures of $\Dn{n}$ for small~$n$:
  \[
    \Dn{0} = \{ 0 \} \,,\quad \Dn{1} = \xymatrix{0 \ar[r] & 1}
    \,,\quad \Dn{2} = \xymatrix@C=3pc@R=3pc{0 \ar@/^2.5ex/[r]_{}="0"
      \ar@/_2.5ex/[r]_{}="1" \ar@2"0";"1" & 1} \quadand \Dn{3} =
    \xymatrix@C=3pc@R=3pc{0 \ar@/^3ex/[r]_(.47){}="0"^(.53){}="10"
      \ar@/_3ex/[r]_(.47){}="1"^(.53){}="11" \ar@<2ex>@2"0";"1"_{}="2"
      \ar@<-2ex>@2"10";"11"^{}="3" \ar@3"3";"2"_{} & 1 \pbox{.}}
  \]
  If $x$ is an $n$-cell of an \oo-category $X$, we will denote by
  $\funcell{x} : \Dn{n} \to X$ the corresponding \oo-functor.
\end{paragraph}

\begin{paragraph}

  Let $n \ge 0$. We will denote by $\bDn{n}$ the $(n-1)$-category obtained
  from $\Dn{n}$ by removing its principal cell. In other words, $\bDn{0}$ is
  the empty \oo-category (which is a $(-1)$-category!) and, for $n \ge 1$,
  $\bDn{n}$ is the free-standing pair of parallel $(n-1)$-cells in
  $\ooCat$. Here are pictures of $\bDn{n}$ for small $n$:
  \[
    \bDn{0} = \{ \, \}
    \,,\quad
    \bDn{1} = \{\xymatrix@C=0.5pc{0 & 1}\}
    \,,\quad
    \bDn{2} = \xymatrix@C=2.5pc@R=3pc{0 \ar@/^2.5ex/[r]_{}="0"
      \ar@/_2.5ex/[r]_{}="1"
    &  1}
    \quadand
    \bDn{3} = \xymatrix@C=2.5pc@R=3pc{0 \ar@/^3ex/[r]_(.47){}="0"^(.53){}="10"
      \ar@/_3ex/[r]_(.47){}="1"^(.53){}="11"
      \ar@<2ex>@2"0";"1"_{}="2" \ar@<-2ex>@2"10";"11"^{}="3"
    &  1 \pbox{.}}
  \]
  If $n \ge 1$ and $x, y$ are two parallel $(n-1)$-cells of an
  \oo-category $X$, we will denote by $\funcell{x,y} : \bDn{n} \to X$
  the corresponding \oo-functor.

  For every $n \ge 0$, we have a canonical inclusion
  \[ i_n : \bDn{n} \hookto \Dn{n}, \] and, for $n \ge 1$, two
  \oo-functors
  \[ s, t : \Dn{n-1} \to \bDn{n} \] corresponding to the source and
  target of the principal cell of $\Dn{n}$, respectively.
\end{paragraph}

\begin{paragraph}\label{paragr:def_cof}
  We will denote by $I$ the set
  \[ I = \{ i_n : \bDn{n} \hookto \Dn{n} \mid n \ge 0\}. \]
  As the category $\ooCat$ is locally presentable, this set generates a weak
  factorization system on $\ooCat$. The \oo-functors in the left class (that
  is, the retracts of transfinite compositions of pushouts of elements of $I$)
  will be called \ndef{folk cofibrations} or simply \ndef{cofibrations}; as
  for the \oo-functors in the right class (that is, the \oo-functors having
  the right lifting property with respect to $I$), they will be called
  \ndef{folk trivial fibrations} or simply \ndef{trivial fibrations}.
\end{paragraph}

\begin{paragraph}\label{paragr:def_J}
  Let $n \ge 1$. Consider the \oo-functor
  $\funcell{d,d} : \bDn{n} \hookto \Dn{n-1}$, where $d$ denotes the
  principal cell of $\Dn{n-1}$.  Fix a factorization
  \[ \bDn{n} \xto{k_n} \Jn{n} \xto{q_n} \Dn{n-1} \]
  of this \oo-functor into a folk cofibration $k_n$ followed by a folk
  trivial fibration $q_n$. As $i_n$ is a folk cofibration and $q_n$ is a
  folk trivial fibration, the commutative square
  \[
    \xymatrix{
      \bDn{n} \ar[d]_{i_n} \ar[r]^{k_n} & J_{n}  \ar[d]^{q_n} \\
      \Dn{n} \ar[r]_{\funcell{\id{d}}} \ar@{.>}[ur] & \Dn{n-1} \pbox{,}
    }
  \]
  where $d$ still denotes the principal cell of $\Dn{n-1}$, admits a lift. We fix
  such a lift
  \[ l_n : \Dn{n} \to \Jn{n}. \]
  By definition, the \ndef{principal cell of $J_n$} is the
  image of the principal cell of~$\Dn{n}$ by $l_n$.

  We will denote by
  \[ j_n : \Dn{n-1} \to \Jn{n} \]
  the composite
  \[ \Dn{n-1} \xto{s} \bDn{n} \xto{k_n} \Jn{n}, \]
  that is, the \oo-functor corresponding to the source of the
  principal cell of $\Jn{n}$, and by~$J$ the set
  \[ J = \{j_n : \Dn{n-1} \to \Jn{n} \mid n \ge 1\}. \]
\end{paragraph}

\begin{theorem}[Lafont--Métayer--Worytkiewicz]
  The category $\ooCat$ is endowed with a model category structure,
  cofibrantly generated by $I$ and $J$, whose weak equivalences are the folk
  weak equivalences and whose cofibrations are the folk cofibrations. All
  the \oo-categories are fibrant for this model category structure.
\end{theorem}

\begin{proof}
  This is \cite[Theorem 4.39 and Proposition 5.1]{Folk}.
\end{proof}

The model category structure of the previous theorem is known as the
\ndef{folk model category structure} on $\ooCat$.
We will now describe a
path object for this structure. We start by some preliminaries on oplax
transformations.

\begin{paragraph}
  If $X$ and $Y$ are two \oo-categories, we will denote by $X \otimes Y$
  their Gray tensor product. We refer the reader to
  Section~\ref{sec:Gray} for more details and a precise definition of
  this tensor product, based on Steiner's work \cite{Steiner}.
  Let us only recall that the Gray tensor product defines a
  (non-symmetric) biclosed monoidal category structure whose unit is the
  terminal \oo-category $\Dn{0}$. Its right and left internal $\Hom$ will be
  denoted by $\HomOpLax$ and $\HomLax$, respectively, so that if $X$, $Y$
  and $Z$ are three \oo-categories, we have natural bijections
  \[
    \Hom_{\ooCat}(X, \HomOpLax(Y, Z)) \simeq \Hom_{\ooCat}(X \otimes
    Y, Z) \simeq \Hom_{\ooCat}(Y, \HomLax(X, Z)).
  \]
\end{paragraph}

\begin{paragraph}
  Let $X$ and $Y$ be two \oo-categories. By adjunction, the set of $0$-cells
  of $\HomOpLax(X, Y)$ can be identified with the set of \oo-functors
  $\Hom_{\ooCat}(X, Y)$. If $f, g : X \to Y$ are two \oo-functors, an
  \ndef{oplax transformation} $\alpha : f \tod g$ is a $1$-cell of
  $\HomOpLax(X, Y)$ from $f$ to
  $g$. Such an oplax transformation can be identified with a functor
  \[ h : \Dn{1} \otimes X \to Y \] making the diagram
  \[
    \xymatrix@R=1.5pc{
      X \ar@/^2.5ex/[drr]^f \ar[dr]_-{\funcell{0} \otimes X} \\
      & \Dn{1} \otimes X \ar[r]^h & Y \\
      X \ar@/_2.5ex/[urr]_g \ar[ur]^-{\funcell{1} \otimes X} & &
      \pbox{,} }
  \]
  where $X$ is identified with $\Dn{0} \otimes X$,
  commute. Alternatively, again by adjunction, such an oplax
  transformation can be seen as an \oo-functor
  \[ k : X \to \Cyl(Y), \] where $\Cyl(Y) = \HomLax(\Dn{1}, Y)$,
  making the diagram
  \[
    \xymatrix@R=1.5pc{
      & & Y \\
      X \ar@/^2.5ex/[rru]^f \ar@/_2.5ex/[rrd]_g \ar[r]^-{k} & \Cyl(Y)
      \ar[ur]_{\pi^-}  \ar[dr]^{\pi^+} \\
      & & Y \pbox{,} }
  \]
  where $Y$ is identified with $\HomLax(\Dn{0}, Y)$ and
  \[
    \pi^- = \HomLax(\funcell{0}, Y) \quadand \pi^+ =
    \HomLax(\funcell{1}, Y),
  \]
  commute.

  One can define lax transformations in a similar way.
\end{paragraph}

\begin{paragraph}\label{paragr:tens_cell}
  Let $x$ be an $m$-cell of an \oo-category $X$ and let $y$ be an
  $n$-cell of an \oo-category~$Y$. One defines an $(m+n)$-cell
  $x \otimes y$ of $X \otimes Y$ in the following way. The \oo-category
  $\Dn{m} \otimes \Dn{n}$ is an $(m+n)$-category that admits a unique
  non-trivial $(m+n)$\nbd-cell.
  We will call this cell the \ndef{principal cell} of $\Dn{m} \otimes
  \Dn{n}$. The $(m+n)$-cell $x \otimes y$ is the cell corresponding to the
  \oo-functor
  \[
    \Dn{m+n} \xto{\funcell{p}} \Dn{m} \otimes \Dn{n} \xto{\funcell{x}
      \otimes \funcell{y}} X \otimes Y,
  \]
  where $p$ denotes the principal cell of $\Dn{m} \otimes \Dn{n}$.
\end{paragraph}

\begin{paragraph}\label{paragr:def_components}
  Let $f, g : X \to Y$ be two \oo-functors and let $\alpha : f \tod g$
  be an oplax transformation. Denote by $h : \Dn{1} \otimes X \to Y$ the
  corresponding \oo-functor and by $(01)$ the principal cell of
  $\Dn{1}$. If $x$ is an $n$-cell of $X$, the \ndef{component} of
  $\alpha$ at $x$ is the $(n+1)$-cell of $Y$
  \[ \alpha_x = h((01) \otimes x). \]
  As the \oo-category $\Dn{1} \otimes X$ is generated by cells of the form
  $0 \otimes x$, $1 \otimes x$ and $(01) \otimes x$, with $x$ a cell of $X$,
  the transformation $\alpha$ is entirely determined by its components.
  Furthermore, oplax transformations can be defined purely in terms of their
  components (see \cite[paragraph 1.9 and Section B.2]{AraMaltsiJoint}).
\end{paragraph}

\begin{paragraph}\label{paragr:desc_cylinder}
  If $Y$ is an \oo-category, the $n$-cells of $\Cyl(Y)$ are called
  \ndef{$n$-cylinders} in $Y$. By adjunction, they correspond to \oo-functors
  $c : \Dn{1} \otimes \Dn{n} \to Y$. If $c$ is such an $n$-cylinder, we can set
  \[ x = c(0 \otimes d) \quadand y = c(1 \otimes d)
  \]
  \goodbreak\noindent
  and, for $0 \le k \le n$,
  \[
    \alpha_k^- = c((01) \otimes s_k(d)) \quadand \alpha_k^+ = c((01)
    \otimes t_k(d)),
  \]
  where $(01)$ denotes the principal cell of $\Dn{1}$ and $d$ the one
  of $\Dn{n}$. Note that $\alpha_n^- = \alpha_n^+$ and we will often write
  $\alpha_n$ for this cell. These cells completely determine $c$ and we will
  often write $c = (x, y, \alpha)$. Moreover, by
  \cite[Proposition~B.1.6]{AraMaltsiJoint}, $n$-cells $x$ and $y$ and
  $(k+1)$-cells $\alpha_k^-, \alpha_k^+$, for $0 \le k \le n$, with
  $\alpha_n^- = \alpha_n^+$, determine an $n$-cylinder if and only if one has
  \[
    \alpha_k^\epsilon: \alpha_{k-1}^+ \comp_{k-1} \alpha_{k-2}^+ \comp_{k-2}
    \cdots \comp_1 \alpha_0^+ \comp_0 x_k^\epsilon \to y_k^\epsilon \comp_0 \alpha_0^-
    \comp_1 \cdots \comp_{k-1} \alpha_{k-1}^-,
  \]
  for $\epsilon = \pm$, where $x_k^- = s_k(x)$ and $x_k^+ = t_k(x)$,  and
  similarly for $y$.

  If  $c = (x, y, \alpha)$ is an $n$-cylinder, the cell $\alpha_n^- = \alpha_n^+$ is called the \ndef{principal cell} of $c$. We say that $c$ is \ndef{reversible} if
  all the cells $\alpha_k^\epsilon$ for $0 \le k \le n$ and $\epsilon = \pm$
  are reversible. It follows from the explicit formulas describing
  the operations of the \oo-category $\Cyl(Y)$ (see \cite[Appendix A]{Folk} or
  \cite[Proposition B.1.15]{AraMaltsiJoint}) that the graded subset
  $\CylRev(Y)$ of~$\Cyl(Y)$ consisting of reversible cylinders is actually a
  sub-\oo-category.
\end{paragraph}

\begin{paragraph}
  Let $f, g : X \to Y$ be two \oo-functors and let $\alpha : f \tod g$
  be an oplax transformation. The transformation $\alpha$ is said to
  be \ndef{reversible} if, for every cell $x$ of $X$, the component
  $\alpha_x$ is a reversible cell of $Y$. A reversible oplax
  transformation will be simply called a \ndef{reversible
    transformation}.

  Essentially by definition, the transformation $\alpha$ is reversible
  if and only if the corresponding \oo-functor $X \to \Cyl(Y)$
  factors through the inclusion $\CylRev(Y) \hookto \Cyl(Y)$. In other words,
  the data of a reversible transformation $\alpha : f \tod g$
  corresponds to the data of an \oo-functor
  \[ k : X \to \CylRev(Y) \]
  making the obvious diagram
  \[
    \xymatrix@R=1.5pc{
      & & Y \\
      X \ar@/^2.5ex/[rru]^f \ar@/_2.5ex/[rrd]_g \ar[r]^-{k} &
      \CylRev(Y)
      \ar[ur]_{\pi^-}  \ar[dr]^{\pi^+} \\
      & & Y
    }
  \]
  commute.
\end{paragraph}

\begin{paragraph}
  Let $f : X \to Y$ be an \oo-functor. The identity on $f$ seen as a
  $0$-cell of $\HomOpLax(X, Y)$ defines an oplax transformation
  $\id{f} : f \tod f$.  This transformation is easily seen to be
  reversible (its components are identities).

  In particular, by applying this to the identity \oo-functor
  $\id{X} : X \to X$, we get a commutative diagram
  \[
    \xymatrix@R=1.5pc{
      & & X \\
      X \ar@/^2.5ex/[rru]^{\id{X}} \ar@/_2.5ex/[rrd]_{\id{X}}
      \ar[r]^-{\iota} & \CylRev(X)
      \ar[ur]_{\pi^-}  \ar[dr]^{\pi^+} \\
      & & X }
  \]
  or, in other words, a factorization
  \[ X \xto{\iota} \CylRev(X) \xto{\pi} X \times X \] of the diagonal
  functor.
\end{paragraph}

\begin{theorem}[Lafont--Métayer--Worytkiewicz]\label{thm:path_obj}
  For every \oo-category $X$, the factorization
  \[ X \xto{\iota} \CylRev(X) \xto{\pi} X \times X \] of the diagonal
  is a path object for the folk model category structure, in the sense
  that $\iota$ a weak equivalence and that $\pi$ is a fibration.
\end{theorem}

\begin{proof}
  This is \cite[Proposition 4.45]{Folk}.
\end{proof}

\begin{remark}\label{rem:trans_right_homot}
  A right homotopy with respect to the path object of the previous
  theorem is precisely a reversible transformation.
\end{remark}

We will now describe the cofibrant objects of the folk model category
structure.

\begin{paragraph}\label{paragr:def_freely_gen}

  Let $X$ be an \oo-category. For $m \ge
  -1$, we will denote by $X_{\le m}$ the $m$\nbd-category obtained from $X$ by
  removing the non-trivial $k$-cells for $k > m$. In particular, if $m =
  -1$, we have $X_{\le -1} = \emptyset$. There is an obvious inclusion
  \oo-functor $X_{\le m} \hookto X_{\le m+1}$.

  Let $B$ be a set of cells of $X$. We will say that $X$ is \ndef{freely
  generated by $B$} if, for every $n \ge 0$, the commutative square
  \[
    \xymatrix@C=3pc@R=3pc{
      \coprod_{x \in B_n} \bDn{n} \ar[d]_{\coprod_{x \in B_n} i_n}
      \ar[r]^-{\funcell{sx, tx}}
      & X_{\le n-1} \ar[d] \\
      \coprod_{x \in B_n} \Dn{n} \ar[r]_-{\funcell{x}} & X_{\le n}
      \pbox{,} }
  \]
  where $B_n$ denotes the set of $n$-cells in $B$ and the right
  vertical arrow is the canonical inclusion, is a pushout square.

  One says that an \oo-category is \ndef{free in the sense of
  polygraphs} if it admits a set of cells that freely generates it.
\end{paragraph}

\begin{theorem}[Métayer]
  The cofibrant objects of the folk model category structure are the
  \oo-categories that are free in the sense of polygraphs.
\end{theorem}

\begin{proof}
  This is the main result of \cite{MetCof}.
\end{proof}

We end the section by introducing important dualities of $\ooCat$ and some
of their properties.

\begin{paragraph}\label{paragr:dualities}
  If $X$ is an \oo-category, we will denote by $X^\op$ \resp{by
  $X^\co$} the \oo-category obtained from $X$ by reversing the
  direction of the cells of odd \resp{even} dimension. The assignments
  $X \mapsto X^\op$ and $X \mapsto X^\co$ are both involutive
  automorphisms of the category~$\ooCat$. Moreover, they are
  anti-monoidal in the sense that the assignment $x \otimes y \mapsto y
  \otimes x$ defines isomorphisms
  \[
    (X \otimes Y)^\op \simeq Y^\op \otimes X^\op \quadand (X \otimes
    Y)^\co \simeq Y^\co \otimes X^\co.
  \]
  Furthermore, there are canonical isomorphisms
  \[
    \begin{split}
      \HomOpLax(X,Y)^\op & \simeq \HomLax(X^\op, Y^\op), \\
      \HomOpLax(X,Y)^\co & \simeq \HomLax(X^\co, Y^\co)
    \end{split}
  \]
  (see for instance \cite[Propositions A.22 and A.23]{AraMaltsiJoint}).

  The symmetry of the definition of a reversible cell shows that a
  cell is reversible in~$X$ if and only if the corresponding cell is
  reversible in~$X^\op$ \resp{in $X^\co$}. This easily implies that an
  \oo-functor $f : X \to Y$ is a folk weak equivalence if and
  only if $f^\op : X^\op \to Y^\op$ \resp{$f^\co : X^\co \to Y^\co$}
  is. Moreover, for every $n \ge 0$, the \oo-functor $i_n^\op$
  \resp{$i_n^\co$} can be identified with the \oo-functor
  $i_n : \bDn{n} \hookto \Dn{n}$. This implies that an \oo-functor $i$ is a
  folk cofibration if and only if $i^\op$ \resp{$i^\co$} is, and hence that
  $j$ is a folk trivial cofibration if and only if $j^\op$ \resp{$j^\co$}
  is.
\end{paragraph}

\section{Preliminaries on the Gray tensor product}
\label{sec:Gray}

The purpose of this section is to define the Gray tensor product of
\oo-categories. This tensor product was introduced by Al-Agl and Steiner
\cite{AlAglSteiner} as a generalization of Gray's tensor product of
$2$-categories \cite{GrayFCT}, and is somehow a lax version of the
cartesian product. The definition we will give in this section is based on
Steiner's theory of augmented directed complexes \cite{Steiner}. The
strategy, due to Steiner, is the following. Steiner's complexes are a tool
to describe a large subclass of the class of free \oo-categories in the
sense of polygraphs. The usual tensor product of chain complexes induces a
tensor product on these free \oo-categories. The general Gray tensor product
is then obtained by density of this subclass in the category of
\oo-categories.

\medskip

We start by briefly recalling Steiner's theory.

\begin{paragraph}
  An \ndef{augmented directed complex} is an augmented chain complex
  of abelian groups in nonnegative degree
  \[
    \cdots \xto{d} K_n \xto{d} K_{n-1} \xto{d} \cdots \xto{d} K_0
    \xto{e} \Z,
  \]
  endowed with, for every $n \ge 0$, a submonoid $K_n^\ast$ of $K_n^{}$
  of so-called \ndef{positive elements}. If $K$ and $L$ and two
  augmented directed complexes, a \ndef{morphism} $f : K \to L$ is a
  morphism of the underlying augmented chain complexes respecting the
  positive elements, that is, such that, for every $n \ge 0$, we have
  $f(K_n^\ast) \subset L_n^\ast$. We will denote by $\ADC$ the
  category of augmented directed complexes.
\end{paragraph}

\begin{paragraph}
  In \cite{Steiner}, Steiner defines a functor
  \[ \nu : \ADC \to \ooCat. \]
  We refer the reader to \cite[Definition 1.6]{Steiner} (or \cite[paragraph
  2.4]{AraMaltsiJoint}) for a detailed definition. Let us just mention that
  if $K$ is an augmented directed complex, then the $n$-cells of the
  \oo-category $\nu(K)$ are given by tables
  \[
    \begin{pmatrix}
      x_0^- & \cdots & x_n^- \\
      \noalign{\vskip 3pt}
      x_0^+ & \cdots & x_n^+ \\
    \end{pmatrix},
  \]
  where
  \begin{itemize}
  \item $x_i^-$ and $x_i^+$ are in $K_i^\ast$, for $0 \le i \le n$,
  \item $x^-_n = x^+_n$,
  \item $d(x^-_i) = x^+_{i-1} - x^-_{i-1} = d(x^+_i)$, for
    $0 < i \le n$,
  \item $e(x_0^-) = 1$ and $e(x_0^+) = 1$.
  \end{itemize}

\end{paragraph}

\begin{paragraph}
  If $K$ is an augmented directed complex, a \ndef{basis} of $K$ is a graded
  set $(B_n)_{n \ge 0}$ such that, for every $n \ge 0$,
  \begin{itemize}
  \item $B_n$ is a basis of the $\Z$-module $K_n$,
  \item $B_n$ generates the submonoid $K_n^\ast$.
  \end{itemize}
  One shows that if such a basis exists, then it is unique.
\end{paragraph}

\begin{paragraph}
  If $K$ is an augmented directed complex with basis $(B_n)$, then for
  every $n$\nbd-chain~$x$, one can write $x = \sum_{b \in B_n} n_b b$,
  where the $n_b$ are integers, in a unique way, and we set
  \[
    x^- = \sum_{\substack{b \in B_n\\ n_b < 0}} (-n_b) b \quadand x^+ =
    \sum_{\substack{b \in B_n\\ n_b > 0}} n_b b.
  \]
  If $x$ is an $n$-chain with $n > 0$, we set
  \[ d^-x = (dx)^- \quadand d^+x = (dx)^+. \]
\end{paragraph}

\begin{paragraph}
  Let $K$ be an augmented directed complex with basis. For every
  $n$-chain $x$ in the basis, we define a table
  \[
    \atom{x} =
    \begin{pmatrix}
      x_0^- & \cdots & x_n^- \\
      \noalign{\vskip 3pt}
      x_0^+ & \cdots & x_n^+ \\
    \end{pmatrix},
  \]
  by induction, setting
  \begin{itemize}
  \item $x^-_n = x$ and $x^+_n = x$,
  \item $x^-_i = d^-(x^-_{i+1})$ and $x^+_i = d^+(x^+_{i+1})$, for
    $0 \le i < n$.
  \end{itemize}
  This table is an $n$-cell of $\nu(K)$ if and only if $e(x^-_0) = 1$
  and $e(x^+_0) = 1$. In this case, one says that the $n$-cell
  $\atom{x}$ is the \ndef{atom} associated to $x$.

  The augmented directed complex with basis $K$ is said to be
  \ndef{unital} if, for every element $x$ of the basis of $K$, one has
  $e(x^-_0) = 1$ and $e(x^+_0) = 1$.
\end{paragraph}

\begin{paragraph}
  One says that an augmented directed complex $K$ with basis $(B_n)$
  is \ndef{strongly loop-free} if there exists a partial order
  $\preceq$ on $\coprod_{n \ge 0} B_n$ such that, for every $n > 0$, every
  $x$ in $B_n$, and every $y$ and $z$ in the support (according to the
  basis~$B_{n-1}$) of $d^-x$ and $d^+x$, respectively, one has
  \[ y \preceq x \preceq z. \]
\end{paragraph}

\begin{paragraph}
  A \ndef{strong Steiner complex} is an augmented directed complex
  with basis that is both unital and strongly loop-free. We will
  denote by $\Stf$ the full subcategory of~$\ADC$ consisting of
  Steiner complexes.
\end{paragraph}

\begin{theorem}[Steiner]\label{thm:Steiner}
  The functor $\restrict{\nu}{\Stf} : \Stf \to \ooCat$ is fully
  faithful. Moreover, if $K$ is a strong Steiner complex, then
  $\nu(K)$ is freely generated (see paragraph~\ref{paragr:def_freely_gen})
  by its atoms.
\end{theorem}

\begin{proof}
  This follows from \cite[Proposition 3.7, Theorem 5.6 and Theorem
  6.1]{Steiner}.
\end{proof}

We will now define the Gray tensor product of \oo-categories, starting with
the tensor product of augmented directed complexes.

\begin{paragraph}\label{paragr:def_tens_ADC}
  The \ndef{tensor product} $K \otimes L$ of two augmented directed
  complexes $K$ and $L$ is defined in the following way:
  \begin{itemize}
  \item The underlying augmented complex of $K \otimes L$ is the
    usual one:
    \begin{itemize}
    \item for $n \ge 0$, we have
      \[ (K \otimes L)_n = \bigoplus_{i + j = n} K_i \otimes L_j, \]
    \item for $x$ in $K_i$ and $y$ in $K_j$, we have
      \[ d(x \otimes y) = dx \otimes y + (-1)^i x \otimes dy, \] where
      by convention $dz = 0$ if the degree of $z$ is $0$,
    \item for $x$ in $K_0$ and $y$ in $L_0$, we have
      \[ e(x \otimes y) = e(x)e(y). \]
    \end{itemize}
  \item The submonoid $(K \otimes L)^\ast_n$ is defined to be
    generated by the subset
    \[ \bigoplus_{i + j = n} K^\ast_i \otimes L^\ast_j \] of
    $(K \otimes L)_n$.
  \end{itemize}
  The tensor product defines a (non-symmetric) monoidal category structure
  on the category of augmented directed complexes. Its unit, that we will
  denote by $\Z$, is the complex concentrated in degree $0$ of value $\Z$
  with the identity augmentation and $\N$ as the submonoid of positive
  elements of degree $0$. Steiner proved (see \cite[Example~3.10]{Steiner})
  that this monoidal category structure restricts to the full subcategory of
  strong Steiner complexes.
\end{paragraph}

\begin{theorem}[Steiner, Ara--Maltsiniotis]\label{thm:def_tens}
  There exists a unique, up to unique isomorphism, biclosed monoidal
  category structure on $\ooCat$ making the functor
  $\nu_{|\Stf} : \Stf \to \ooCat$ a monoidal functor, where $\Stf$ is
  endowed with the monoidal category structure given by the tensor
  product.
\end{theorem}

\begin{proof}
  See \cite[Section 7]{Steiner}, whose proof was completed by
  \cite[Theorem A.15]{AraMaltsiJoint}.
\end{proof}

\begin{paragraph}\label{paragr:def_tens}
  We define the \ndef{Gray tensor product} to be the tensor product given by
  the previous theorem. If $X$ are $Y$ are two \oo-categories, their Gray
  tensor product will be denoted by $X \otimes Y$. Explicitly, one has
  \[
    X \otimes Y =
    \limind_{\substack{\nu(K) \to X, \, K \in \Stf\\
      \nu(L) \to Y, \, L \in \Stf}}
      \nu(K \otimes L).
  \]
  The unit of the Gray tensor product is the terminal \oo-category
  $\Dn{0}$.

  The right and left internal $\Hom$ of the Gray tensor product will be
  denoted by $\HomOpLax$ and $\HomLax$, respectively, so that if
  $X$, $Y$ and $Z$ are three \oo-categories, we have natural bijections
  \[
    \Hom_{\ooCat}(X, \HomOpLax(Y, Z))
    \simeq
    \Hom_{\ooCat}(X \otimes Y, Z)
    \simeq
    \Hom_{\ooCat}(Y, \HomLax(X, Z)).
  \]
\end{paragraph}

\begin{examples}
  Here are some examples of Gray tensor products of \oo-categories:
  \[
    \Dn{1} \otimes \Dn{1} =
    \raisebox{1.75pc}{
    \xymatrix{
      \bullet \ar[r] \ar[d] &
      \bullet \ar[d] \\
      \bullet \ar[r] & \bullet
      \ar@{}[u];[l]_(.30){}="s"
      \ar@{}[u];[l]_(.70){}="t"
      \ar@2"s";"t"
      \pbox{,}
    }}
   \qquad
   \qquad
    \Dn{1} \otimes \Deltan{2} =
    \raisebox{1.75pc}{
     \xymatrix{
       \bullet
       \ar[r]^{}
       \ar[d]_{}
       &
       \bullet
       \ar[r]^{}
       \ar[d]_{}
       \ar@{}[ld]|-{}
       &
       \bullet
       \ar[d]^{}
       \\
       \bullet
       \ar[r]_{}
       &
       \bullet
       \ar[r]_{}
       &
       \bullet \pbox{,}
       \ar@{}[lu];[ll]_(0.30){}="s1"^(0.70){}="t1"
       \ar@2"s1";"t1"_{}
       \ar@{}[u];[l]_(0.30){}="s2"^(0.70){}="t2"
       \ar@2"s2";"t2"_{}
     }}
   \]
   \[
    \Dn{1} \otimes \Dn{2} =
    \raisebox{2.25pc}{
    \xymatrix@C=3pc@R=3pc{
      \bullet
      \ar@/^2ex/[r]^(0.70){}_{}="0"
      \ar@/_2ex/[r]_(0.70){}_{}="1"
      \ar[d]_{}="f"
      \ar@2"0";"1"
      &
      \bullet
      \ar[d] \\
      \bullet
      \ar@{.>}@/^2ex/[r]^(0.30){}_{}="0"
      \ar@/_2ex/[r]_(0.30){}_{}="1"
      \ar@{:>}"0";"1"_{}
      &
      \bullet
      \ar@{}[u];[l]_(.40){}="x"
      \ar@{}[u];[l]_(.60){}="y"
      \ar@<-1.5ex>@/_1ex/@{:>}"x";"y"_(0.60){}_{}="0"
      \ar@<1.5ex>@/^1ex/@2"x";"y"^(0.40){}_{}="1"
      \ar@{}"1";"0"_(.05){}="z"
      \ar@{}"1";"0"_(.95){}="t"
      \ar@3{>}"z";"t"_{}
      \pbox{,}
    }}
   \]
   where $\Deltan{2} = \xymatrix{\bullet \ar[r] & \bullet \ar[r] & \bullet}$.
\end{examples}

\begin{remark}
  If $x$ is an $m$-cell of an \oo-category $X$ and $y$ is an $n$-cell of an
  \oo-category~$Y$, we saw in paragraph~\ref{paragr:tens_cell} that one can
  define an $(m+n)$-cell $x \otimes y$ of~$X \otimes Y$. For instance, the tensor product
  of the principal cells of the disks appearing in the examples above is, in
  both cases, the unique non-trivial cell of maximal dimension. The formula that
  we gave as a definition for the Gray tensor product easily implies that the
  \oo-category $X \otimes Y$ is generated by the set of cells of the
  form~$x \otimes y$, with $x$ a cell of $X$ and $y$ a cell of $Y$.
\end{remark}

\begin{remark}
  The Gray tensor product used in this paper is what we like to call the
  \emph{oplax} Gray tensor product. The \ndef{lax} version is the functor
  $(X, Y) \mapsto Y \otimes X$ and is actually the one introduced by Gray in
  the $2$-categorical case \cite{GrayFCT}. The natural isomorphism
  $(X \otimes Y)^\op \simeq Y^\op \otimes X^\op$ and the stability of the
  data of the folk model category structure by the duality $Z \mapsto Z^\op$
  (see paragraph~\ref{paragr:dualities}) show that the results we prove in
  this paper for the oplax version of the Gray tensor product can be adapted
  to the lax version.
\end{remark}

\section{Compatibility of the tensor product with cofibrations}
\label{sec:Gray_cof}

The purpose of this section is to prove that $\ooCat$ endowed with the Gray
tensor product $\otimes$ satisfies the part of the axioms of monoidal model
categories (see paragraph~\ref{paragr:def_monoidal_model}) dealing with
cofibrations. In other words, given two folk cofibrations
\[ i : X \to Y \quadand j : Z \to T, \]
we will prove that the \oo-functor
\[
  i \boxprod j : Y \otimes Z \amalg_{X \otimes Z} X \otimes T \to Y
  \otimes T,
\]
is also a folk cofibration. This immediately follow from the case of
generating cofibrations, for which we will use Steiner's theory.

\medbreak

We start by some supplements on pushouts of strong Steiner complexes.

\begin{paragraph}\label{paragr:conv_B_K}
  If $K$ is an augmented directed complex with basis, we will denote
  its basis by~$B_K$.

  Let $f : K \to L$ be a monomorphism of augmented directed complexes with
  basis. One says that $f$ is a \ndef{rigid monomorphism} if it sends
  elements of the basis $B_K$ of $K$ to elements of the basis
  $B_L$ of $L$.
\end{paragraph}

\begin{proposition}[Ara--Maltsiniotis]\label{prop:pushout_ADC}
  Consider a pushout square
  \[
    \xymatrix{
      K \ar[d]_f \ar[r]^u & M \ar[d]^g \\
      L \ar[r]_v & N }
  \]
  in the category of augmented directed complexes such that:
  \begin{itemize}
  \item $K$, $L$, $M$, $N$ are strong Steiner complexes,
  \item $f$ and $u$ are rigid monomorphisms.
  \end{itemize}
  Then
  \begin{itemize}
  \item we have $B_N = B_L \amalg_{B_K} B_M$ (as sets),
  \item the morphisms $g$ and $v$ are rigid monomorphisms,
  \item the functor $\nu : \ADC \to \ooCat$ sends this square to a
    pushout square in $\ooCat$.
  \end{itemize}
\end{proposition}

\begin{proof}
  The first assertion is a particular case of \cite[Proposition
  3.6]{AraMaltsiJoint}. The second one follows from \cite[Proposition
  3.12]{AraMaltsiJoint}. As for the third one, it is a special case of
  \cite[Theorem 3.8]{AraMaltsiJoint}.
\end{proof}

\begin{remark}
  The proposition remains true if one only assumes that the complexes
  are Steiner complexes (named ``augmented directed complexes with a loop-free
  unital basis'' in \cite{AlAglSteiner}) as opposed to strong Steiner
  complexes. We stated the more restrictive result only because we did not
  include the definition of a Steiner complex in this paper.
\end{remark}

\begin{paragraph}
  If $K$ and $L$ are two augmented directed complexes with basis, one
  immediately checks that $K \otimes L$ is an augmented directed
  complex with basis
  \[
    B_{K \otimes L} = B_K \otimes B_L = \{ x \otimes y \mid x \in B_K,
    \, y \in B_L\}.
  \]
\end{paragraph}

\begin{proposition}\label{prop:boxprod_basis}
  Let $i : K \to L$ and $j : M \to N$ be two rigid monomorphisms
  between augmented directed complexes with basis. Then the morphism
  \[ i \boxprod j : L \otimes M \amalg_{K \otimes M} K \otimes N \to 
     L \otimes N \]
  is a rigid monomorphism between augmented directed
  complexes with basis which identifies
  $L \otimes M \amalg_{K \otimes M} K \otimes N$ with the subcomplex
  generated by $B_L \otimes B_M \cup B_K \otimes B_N$.
\end{proposition}

\begin{proof}
  Colimits in the category of augmented directed complexes are
  computed degreewise (see \cite[paragraph 3.1]{AraMaltsiJoint}). Let
  $n \ge 0$. If $B$ is the basis of an augmented directed complex, we
  will denote by $B_n$ the set of $n$-chains in $B$.  As the free
  abelian group functor commutes with colimits, the abelian group
  $(L \otimes M \amalg_{K \otimes M} K \otimes N)_n$ is free with
  basis
  $(B_L \otimes B_M)_n \amalg_{(B_K \otimes B_M)_n} (B_K \otimes
  B_N)_n = (B_L \otimes B_M)_n \cup (B_K \otimes B_N)_n$.  Similarly,
  the submonoid
  $(L \otimes M \amalg_{K \otimes M} K \otimes N)^\ast_n$ is generated
  by this basis.  This proves that
  $L \otimes M \amalg_{K \otimes M} K \otimes N$ is free with basis
  $B_L \otimes B_M \cup B_K \otimes B_N$.  Moreover, this shows that
  the map
  $(i \boxprod j)_n : (L \otimes M \amalg_{K \otimes M} K \otimes N)_n
  \to (L \otimes N)_n$ can be identified with the image of the map
  $(B_L \otimes B_M)_n \cup (B_K \otimes B_N)_n \hookto (B_L \otimes B_N)_n$
  by the free abelian group functor. As this functor preserves
  monomorphisms, this implies that $i \boxprod j$ is a monomorphism. The
  fact that it is rigid being obvious, this ends the proof.
\end{proof}

\begin{proposition}
  Let $i : K \to L$ and $j : M \to N$ be two rigid monomorphisms
  between strong Steiner complexes.  Then the pushout square
  associated to
  \[
    \xymatrix@C=2.5pc{ L \otimes M & K \otimes M \ar[l]_-{i \otimes M}
      \ar[r]^-{K \otimes j} & K \otimes N }
  \]
  satisfies the hypotheses of Proposition~\ref{prop:pushout_ADC}.
\end{proposition}

\begin{proof}
  Strong Steiner complexes and rigid monomorphisms are both stable
  under tensor product by \cite[Example 3.10]{Steiner} and
  \cite[Proposition A.6]{AraMaltsiJoint}. It thus suffices to prove
  that $L \otimes M \amalg_{K \otimes M} K \otimes N$ is a strong
  Steiner complex. This follows immediately from the fact that, by the
  previous proposition, $L \otimes M \amalg_{K \otimes M} K \otimes N$
  is a subcomplex of the strong Steiner complex $L \otimes N$
  generated by a subset of its basis.
\end{proof}

\begin{proposition}\label{prop:boxprod_rigid_mono}
  Let $i : K \to L$ and $j : M \to N$ be two rigid monomorphisms
  between strong Steiner complexes.  Then the \oo-functor
  \[
    \nu(i) \boxprod \nu(j) : \nu(L) \otimes \nu(M) \amalg_{\nu(K)
      \otimes \nu(M)} \nu(K) \otimes \nu(N) \to \nu(L) \otimes \nu(N)
  \]
  is a folk cofibration.
\end{proposition}

\begin{proof}
  By applying Proposition~\ref{prop:pushout_ADC} to the pushout square
  of the previous proposition and using the fact that the functor
  $\nu_{|\Stf} : \Stf \to \ooCat$ is monoidal for the tensor product
  (Theorem~\ref{thm:def_tens}), one gets that the \oo-functor
  \[
    \nu(i) \boxprod \nu(j) : \nu(L) \otimes \nu(M) \amalg_{\nu(K)
      \otimes \nu(M)} \nu(K) \otimes \nu(N) \to \nu(L) \otimes \nu(N)
  \]
  can be identified with
  \[ \nu(i \boxprod j) : \nu(L \otimes M \amalg_{K \otimes M} K
    \otimes N) \to \nu(L \otimes N). \]
  As the functor $\nu$ respects
  monomorphisms (this follows from its concrete description but also
  from the fact that it admits a left adjoint, see \cite[Theorem
  2.11]{Steiner}), $\nu(L \otimes M \amalg_{K \otimes M} K \otimes N)$
  can be identified with a sub-\oo-category of $\nu(L \otimes
  M)$. Moreover, by Steiner's Theorem~\ref{thm:Steiner}, these two
  \oo-categories are freely generated by their atoms, which, by
  Proposition~\ref{prop:boxprod_basis}, are in bijection with
  $B_L \otimes B_M \cup B_K \otimes B_N$ and $B_L \otimes B_N$,
  respectively.
  The \oo-category $\nu(L \otimes N)$ can thus be obtained from
  $\nu(L \otimes M \amalg_{K \otimes M} K \otimes N)$ by freely adding
  cells (in the sense of taking pushouts along some
  $i_n : \bDn{n} \hookto \Dn{n}$) indexed by
  $(B_L \otimes B_N) \backslash (B_L \otimes B_M \cup B_K \otimes
  B_N)$.  The inclusion morphism is therefore a cofibration.
\end{proof}

To apply the previous proposition to the generating cofibrations, we
need the following lemma:

\begin{lemma}\label{lemma:gen_cof_rig}
  Let $n \ge 0$. The \oo-functor $i_n : \bDn{n} \hookto \Dn{n}$
  can be written $\nu(\lambda(i_n))$, where $\lambda(i_n) : \lambda(\bDn{n})
  \hookto \lambda(\Dn{n})$ is a rigid monomorphism between strong Steiner
  complexes.
\end{lemma}

\begin{proof}
  Let $\lambda(\Dn{n})$ be the free augmented directed complex with
  basis the set of non-trivial cells of $\Dn{n}$ (see \cite[paragraph
  4.10]{AraMaltsiJoint} for an explicit description), let
  $\lambda(\bDn{n})$ be the augmented directed subcomplex generated by
  the subset of non-trivial cells of $\bDn{n}$ and let $\lambda(i_n)$
  be the resulting inclusion morphism. One easily checks that $i_n$
  can be identified with $\nu(\lambda(i_n))$, that $\lambda(\Dn{n})$
  is indeed a strong Steiner complex (see \cite[Example~4.7]{Steiner}
  or~\cite[paragraph~4.10]{AraMaltsiJoint}) and thus that any subset
  of its basis defines a strong Steiner subcomplex whose inclusion
  morphism is a rigid monomorphism.
\end{proof}

\begin{theorem}\label{thm:tens_comp_cof}
  If
  \[ i : X \to Y \quadand j : Z \to T \] are two folk cofibrations,
  then the \oo-functor
  \[
    i \boxprod j : Y \otimes Z \amalg_{X \otimes Z} X \otimes T \to Y
    \otimes T
  \]
  is also a folk cofibration.
\end{theorem}

\begin{proof}
  By the classical Lemma~\ref{lemma:monoidal_model_gen}, it suffices
  to prove the result when $i$ and~$j$ are generating
  cofibrations. But this case follows from
  Proposition~\ref{prop:boxprod_rigid_mono} by the previous lemma.
\end{proof}

\begin{remark}
  The previous result was first established by the second author (see
  \cite[Proposition 5.1.2.7]{LucasThesis}) using cubical sets. The
  advantage of the method of the present paper is that it will adapt
  directly to the join of \oo-categories (see Section~\ref{sec:join}).
\end{remark}

\begin{corollary}
  The tensor product of two cofibrant \oo-categories is a cofibrant
  \oo-category.
\end{corollary}

\begin{proof}
  Let $X$ and $Y$ be two \oo-categories. The corollary follows from the
  theorem applied to the \oo-functors $\emptyset \to X$ and
  $\emptyset \to Y$.
\end{proof}

\begin{remark}
  This corollary was first proved directly by Hadzihasanovic (see
  \cite[Theorem 1.35]{HadziThesis}).
\end{remark}

\section{A cylinder object for the folk model category structure}
\label{sec:cyl_obj}

In paragraph~\ref{paragr:def_J}, we introduced an \oo-category
$\J$. The goal of this section is to prove that if $X$ is a cofibrant
\oo-category, then $\J \otimes X$ is a cylinder object for $X$ in the
folk model category structure.

\medbreak

We will start by showing that the tensor product of a reversible cell by any
other $n$-cell is reversible, dealing first with the case $n = 1$. The
proof will be a bit involved and we begin by the following technical lemmas:

\begin{lemma}\label{lem:exist_adjoint_reverse}
  Let $u: x \to y$ be a reversible $n$-cell of an \oo-category
  $X$. Fix $\bar u$ a reverse of $u$ and
  $\e : u \comp_{n-1} \bar u \to 1_y$ a reversible cell. Then there
  exists a reversible $n$-cell
  $\eta : 1_x \to \bar u \comp_{n-1} u$ for which there exists a
  reversible $(n+1)$-cell
  $(\e \comp_{n-1} u) \comp_n (u \comp_{n-1} \eta) \to 1_{u}$.
\end{lemma}

\begin{proof}
  Considering the $n$-cell $u : x \to y$, when $n > 1$, as a $1$-cell of the
  \oo-category of cells of $X$ from $sx$ to $ty$, we reduce to the case $n =
  1$.

  In this case, since  $u$ is reversible, there
  exists a reversible $2$-cell $\eta' : 1_x \to \bar u \comp_{0} u$. We set
  \[ \eta = (\bar u \comp_{0} u \comp_{0} \bar\eta') \comp_1
    (\bar u \comp_{0} \bar \e \comp_{0} u) \comp_1 \eta', \] where
  $\bar \eta'$ and $\bar \e$ denote reverses of $\eta'$ and $\e$,
  respectively. The cell $\eta$ is reversible, as a composite of
  reversible cells. Moreover, we have
  \[
    \begin{split}
      \MoveEqLeft
      (\e \comp_{0} u) \comp_1 (u \comp_{0} \eta) \\
      & = (\e \comp_{0} u) \comp_1 (u
      \comp_{0} \bar u \comp_{0} u \comp_0 \bar\eta') \comp_1 (u \comp_0 \bar u
      \comp_{0} \bar \e \comp_{0} u) \comp_1 (u \comp_{0} \eta')
      \\
      & = (u\comp_{0} \bar\eta') \comp_1 (\e \comp_{0} u \comp_{0} \bar u
      \comp_{0} u) \comp_1 (u \comp_0 \bar u
      \comp_{0} \bar \e \comp_{0} u) \comp_1 (u \comp_{0} \eta')
      \\
      & = (u \comp_{0} \bar \eta') \comp_1
      ((\bar \e \comp_1 \e) \comp_{0} u) \comp_1
      (u \comp_{0} \eta') \\
      & \to (u \comp_{0} \bar \eta') \comp_1 (u \comp_{0} \eta) = u
      \comp_{0} (\bar \eta' \comp_1 \eta) \to 1_u,
    \end{split}
  \]
  where the two arrows are reversible $2$-cells coming from the fact
  that $\bar\e$ and $\bar\eta'$ are reverses of $\e$ and $\eta'$,
  respectively, which concludes the proof of the lemma.
\end{proof}

\begin{lemma}\label{lem:Gamma_rev_is_rev}
  Let $X$ be an \oo-category and let $c$ be an $n$-cell of
  $\Cyl(X)$. If $c$ is reversible in $\Cyl(X)$, then the principal
  cell of $c$ is reversible in $X$.
\end{lemma}

\begin{proof}
  In this proof, we will use freely the explicit formulas for the structure
  of \oo-category of $\Cyl(X)$, as given for instance in \cite[Proposition
  B.1.15]{AraMaltsiJoint}.  We prove the lemma by coinduction (see
  paragraph~\ref{paragr:def_coind}). Let $c = (x, y, \alpha)$ be an $n$-cell
  of $\Cyl(X)$ (see paragraph~\ref{paragr:desc_cylinder}). We
  have
  \[
    \alpha_n: \alpha_{n-1}^+ \comp_{n-1} \alpha_{n-2}^+ \comp_{n-2}
    \cdots \comp_1 \alpha_0^+ \comp_0 x \to y \comp_0 \alpha_0^-
    \comp_1 \cdots \comp_{n-1} \alpha_{n-1}^-.
  \]
  Suppose that $(x,y,\alpha)$ is reversible in $\Cyl(X)$ and let $(\bar x,
  \bar y, \beta)$ be a reverse. Note that $\bar x$ and $\bar y$ are
  reverses of $x$ and $y$, respectively. The relationship between the source and target of
  $(x,y,\alpha)$ and $(\bar x, \bar y, \beta)$ implies that
  \begin{gather*}
    \beta_k^\e = \alpha_k^\e \qquad \text{if $0 \le k < n-1$ and $\e =
    \pm$,} \\
    \beta_{n-1}^- = \alpha_{n-1}^+ \quadand
    \beta_{n-1}^+ = \alpha_{n-1}^+, \\
    \beta_n: \alpha_{n-1}^- \comp_{n-1} \alpha_{n-2}^+
    \comp_{n-2} \cdots \comp_1 \alpha_0^+ \comp_0 \bar x \to \bar y
    \comp_0 \alpha_0^- \comp_1 \cdots \comp_{n-2} \alpha_{n-2}^-
    \comp_{n-1} \alpha_{n-1}^+.
  \end{gather*}
  By hypothesis, there exists a reversible cell
  \[
    (\e_x,\e_y,\Lambda) : (x,y,\alpha) \comp_{n-1} (\bar x, \bar
    y, \beta) \to \id{t(x, y, \alpha)}.
  \]
  In particular, the cells
  \[
    \e_x : x \comp_{n-1} \bar{x} \to \id{tx}
    \quadand
    \e_y : y \comp_{n-1} \bar{y} \to \id{ty}
  \]
  are reversible. The cell $(x,y,\alpha) \comp_{n-1} (\bar x, \bar y,
  \beta)$ is of the form $(x \comp_{n-1} \bar x, y \comp_{n-1} \bar y,
  \gamma)$, with
    \begin{gather*}
    \gamma_k^\e = \alpha_k^\e \qquad\text{if $0 \le k < n - 1$ and $\e =
    \pm$}, \\
    \gamma_{n-1}^- = \beta_{n-1}^- = \alpha_{n-1}^+
    \quadand \gamma_{n-1}^+ = \alpha_{n-1}^+, \\
    \gamma_n = (y \comp_0 \alpha_0^- \comp_1 \cdots \comp_{n-2}
    \alpha_{n-2}^- \comp_{n-1} \beta_n) \comp_n (\alpha_n
    \comp_{n-1} \alpha_{n-2}^+ \comp_{n-2} \cdots \comp_1 \alpha_0^+
    \comp_0 \bar x)
    \end{gather*}
  and we have
  \[
    \Lambda_{n+1}: \alpha_{n-1}^+ \comp_{n-1} \cdots \comp_1
    \alpha_0^+ \comp_0 \e_x \to \e_y \comp_0 \alpha_0^- \comp_1
    \cdots \comp_{n-2} \alpha_{n-2}^- \comp_{n-1} \alpha_{n-1}^+
    \comp_{n} \gamma_n.
  \]
   Applying Lemma \ref{lem:exist_adjoint_reverse} to the cell
   $(\e_x,\e_y,\Lambda)$ gives a reversible cell
  \[
    (\eta_x,\eta_y,\Gamma) : \id{s(x, y, \alpha)} \to
      (\bar x, \bar y, \beta) \comp_{n-1} (x,y,\alpha)
  \]
  and, using the projections $\pi^-$
  and~$\pi^+$, reversible cells
  \[
    (\e_x \comp_{n-1} x) \comp_n (x \comp_{n-1} \eta_x) \to 1_{x}
    \quadand
    (\e_y \comp_{n-1} y) \comp_n (y \comp_{n-1} \eta_y) \to 1_{y}.
  \]
  The cell $(\bar x, \bar y, \beta) \comp_{n-1} (x,y,\alpha)$ is of the form
  $(\bar x \comp_{n-1} x, \bar y \comp_{n-1} y, \delta)$ with
  \begin{gather*}
    \delta_k^\e = \alpha_k^\e \qquad\text{if $0 \le k < n - 1$ and $\e =
    \pm$}, \\
    \delta_{n-1}^- = \alpha_{n-1}^-
    \quadand \gamma_{n-1}^+ = \beta_{n-1}^+ = \alpha_{n-1}^-, \\
    \delta_n = (\bar y \comp_0 \alpha_0^- \comp_1 \cdots \comp_{n-2}
    \alpha_{n-2}^- \comp_{n-1} \alpha_n) \comp_n (\beta_n
    \comp_{n-1} \alpha_{n-2}^+ \comp_{n-2} \cdots \comp_1 \alpha_0^+
    \comp_0 x)
  \end{gather*}
  and we have
  \[
    \Gamma_{n+1} : \delta_n \comp_n \alpha_{n-1}^- \comp_{n-1}
    \alpha_{n-2}^+ \comp_{n-2} \cdots \comp_1 \alpha_0^+ \comp_0 \eta_x
    \to \eta_y \comp_0 \alpha_0^- \comp_1 \cdots \comp_{n-1}
    \alpha_{n-1}^-.
  \]

  We now define our candidate $n$-cell $\rho$ to be a reverse of
  $\alpha_n$:
  \[
    \begin{split}
      \rho & = \big((\e_y \comp_0 \alpha_0^- \comp_1 \cdots \comp_{n-2}
      \alpha_{n-2}^-) \comp_{n-1} \alpha_{n-1}^+ \comp_{n-1}
      (\alpha_{n-2}^+ \comp_{n-2} \cdots \comp_1 \alpha_0^+ \comp_0
      x)\big)
      \\
      & \quad\quad \comp_n \big((y \comp_0 \alpha_0^- \comp_1 \cdots
      \comp_{n-2} \alpha_{n-2}^-) \comp_{n-1} \beta_n
      \comp_{n-1} (\alpha_{n-2}^+ \comp_{n-2} \cdots \comp_1
      \alpha_0^+ \comp_0 x)\big)
      \\
      & \quad\quad \comp_n \big((y \comp_0 \alpha_0^- \comp_1 \cdots
      \comp_{n-1} \alpha_{n-1}^-) \comp_{n-1} (\alpha_{n-2}^+
      \comp_{n-2} \cdots \comp_1 \alpha_0^+ \comp_0 \eta_x)\big).
    \end{split}
  \]

  We first produce a reversible cell between $\rho \comp_n \alpha_n$
  and $\id{s(\alpha_n)}$. We have
  {
    \allowdisplaybreaks
    \begin{align*}
      \MoveEqLeft
      \rho \comp_n \alpha_n \\*
      & = \big((\e_y \comp_0 \alpha_0^- \comp_1 \cdots \comp_{n-2}
      \alpha_{n-2}^-) \comp_{n-1} \alpha_{n-1}^+ \comp_{n-1}
      (\alpha_{n-2}^+ \comp_{n-2} \cdots \comp_1 \alpha_0^+ \comp_0
      x)\big)
      \\*
      & \quad\quad \comp_n \big((y \comp_0 \alpha_0^- \comp_1 \cdots
      \comp_{n-2} \alpha_{n-2}^-) \comp_{n-1} \beta_n
      \comp_{n-1} (\alpha_{n-2}^+ \comp_{n-2} \cdots \comp_1
      \alpha_0^+ \comp_0 x)\big)
      \\*
      & \quad\quad \comp_n \big((y \comp_0 \alpha_0^- \comp_1 \cdots
      \comp_{n-1} \alpha_{n-1}^-) \comp_{n-1} (\alpha_{n-2}^+
      \comp_{n-2} \cdots \comp_1 \alpha_0^+ \comp_0 \eta_x)\big)
      \\
      & \quad\quad \comp_n \alpha_n
      \\
      & = \big((\e_y \comp_0 \alpha_0^- \comp_1 \cdots \comp_{n-2}
      \alpha_{n-2}^-) \comp_{n-1} \alpha_{n-1}^+ \comp_{n-1}
      (\alpha_{n-2}^+ \comp_{n-2} \cdots \comp_1 \alpha_0^+ \comp_0
      x)\big)
      \\*
      & \quad\quad \comp_n \big((y \comp_0 \alpha_0^- \comp_1 \cdots
      \comp_{n-2} \alpha_{n-2}^-) \comp_{n-1} \beta_n
      \comp_{n-1} (\alpha_{n-2}^+ \comp_{n-2} \cdots \comp_1
      \alpha_0^+ \comp_0 x)\big)
      \\*
      & \quad\quad \comp_n \big(\alpha_n
      \comp_{n-1} (\alpha_{n-2}^+ \comp_{n-2} \cdots \comp_1
      \alpha_0^+ \comp_0 (\bar x \comp_{n-1} x))\big)
      \\*
      & \quad\quad
      \comp_n \big((\alpha_{n-1}^+ \comp_{n-1} \alpha_{n-2}^+ \comp_{n-2}
      \cdots \comp_1 \alpha_0^+ \comp_0 x) \comp_{n-1} (\alpha_{n-2}^+
      \comp_{n-2} \cdots \comp_1 \alpha_0^+ \comp_0 \eta_x)\big)
      \\
      & = \big((\e_y \comp_0 \alpha_0^- \comp_1 \cdots \comp_{n-2}
      \alpha_{n-2}^-) \comp_{n-1} \alpha_{n-1}^+ \comp_{n-1}
      (\alpha_{n-2}^+ \comp_{n-2} \cdots \comp_1 \alpha_0^+ \comp_0
      x)\big)
      \\*
      & \quad\quad \comp_n \big((y \comp_0 \alpha_0^- \comp_1 \cdots
      \comp_{n-2} \alpha_{n-2}^-) \comp_{n-1} \beta_n
      \comp_{n-1} (\alpha_{n-2}^+ \comp_{n-2} \cdots \comp_1
      \alpha_0^+ \comp_0 x)\big)
      \\*
      & \quad\quad \comp_n \big(\alpha_n
      \comp_{n-1} (\alpha_{n-2}^+ \comp_{n-2} \cdots \comp_1
      \alpha_0^+ \comp_0 \bar x) \comp_{n-1} (\alpha_{n-2}^+
      \comp_{n-2} \cdots \comp_1 \alpha_0^+ \comp_0 x)\big) \\
      & \quad\quad \comp_n \big(\alpha_{n-1}^+
      \comp_{n-1} \alpha_{n-2}^+ \comp_{n-2} \cdots \comp_1 \alpha_0^+
      \comp_0 (x \comp_{n-1} \eta_x)\big)
      \\
      & = \big((\e_y \comp_0 \alpha_0^- \comp_1 \cdots \comp_{n-2}
      \alpha_{n-2}^-) \comp_{n-1} \alpha_{n-1}^+ \comp_{n-1}
      (\alpha_{n-2}^+ \comp_{n-2} \cdots \comp_1 \alpha_0^+ \comp_0
      x)\big)
      \\*
      & \quad\quad \comp_n \big(
      \big[(y \comp_0 \alpha_0^- \comp_1 \cdots \comp_{n-2}
      \alpha_{n-2}^- \comp_{n-1} \beta_n)
      \\*
      & \quad\quad\quad\quad
      \comp_{n-1} (\alpha_{n-2}^+ \comp_{n-2} \cdots \comp_1 \alpha_0^+
      \comp_0 \bar x)\big]
      \comp_{n-1} (\alpha_{n-2}^+
      \comp_{n-2} \cdots \comp_1 \alpha_0^+ \comp_0 x)\big)
      \\*
      & \quad\quad \comp_n \big(\alpha_{n-1}^+
      \comp_{n-1} \alpha_{n-2}^+ \comp_{n-2} \cdots \comp_1 \alpha_0^+
      \comp_0 (x \comp_{n-1} \eta_x)\big)
      \\
      & = \big((\e_y \comp_0 \alpha_0^- \comp_1 \cdots \comp_{n-2}
      \alpha_{n-2}^-) \comp_{n-1} \alpha_{n-1}^+ \comp_{n-1}
      (\alpha_{n-2}^+ \comp_{n-2} \cdots \comp_1 \alpha_0^+ \comp_0
      x)\big)
      \\*
      & \quad\quad \comp_n \big(\gamma_n \comp_{n-1} (\alpha_{n-2}^+
      \comp_{n-2} \cdots \comp_1 \alpha_0^+ \comp_0 x)\big)
      \\*
      & \quad\quad \comp_n \big(\alpha_{n-1}^+
      \comp_{n-1} \alpha_{n-2}^+ \comp_{n-2} \cdots \comp_1 \alpha_0^+
      \comp_0 (x \comp_{n-1} \eta_x)\big)
      \\
      & = \big((\e_y \comp_0 \alpha_0^- \comp_1 \cdots \comp_{n-2}
      \alpha_{n-2}^- \comp_{n-1} \alpha_{n-1}^+ \comp_n \gamma_n)
      \comp_{n-1} (\alpha_{n-2}^+ \comp_{n-2} \cdots \comp_1
      \alpha_0^+ \comp_0 x)\big)
      \\*
      & \quad\quad \comp_n \big(\alpha_{n-1}^+
      \comp_{n-1} \alpha_{n-2}^+ \comp_{n-2} \cdots \comp_1 \alpha_0^+
      \comp_0 (x \comp_{n-1} \eta_x)\big).
    \end{align*}
  }%
  By coinduction, the cell
  \[
    \Lambda_{n+1}: \alpha_{n-1}^+ \comp_{n-1} \cdots \comp_1
    \alpha_0^+ \comp_0 \e_x \to \e_y \comp_0 \alpha_0^- \comp_1
    \cdots \comp_{n-2} \alpha_{n-2}^- \comp_{n-1} \alpha_{n-1}^+
    \comp_{n} \gamma_n
  \]
  is reversible and we thus get a reversible cell between $\rho \comp_n
  \alpha_n$ and the cell
  \[
    \begin{split}
      \MoveEqLeft \big((\alpha_{n-1}^+ \comp_{n-1} \cdots \comp_1
      \alpha_0^+ \comp_0 \e_x) \comp_{n-1} (\alpha_{n-2}^+
      \comp_{n-2} \cdots \comp_1 \alpha_0^+ \comp_0 x)\big)
      \\
      \MoveEqLeft \quad\quad \comp_n \big(\alpha_{n-1}^+ \comp_{n-1}
      \alpha_{n-2}^+ \comp_{n-2} \cdots \comp_1 \alpha_0^+ \comp_0 (x
      \comp_{n-1} \eta_x)\big)
      \\
      & = \alpha_{n-1}^+ \comp_{n-1}
      \alpha_{n-2}^+ \comp_{n-2} \cdots \comp_1 \alpha_0^+ \comp_0
      ((\e_x \comp_{n-1} x) \comp_n (x \comp_{n-1} \eta_x))
    \end{split}
  \]
  and hence a reversible cell between $\rho \comp_n \alpha_n$ and the
  identity on
  \[
    \alpha_{n-1}^+ \comp_{n-1} \alpha_{n-2}^+
    \comp_{n-2} \cdots \comp_1 \alpha_0^+ \comp_0 x
    = s(\alpha_n).
  \]

  We now produce a reversible cell between $\alpha_n \comp_n \rho$
  and $\id{t(\alpha_n)}$. We have
  {
    \allowdisplaybreaks
    \begin{align*}
      \MoveEqLeft \alpha_n \comp_n \rho \\*
      & =
      \alpha_n \\*
      & \quad\quad \comp_n \big((\e_y \comp_0 \alpha_0^- \comp_1 \cdots
      \comp_{n-2} \alpha_{n-2}^-) \comp_{n-1} \alpha_{n-1}^+
      \comp_{n-1} (\alpha_{n-2}^+ \comp_{n-2} \cdots \comp_1
      \alpha_0^+ \comp_0 x)\big)
      \\*
      & \quad\quad \comp_n \big((y \comp_0 \alpha_0^- \comp_1 \cdots
      \comp_{n-2} \alpha_{n-2}^-) \comp_{n-1} \beta_n
      \comp_{n-1} (\alpha_{n-2}^+ \comp_{n-2} \cdots \comp_1
      \alpha_0^+ \comp_0 x)\big)
      \\*
      & \quad\quad \comp_n \big((y \comp_0 \alpha_0^- \comp_1 \cdots
      \comp_{n-1} \alpha_{n-1}^-) \comp_{n-1} (\alpha_{n-2}^+
      \comp_{n-2} \cdots \comp_1 \alpha_0^+ \comp_0 \eta_x)\big)
      \\
      & = \big((\e_y \comp_0 \alpha_0^- \comp_1 \cdots \comp_{n-2}
      \alpha_{n-2}^-) \comp_{n-1} (y \comp_0 \alpha_0^- \comp_1 \cdots
      \comp_{n-1} \alpha_{n-1}^-)\big)
      \\*
      & \quad\quad \comp_n \big(((y \comp_{n-1} \bar y) \comp_0 \alpha_0^-
      \comp_1 \cdots \comp_{n-2} \alpha_{n-2}^-) \comp_{n-1} \alpha_n\big)
      \\*
      & \quad\quad \comp_n \big((y \comp_0 \alpha_0^- \comp_1 \cdots
      \comp_{n-2} \alpha_{n-2}^-) \comp_{n-1} \beta_n
      \comp_{n-1} (\alpha_{n-2}^+ \comp_{n-2} \cdots \comp_1
      \alpha_0^+ \comp_0 x)\big)
      \\*
      & \quad\quad \comp_n \big((y \comp_0 \alpha_0^- \comp_1 \cdots
      \comp_{n-1} \alpha_{n-1}^-) \comp_{n-1} (\alpha_{n-2}^+
      \comp_{n-2} \cdots \comp_1 \alpha_0^+ \comp_0 \eta_x)\big)
      \\
      & = \big((\e_y \comp_0 \alpha_0^- \comp_1 \cdots \comp_{n-2}
      \alpha_{n-2}^-) \comp_{n-1} (y \comp_0 \alpha_0^- \comp_1 \cdots
      \comp_{n-1} \alpha_{n-1}^-)\big)
      \\*
      & \quad\quad \comp_n \big((y \comp_0 \alpha_0^- \comp_1 \cdots
      \comp_{n-1} \alpha_{n-2}^-) \comp_{n-1}
      \big[(\bar y \comp_0 \alpha_0^- \comp_1 \cdots \comp_{n-2}
      \alpha_{n-2}^- \comp_{n-1} \alpha_n)
      \\*
      & \qquad\qquad \comp_n (\beta_n
      \comp_{n-1} \alpha_{n-2}^+ \comp_{n-2} \cdots \comp_1 \alpha_0^+
      \comp_0 x)\big]\big)
      \\*
      & \quad\quad \comp_n \big((y \comp_0 \alpha_0^- \comp_1 \cdots
      \comp_{n-1} \alpha_{n-1}^-) \comp_{n-1} (\alpha_{n-2}^+
      \comp_{n-2} \cdots \comp_1 \alpha_0^+ \comp_0 \eta_x)\big)
      \\
      & = \big((\e_y \comp_0 \alpha_0^- \comp_1 \cdots \comp_{n-2}
      \alpha_{n-2}^-) \comp_{n-1} (y \comp_0 \alpha_0^- \comp_1 \cdots
      \comp_{n-1} \alpha_{n-1}^-)\big)
      \\*
      & \quad\quad \comp_n \big((y \comp_0 \alpha_0^- \comp_1 \cdots
      \comp_{n-1} \alpha_{n-2}^-) \comp_{n-1} \delta_n\big)
      \\*
      & \quad\quad \comp_n \big((y \comp_0 \alpha_0^- \comp_1 \cdots
      \comp_{n-1} \alpha_{n-1}^-) \comp_{n-1} (\alpha_{n-2}^+
      \comp_{n-2} \cdots \comp_1 \alpha_0^+ \comp_0 \eta_x)\big)
      \\
      & = \big((\e_y \comp_0 \alpha_0^- \comp_1 \cdots \comp_{n-2}
      \alpha_{n-2}^-) \comp_{n-1} (y \comp_0 \alpha_0^- \comp_1 \cdots
      \comp_{n-1} \alpha_{n-1}^-)\big)
      \\*
      & \quad\quad \comp_n \big((y \comp_0 \alpha_0^- \comp_1 \cdots
      \comp_{n-1} \alpha_{n-2}^-)
      \\*
      & \qquad\qquad \comp_{n-1} (\delta_n \comp_n
      \alpha_{n-1}^- \comp_{n-1} \alpha_{n-2}^+ \comp_{n-2} \cdots
      \comp_1 \alpha_0^+ \comp_0 \eta_x)\big).
    \end{align*}
  }%
  By coinduction, the cell
  \[
    \Gamma_{n+1} : \delta_n \comp_n \alpha_{n-1}^- \comp_{n-1}
    \alpha_{n-2}^+ \comp_{n-2} \cdots \comp_1 \alpha_0^+ \comp_0 \eta_x
    \to \eta_y \comp_0 \alpha_0^- \comp_1 \cdots \comp_{n-1}
    \alpha_{n-1}^-
  \]
  is reversible and we thus get a reversible cell from $\alpha_n \comp \rho$
  to
  \[
    \begin{split}
      \MoveEqLeft \big((\e_y \comp_0 \alpha_0^- \comp_1 \cdots
      \comp_{n-2} \alpha_{n-2}^-) \comp_{n-1} (y \comp_0 \alpha_0^-
      \comp_1 \cdots \comp_{n-1} \alpha_{n-1}^-)\big)
      \\
      \MoveEqLeft \quad\quad \comp_n \big((y \comp_0 \alpha_0^- \comp_1
      \cdots \comp_{n-1} \alpha_{n-2}^-) \comp_{n-1} (\eta_y
      \comp_0 \alpha_0^- \comp_1 \cdots \comp_{n-1} \alpha_{n-1}^-)\big)
      \\
      & = ((\e_y \comp_{n-1} y) \comp_n (y \comp_{n-1} \eta_y))
      \comp_0 \alpha_0^- \comp_1 \cdots \comp_{n-1} \alpha_{n-1}^-
    \end{split}
  \]
  and hence a reversible cell between $\alpha_n \comp_n \rho$ and the
  identity on
  \[
    y \comp_0 \alpha_0^- \comp_1 \cdots \comp_{n-1} \alpha_{n-1}^-
    =
    t(\alpha_n),
  \]
  hence the result.
\end{proof}

\begin{paragraph}
  We will denote by $\Rn{n}$ the free-standing reversible $n$-cell
  and by $r_n : \Dn{n} \to \Rn{n}$ the canonical \oo-functor.
  This \oo-category $\Rn{n}$ is freely generated by
  \begin{itemize}
    \item two $n$-cells
      \[ r : x \to y, \quad \winv{r} : y \to x, \]
    \item four $(n+1)$-cells
      \[
        \alpha : \winv{r} \comp_{n-1} r \to \id{x}, \,\,
        \winv{\alpha} : \id{x} \to \winv{r} \comp_{n-1} r,\,\, \beta : r
        \comp_{n-1} \winv{r} \to y,\,\, \winv{\beta} : y \to r \comp_{n-1} \winv{r},
      \]
    \item eight $(n+2)$-cells comparing $\winv{\alpha} \comp_n \alpha$,
      $\alpha \comp_n \winv{\alpha}$, $\winv{\beta} \comp_n \beta$
      and $\beta \comp_n \winv{\beta}$ to identities,
    \item etc.~(see the remark below for a formal description).
  \end{itemize}
  The \oo-functor $r_n : \Dn{n} \to \Rn{n}$ sends the principal cell of
  $\Dn{n}$ to $r$, \ndef{the principal cell of~$\Rn{n}$}. Note that this
  \oo-functor is a folk cofibration. By definition, an $n$-cell $x$ of an
  \oo-category $X$ is reversible if and only if the \oo-functor $\langle x
  \rangle : \Dn{n} \to X$ factors through~$r_n$. Note that such a
  factorization corresponds to a choice of witnesses that $r$ is reversible
  and is hence not unique.
\end{paragraph}

\begin{remark}
  More formally, the \oo-category $\Rn{n}$ is generated by two $n$-cells
      \[ r : x \to y, \quad \winv{r} : y \to x, \]
 and, for every $i > n$, two sets of $2^{i-n}$ $i$-cells
  \[ r^{}_{l_1, \dots, l_{i-n}} \quadand \winv{r}^{}_{l_1, \dots, l_{i-n}}, \]
  where $l_j = \pm1$ for $1 \le j \le i - n$, whose sources and targets are
  given by
  \begin{align*}
    r^{}_{l_1, \dots, l_{i-n-1}, -}
    &:&
    \winv{r}^{}_{l_1, \dots, l_{i-n-1}} \comp_{i-1} r^{}_{l_1, \dots, l_{i-n-1}}
    &\to
    \id{s(r_{l_1, \dots, l_{i-n-1}})}
    \\
    \winv{r}^{}_{l_1, \dots, l_{i-n-1}, -}
    &:&
    \id{s(r_{l_1, \dots, l_{i-n-1}})}
    &\to
    \winv{r}^{}_{l_1, \dots, l_{i-n-1}} \comp_{i-1} r^{}_{l_1, \dots, l_{i-n-1}}
    \\
    r^{}_{l_1, \dots, l_{i-n-1}, +}
    &:&
    r^{}_{l_1, \dots, l_{i-n-1}} \comp_{i-1} \winv{r}^{}_{l_1, \dots, l_{i-n-1}}
    &\to
    \id{t(r_{l_1, \dots, l_{i-n-1}})}
    \\
    \winv{r}^{}_{l_1, \dots, l_{i-n-1}, +}
    &:&
    \id{t(r_{l_1, \dots, l_{i-n-1}})}
    &\to
    r^{}_{l_1, \dots, l_{i-n-1}} \comp_{i-1} \winv{r}^{}_{l_1, \dots, l_{i-n-1}}
    \pbox{.}
  \end{align*}
  (With this notation, the $(n+1)$-cells $\alpha$ and $\beta$ of the previous
  paragraph are $\alpha = r_{-}$ and~$\beta = r_+$.)
\end{remark}

\begin{proposition}\label{prop:prod_rev_by_1cell}
  If $x$ is a $1$-cell of an \oo-category $X$ and $y$ is a reversible
  $n$-cell of an \oo-category $Y$, then $x \otimes y$ is reversible in
  $X \otimes Y$.
\end{proposition}

\begin{proof}
  Since $y$ is reversible, $\langle y \rangle : \Dn{n} \to Y$ factors
  through $r_n : \Dn{n} \to \Rn{n}$. Tensoring by $\langle x \rangle$,
  we therefore get an \oo-functor 
  $\Dn{1} \otimes \Rn{n} \to X \otimes Y$, which corresponds by adjunction
  to an \oo-functor $\Rn{n} \to \Cyl(X \otimes Y)$ and hence to a reversible
  cell of~$\Cyl(X \otimes Y)$. By Lemma \ref{lem:Gamma_rev_is_rev}, the
  principal cell of this cylinder is reversible in $X \otimes Y$. But this
  principal cell corresponds to the composite
  \[
    \Dn{n+1} \xto{\funcell{p}} \Dn{1} \otimes \Dn{n} \xto{\langle x \rangle
      \otimes \langle y \rangle} X \otimes Y,
  \]
  where $p$ denotes the principal cell of $\Dn{1} \otimes \Dn{n}$,
  which is $\langle x \otimes y \rangle$. This shows that
  $x \otimes y$ is reversible.
\end{proof}

\begin{corollary}\label{coro:prod_rev_dual}
  If $x$ is a reversible $n$-cell of an \oo-category $X$ and $y$ is a
  $1$-cell of an \oo-category $Y$, then $x \otimes y$ is reversible in
  $X \otimes Y$.
\end{corollary}

\begin{proof}
  We will use the duality $Z \mapsto Z^\op$ introduced in
  paragraph~\ref{paragr:dualities}, denoting by $z^\op$ the cell of
  $Z^\op$ corresponding to a cell $z$ of $Z$. By this same paragraph,
  if $x$ is a reversible cell of $X$ and $y$ is a $1$-cell of $Y$,
  then $x \otimes y$ is reversible in $X \otimes Y$ if and only if
  $(x \otimes y)^\op$ is reversible in $(X \otimes Y)^\op$ if and only
  if $y^\op \otimes x^\op$ is reversible in $Y^\op \otimes X^\op$. As
  $x^\op$ is reversible in $X^\op$, the result thus follows from the
  previous proposition.
\end{proof}

We will now show that the tensor product of a reversible cell by any cell is reversible.
We will need a specific \oo-functor from $\Dn{n-1} \otimes \Dn{1}$ to
$\Dn{n}$ that we now introduce.

\begin{paragraph}\label{paragr:prod_DnD1_to_Dn}
  Let $n \ge 1$. The \oo-category obtained from $\Dn{n-1} \otimes \Dn{1}$ by collapsing,
  independently, the sub-\oo-categories $\Dn{n-1} \otimes \{0\}$ and
  $\Dn{n-1} \otimes \{1\}$ is canonically isomorphic to~$\Dn{n}$. (This
  follows for instance from \cite[Corollary B.6.6]{AraMaltsiJoint} using
  the duality $X \mapsto X^\co$.) In particular, there is a canonical
  \oo-functor
  \[ \Dn{n-1} \otimes \Dn{1} \to \Dn{n} \]
  sending $\Dn{n-1} \otimes \{0\}$ to $0$, $\Dn{n-1} \otimes \{1\}$ to $1$, and
  the principal cell of $\Dn{n-1} \otimes \Dn{1}$ to the principal cell of
  $\Dn{n}$.

  By iterating this construction, we get an \oo-functor
  \[ \Dn{1}^{\otimes n} \to \Dn{n} \]
  sending \ndef{the principal cell of $\Dn{1}^{\otimes n}$}, that is, the tensor
  product of the principal cells of the $n$ copies of $\Dn{1}$, to the
  principal cell of $\Dn{n}$.
\end{paragraph}

\begin{lemma}\label{lem:prod_Rk_Dn}
  If $r$ is the principal cell of $\Rn{k}$ and $d$ is the principal
  cell of $\Dn{n}$, then $r \otimes d$ is reversible in
  $\Rn{k} \otimes \Dn{n}$.
\end{lemma}

\begin{proof}
  By the Corollary~\ref{coro:prod_rev_dual},
  for every $m \ge 0$, the tensor product of the principal cells of $\Rn{m}$
  and $\Dn{1}$ is reversible, showing that there exists an \oo-functor
  $p' : \Rn{m+1} \to \Rn{m} \otimes \Dn{1}$ making commutative the square
  \[
    \xymatrix@C=2.5pc{
      \Dn{m+1} \ar[d]_p \ar[r]^{r_{m+1}} & \Rn{m+1} \ar[d]^{p'} \\
      \Dn{m} \otimes \Dn{1} \ar[r]_{r_m \otimes \Dn{1}} & \Rn{m}
      \otimes \Dn{1} \pbox{,} }
  \]
  where $p$ corresponds to the principal cell of $\Dn{m} \otimes \Dn{1}$.
  Denote this square by $S_m$. By composing
  the $n$ squares $S_{n+k-1-m} \otimes \Dn{1}^{\otimes m}$
  for $0 \le m \le n-1$
  \[
    \xymatrix{
      \Dn{n+k} \ar[d] \ar[r] & \Rn{n+k} \ar[d] \\
      \Dn{n+k-1} \otimes \Dn{1} \ar[d] \ar[r] & \Rn{n+k-1} \otimes \Dn{1} \ar[d] \\
      (\Dn{n+k-2} \otimes \Dn{1}) \otimes \Dn{1} \ar@{.>}[d] \ar[r] &
      (\Rn{n+k-2}
      \otimes \Dn{1}) \otimes \Dn{1} \ar@{.>}[d] \\
      (\Dn{k} \otimes \Dn{1}) \otimes \Dn{1}^{\otimes n-1} \ar[r] &
      (\Rn{k} \otimes \Dn{1}) \otimes \Dn{1}^{\otimes n-1} \pbox{,} }
  \]
  we get a commutative square
  \[
    \xymatrix@C=2.5pc{
      \Dn{n+k} \ar[d] \ar[r]^{r_{n+k}} & \Rn{n+k}  \ar[d] \\
      \Dn{k} \otimes \Dn{1}^{\otimes n} \ar[r]_{r_k \otimes
        \Dn{1}^{\otimes n}} & \Rn{k} \otimes \Dn{1}^{\otimes n}
      \pbox{.}  }
  \]
  By composing this square with the \oo-functor $\Dn{1}^{\otimes n} \to
  \Dn{n}$ of the previous paragraph,
  we get a commutative square \[
    \xymatrix@C=2.5pc{
      \Dn{n+k} \ar[d] \ar[r]^{r_{n+k}} & \Rn{n+k}  \ar[d] \\
      \Dn{k} \otimes \Dn{n} \ar[r]_{r_k \otimes \Dn{n}} & \Rn{k}
      \otimes \Dn{n}
    }
  \]
  showing that $\funcell{r \otimes d}$ factors through $r_{n+k}$ and
  hence that $r \otimes d$ is reversible.
\end{proof}

\goodbreak

\begin{proposition}\label{prop:prod_Rm_Dn}
  Let $x$ be an $m$-cell of an \oo-category $X$ and let $y$ be an $n$-cell
  of an \oo-category $Y$. If either $x$ or $y$ is reversible, then $x
  \otimes y$ is reversible in $X \otimes Y$.
\end{proposition}

\begin{proof}
  We start with the case where $x$ is reversible. The cell
  $x \otimes y$ then corresponds to the composite
  \[
    \funcell{x \otimes y} : \Dn{m+n} \xto{\funcell{r \otimes d}}
    \Rn{m} \otimes \Dn{n} \xto{\funcell{x} \otimes \funcell{y}} X
    \otimes Y,
  \]
  where $r$ and $d$ denotes the principal cells of $\Rn{m}$ and $\Dn{n}$
  respectively. But by the previous lemma, $r \otimes d$ is reversible, and
  thus so is
  \[
    x \otimes y = (\funcell{x} \otimes \funcell{y})(r \otimes d).
  \]

  Suppose now that $y$ is reversible. Then $y^\op$ is reversible, and so is
  $y^\op \otimes x^\op$ in $Y^\op \otimes X^\op$ by the previous case. This
  proves that $x \otimes y = (y^\op \otimes x^\op)^\op$ is reversible.
\end{proof}

We will use the previous proposition to study the notion of
a $\J$-transformation that we now introduce.

\begin{paragraph}
  Let $f, g : X \to Y$ be two \oo-functors. A \ndef{$\J$-transformation}
  from $f$ to $g$ is an \oo-functor
  \[ h : \J \otimes X \to Y \]
  making commutative the diagram
  \[
    \xymatrix@R=1.5pc{
      X \ar@/^2.5ex/[drr]^f \ar[dr]_-{\funcell{0} \otimes X} \\
      & \J \otimes X \ar[r]^h & Y \\
      X \ar@/_2.5ex/[urr]_g \ar[ur]^-{\funcell{1} \otimes X} & &
      \pbox{,} }
  \]
  where we denoted by $0$ and $1$ the image by $l_1 : \Dn{1} \to \J$
  (see paragraph~\ref{paragr:def_J}) of the objects $0$ and $1$ of
  $\Dn{1}$.

  Note that if $h : \J \otimes X \to Y$ is a $\J$-transformation then
  the composite
  \[ \Dn{1} \otimes X \xto{l_1 \otimes X} \J \otimes X \xto{\,\,\,
      h\,\,\,} Y \] defines an oplax transformation from $f$ to $g$.
\end{paragraph}

\begin{proposition}\label{prop:J-trans_rev}
  The oplax transformation associated to a $\J$-transformation is
  reversible.
\end{proposition}

\begin{proof}
  Let $h : \J \otimes X \to Y$ be a $\J$-transformation and let $x$ be
  an $n$-cell of~$X$. We have to show that $h((01) \otimes x)$,
  where $(01)$ denotes the principal cell of $\J$ (see
  paragraph~\ref{paragr:def_J}) is reversible. It suffices to show that
  $(01) \otimes x$ is reversible and hence, by
  Proposition~\ref{prop:prod_Rm_Dn}, that
  $(01)$ is reversible in $\J$. As the \oo-functor $r_1 : \Dn{1} \to \Rn{1}$
  is a cofibration and the \oo-functor $q_1 : \J \to \Dn{0}$ is a
  trivial fibration, the commutative square
  \[
    \xymatrix{
      \Dn{1} \ar[d]_{r_1} \ar[r]^{l_1} & \J  \ar[d]^{q_1} \\
      \Rn{1} \ar[r] \ar@{.>}[ur] & \Dn{0} }
  \]
  admits a lift, showing that $(01)$ is indeed reversible in $\J$.
\end{proof}

\begin{paragraph}\label{paragr:def_retr}
  We say that an \oo-functor $i : X \to Y$ is an \ndef{oplax transformation
    retract} \respd{a \ndef{reversible transformation retract}}{a
    \ndef{$\J$-transformation retract}} if it admits a retraction $r$,
  that is, an \oo-functor $r : Y \to X$ such that $ri = \id{X}$, and
  an oplax transformation \respd{a reversible transformation}{a
    $\J$-transformation} $\alpha : ir \tod \id{Y}$.

  It follows from the fact that reversible transformations are right
  homotopies for the folk model category structure (see
  Remark~\ref{rem:trans_right_homot}) that a reversible transformation
  retract is a folk weak equivalence.

  We say that a transformation retract $i : X \to Y$ (oplax,
  reversible or $\J$-) is \emph{strong} if $r$ and $\alpha$ as above
  can be chosen so that $\alpha \comp i = \id{i}$, in the sense that,
  if $\alpha$ is given by an \oo-functor
  \[ h : \Dn{1} \otimes Y \to Y \quador h' : \J \otimes Y \to Y, \]
  then the diagrams
  \[
    \xymatrix{ \Dn{1} \otimes X \ar[d]_{\Dn{1} \otimes i} \ar[r]^-p &
      X \ar[d]^{i \zbox{\qquad\text{or}}}
      \\
      \Dn{1} \otimes Y \ar[r]_-h & Y } \qquad \qquad \xymatrix{ \J
      \otimes X \ar[d]_{\J \otimes i} \ar[r]^-{p'} & X
      \ar[d]^{i}
      \\
      \J \otimes Y \ar[r]_-{h'} & Y \pbox{,} }
  \]
  where $p$ and $p'$ are the ``projection'' \oo-functors induced by
  $\Dn{1} \to \Dn{0}$ and $\J \to \Dn{0}$, commute.

  By \cite[Corollary 4.30]{Folk}, every folk trivial cofibration is a
  strong reversible transformation retract. (Note that strong
  reversible transformation retracts are called ``immersions'' in
  \cite{Folk}.)
\end{paragraph}

\begin{proposition}\label{prop:J1_immer_are_we}
  A $\J$-transformation retract is a reversible transformation retract
  and in particular a folk weak equivalence.
\end{proposition}

\begin{proof}
  This follows immediately from Proposition~\ref{prop:J-trans_rev}.
\end{proof}

\begin{proposition}\label{prop:J_retr_we}
  For every \oo-category $X$, the \oo-functors $X \to \J \otimes X$,
  obtained by tensoring $\funcell{0}, \funcell{1} : \Dn{0} \to \J$ by
  $X$, are $\J$-transformation retracts and hence folk weak
  equivalences.
\end{proposition}

\begin{proof}
  The class of $\J$-transformation retracts is clearly stable under
  tensor product by an object on the right. Therefore it suffices to
  prove that $\funcell{0}, \funcell{1}: \Dn{0} \to \J$ are
  $\J$-transformation retracts. Let $\epsilon = 0, 1$. The \oo-functor
  $r : \J \to \Dn{0}$ is clearly a retraction of
  $\funcell{\epsilon}$. Moreover, a $\J$-transformation
  $\funcell{\epsilon}r \tod \id{\J}$ is precisely a lift to the lifting
  problem
  \[
    \xymatrix@C=3.5pc{ \bDn{1} \otimes \J \ar[d]_{k_1 \otimes \J}
      \ar[r]^-{\funcell{\epsilon}r \, \amalg \, \J}
        & \J  \ar[d] \\
      \J \otimes \J \ar[r] \ar@{.>}[ur] & \Dn{0} \pbox{,} }
  \]
  where we identified $\bDn{1} \otimes \J$ with $\J \amalg \J$. As
  $k_1$ is a cofibration and $\J$ is a cofibrant object, it follows from
  Theorem~\ref{thm:tens_comp_cof} that the left vertical arrow is a
  cofibration. As the right vertical arrow is a trivial fibration
  by definition of $\J$, the desired lift exists, thereby proving the
  result.
\end{proof}

\begin{theorem}
  Let $X$ be an \oo-category.
  \begin{enumerate}
  \item The \oo-functor $\J \otimes X \to X$, obtained by tensoring
    $\J \to \Dn{0}$ by $X$, is a folk weak equivalence.
  \item If moreover, $X$ is cofibrant, then the factorization
    \[
      X \amalg X \to \J \otimes X \to X
    \]
    of the codiagonal, obtained by tensoring
    \[ \bDn{1} \xto{k_1} \J \to \Dn{0} \]
    by $X$, is a cylinder object for the folk model category structure, in
    the sense that the first arrow is a cofibration and the second one is a
    weak equivalence.
  \end{enumerate}
\end{theorem}

\begin{proof}
  The \oo-functor $\J \otimes X \to X$ is a retract of any of the
  \oo-functors $X \to \J \otimes X$ considered in the previous
  proposition and is hence a weak equivalence by this same proposition.
  This proves the first point. As for the second one, it follows from
  Theorem~\ref{thm:tens_comp_cof} since $k_1$ is a cofibration and $X$ is
  cofibrant.
\end{proof}

\begin{remark}
  The \oo-functor $X \amalg X \to \J \otimes X$ is not a folk cofibration in
  general. To see this, recall first that if $p : y_0 \to y_1$ is a $1$-cell
  in an \oo-category $Y$ and $f : x_0 \to x_1$, $g : x_1 \to x_2$ are two
  $1$-cells in an \oo-category $X$, then the following relation holds in~$Y
  \otimes X$:
  \[
    ((y_1 \otimes g) \comp_0 (p \otimes f))\comp_1
    ((p \otimes g) \comp_0 (y_0 \otimes f)) =
    p \otimes (g \comp_0 f).
  \]
  Diagrammatically, this means that the following $2$-cells are equal:
   \[
     \xymatrix@C=3pc@R=3pc{
       y_0 \otimes x_0
       \ar[r]^{y_0 \otimes f}
       \ar[d]_{p \otimes x_0}
       &
       y_0 \otimes x_1
       \ar[r]^{y_0 \otimes g}
       \ar[d]_{p \otimes x_1}
       \ar@{}[ld]|-{}
       &
       y_0 \otimes x_2
       \ar[d]^{p \otimes x_2}
       \\
       y_1 \otimes x_0
       \ar[r]_{y_1 \otimes f}
       &
       y_1 \otimes x_1
       \ar[r]_{y_1 \otimes g}
       &
       y_1 \otimes x_2
       \ar@{}[lu];[ll]_(0.35){}="s1"^(0.65){}="t1"
       \ar@2"s1";"t1"_{p \otimes f}
       \ar@{}[u];[l]_(0.35){}="s2"^(0.65){}="t2"
       \ar@2"s2";"t2"_{p \otimes g}
     }
     \qquad
     \xymatrix@C=3.5pc@R=3pc{
       y_0 \otimes x_0
       \ar[r]^{y_0 \otimes (g \comp_0 f)}
       \ar[d]_{p \otimes x_0}
       &
       y_0 \otimes x_2
       \ar[d]^{p \otimes x_2}
       \\
       y_1 \otimes x_0
       \ar[r]_{y_1 \otimes (g \comp_0 f)}
       &
       y_1 \otimes x_2 \pbox{.}
       \ar@{}[u];[l]_(0.35){}="s1"^(0.65){}="t1"
       \ar@2"s1";"t1"_{p \otimes (g \comp_0 f)\,\,\,}
     }
  \]

  Now take $X$ to be the category generated by one object $x$ and one arrow $f : x \to x$,
  subject to the relation $f \comp_0 f = 1_x$. In other words, $X$ is the cyclic
  group with $2$ elements, seen as a one-object category. Let $p: y_0 \to
  y_1$ be the principal cell of $\J$. Then the previous formula (taking $g =
  f$) shows that the following equality holds in $\J \otimes X$:
  \[
    ((y_1 \otimes f) \comp_0 (p \otimes f))\comp_1 ((p \otimes f) \comp_0
  (y_0 \otimes f)) = 1_{p \otimes x}.  \]
  This relation implies that any \oo-functor from $\J \otimes X$ to a
  cofibrant \oo-category must send $p \otimes f$ to an identity. Consider
  now $u : \J \otimes X \to Z$ the pushout of the obvious \oo-functor $X
  \amalg X \to \Dn{0} \amalg \Dn{0} = \bDn{1}$ along the
  \oo-functor $X \amalg X \to \J \otimes X$. One can check that the cell $u
  (p \otimes f)$ is non-trivial in $Z$. Factoring the map $\bDn{1} \to Z$
  into a cofibration followed by a trivial fibration $t : Z' \to
  Z$, we get a commutative square
  \[
   \xymatrix{
       X \amalg X
       \ar[r]
       \ar[d]
       &
       Z'
       \ar[d]^{t}
       \\
       \J \otimes X
       \ar[r]_-{u}
       &
       Z \pbox{,}
     }
  \]
  with $Z'$ a cofibrant \oo-category. But this square cannot admit a
  lift for such a lift would map $p \otimes f$ to an identity in $Z'$,
  contradicting the fact that $u(p \otimes f)$ is non-trivial. Therefore the
  map $X \amalg X \to \J \otimes X$ is not a cofibration.
\end{remark}

\begin{corollary}\label{coro:left_right_homot}
  Let $f, g : X \to Y$ be two \oo-functors, where $X$ is a cofibrant
  \oo-category. If there exists a reversible transformation from $f$
  to $g$, then there exists a $\J$-transformation from $f$ to $g$.
\end{corollary}

\begin{proof}
  As all the \oo-categories are fibrant for the folk model category
  structure and $X$ is cofibrant by hypothesis, the relations of left
  homotopy and right homotopy on $\Hom_{\ooCat}(X, Y)$ coincide. We
  get the result using the path object $\CylRev(Y)$ (see
  Theorem~\ref{thm:path_obj}) and the cylinder object $\J \otimes X$
  given by the previous theorem.
\end{proof}

\begin{corollary}\label{coro:rev_retract_J1}
  Any reversible transformation retract whose target is cofibrant is a
  $\J$-transformation retract.
\end{corollary}

\begin{proof}
  This follows immediately from the previous corollary.
\end{proof}

\begin{remark}

The following diagram sums up the relationship between the different classes
of \oo-functors considered in this section in between folk trivial
cofibrations and folk weak equivalences:
  
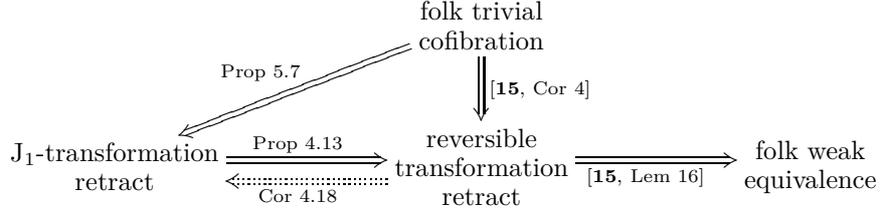
\[
    \xymatrix@C=5pc{
      & \txt{folk trivial \\ cofibration}
      \ar@2[d]^{\text{\cite[Cor 4]{Folk}}}
      \ar@2[ld]_{\text{Prop \ref{prop:triv_cof_retr}}\quad} & \\
      \txt{$\J$-transformation \\retract}
      \ar@2[r]^{\text{Prop \ref{prop:J1_immer_are_we}}}
      &
      \txt{reversible \\ transformation \\ retract}
      \ar@2[r]_{\text{\cite[Lem 16]{Folk}}}
      \ar@{:>}@<8pt>[l]^{\text{Cor \ref{coro:rev_retract_J1}}}
      &
      \txt{folk weak \\ equivalence} \pbox{.}
  }
  \]
  (The dotted arrow means that the implication holds under the additional
  assumption that the target is cofibrant.)
\end{remark}

\section{The folk model category structure is monoidal for the tensor
  product}
\label{sec:Gray_folk}

In this section, we will end the proof of the compatibility of the Gray
tensor product with the folk model category structure and give some
supplements on the resulting monoidal model category.

\medbreak

We start by a general lemma abstracting our strategy to prove this
compatibility:

\begin{lemma}\label{lemma:model_monoidal_criterion}
  Let $\M$ be a cofibrantly generated model category endowed with a (not
  necessarily symmetric nor closed) monoidal category structure satisfying
  the following hypotheses:
  \begin{enumerate}[label=H\arabic*)]
  \item \label{hyp:unit} the unit of the tensor product $\otimes$ is cofibrant,
  \item \label{hyp:target_cof} the sources of the generating cofibrations
  are cofibrant,
  \item \label{hyp:tens_I} if $i$ is a cofibration \resp{a trivial
  cofibration}, then so are $i \otimes \emptyset$ and $\emptyset \otimes i$,
  where $\emptyset$ denotes the initial object of $\M$,
  \item \label{hyp:cof} for every generating cofibrations $i : A \to B$ and
  $j : C \to D$, the pushout-product
    \[
      i \boxprod j : B \otimes C \amalg_{A \otimes C} A \otimes D \to
      B \otimes D
    \]
  is a cofibration,
  \item \label{hyp:trivial_cof} for every generating trivial cofibration $i$
  and every cofibrant object $A$, the morphisms $i \otimes A$ and $A \otimes
  i$ are weak equivalences.
  \end{enumerate}
  Then $\M$ is a monoidal model category.
\end{lemma}

\begin{proof}
  First note that by Remark~\ref{rem:M1}, the hypothesis \ref{hyp:unit} and
  \ref{hyp:tens_I} imply that the unit axiom (see
  paragraph~\ref{paragr:def_monoidal_model}) is satisfied.

  Moreover, by the classical Lemma \ref{lemma:monoidal_model_gen} (see also
  Remark~\ref{rem:monoidal_model_gen}), the pushout-product axiom
  (see again paragraph~\ref{paragr:def_monoidal_model}) can be checked on
  generators. The hypothesis \ref{hyp:cof} thus implies that the pushout-product
  of two cofibrations is a cofibration, and it suffices to show that if,
  either $i : A \to B$ is a generating trivial cofibration and $j : C \to D$
  is a generating cofibration, or $i$ is a generating cofibration and $j$ is
  a generating trivial cofibration, then $i \boxprod j$ is a weak
  equivalence. Let us prove the first case, the proof of the second one
  being dual.

  Consider the commutative diagram
  \[
    \xymatrix@C=1pc@R=1pc{ A \otimes C \ar[rr]^{A \otimes j}
      \ar[dd]_{i \otimes C} && A \otimes D \ar[dd]_{\epsilon_2}
      \ar@/^3ex/[dddr]^{i \otimes D}
      \\\\
      B \otimes C \ar[rr]^-{\epsilon_1} \ar@/_3ex/[rrrd]_{B \otimes j}
      && B \otimes C \amalg_{A \otimes C} A \otimes D \ar[dr]^(.65){i
        \boxprod
        j} \\
      &&& B \otimes D \pbox{,} }
  \]
  where $\epsilon_1$ and $\epsilon_2$ are the canonical morphisms. The
  pushout-product axiom for cofibrations and hypothesis~\ref{hyp:tens_I}
  imply that the tensor product of a cofibration and a cofibrant object is
  a cofibration (see Remark~\ref{rem:M1}). It thus follows from
  \ref{hyp:target_cof} and~\ref{hyp:trivial_cof} that $i
  \otimes C$ is a trivial cofibration. The morphism $\epsilon_2$ being a
  pushout of $i \otimes C$, it is a trivial
  cofibration as well. For the same reasons as above, the morphism $i
  \otimes D$ is a weak equivalence and the two-out-of-three property implies
  that $i \boxprod j$ is also a weak equivalence, thereby proving the lemma.
\end{proof}

\begin{remark}
  Note that under the hypothesis \ref{hyp:unit}, \ref{hyp:target_cof} and
  \ref{hyp:tens_I} of the lemma, the fact that $\M$ is a monoidal model
  category is actually equivalent to the hypothesis~\ref{hyp:cof} and
  \ref{hyp:trivial_cof}.
\end{remark}

\begin{remark}\label{rem:model_monoidal_criterion}
  The hypothesis \ref{hyp:unit} and \ref{hyp:target_cof} are fulfilled by
  the folk model category structure. Moreover, the hypothesis
  \ref{hyp:tens_I} is fulfilled in any biclosed monoidal structure (we will
  also apply this lemma to the join of \oo-categories which is not
  biclosed). To prove that the folk model category structure is monoidal
  for the Gray tensor product, it thus suffices to prove \ref{hyp:cof} and
  \ref{hyp:trivial_cof}. The hypothesis \ref{hyp:cof} is
  Theorem~\ref{thm:tens_comp_cof} and will now prove
  \ref{hyp:trivial_cof}.
\end{remark}

Let us see that \ref{hyp:trivial_cof} is a direct consequence of results
from the previous section.

\begin{proposition}\label{prop:i_otimes_Z}
  Let $i : X \to Y$ be a folk trivial cofibration between cofibrant
  objects.  Then, for any \oo-category $Z$, the \oo-functor
  $i \otimes Z : X \otimes Z \to Y \otimes Z$ is a $\J$-transformation
  retract and in particular a folk weak equivalence.
\end{proposition}

\begin{proof}
  As noted in paragraph~\ref{paragr:def_retr}, it is proved in
  \cite{Folk} that such an \oo-functor $i$ is a reversible
  transformation retract. As $Y$ is cofibrant,
  Corollary~\ref{coro:rev_retract_J1} implies that $i$ is a
  $\J$-transformation retract. But it is immediate that
  $\J$-transformation retracts are stable by tensoring by an object on the
  right, hence the result by Proposition~\ref{prop:J1_immer_are_we}.
\end{proof}

\begin{corollary}\label{coro:Z_otimes_i}
  Let $i : X \to Y$ be a folk trivial cofibration between cofibrant
  objects. Then, for any \oo-category $Z$, the \oo-functor
  $Z \otimes i : Z \otimes X \to Z \otimes Y$ is a folk weak equivalence.
\end{corollary}

\begin{proof}
  We will use the duality $T \mapsto T^\op$ introduced in
  paragraph~\ref{paragr:dualities}. By this same paragraph, this duality
  preserves cofibrations, trivial cofibrations and weak equivalences, and
  we have a natural isomorphism $(A \otimes B)^\op \simeq B^\op \otimes
  A^\op$. This implies that the \oo-functor $i^\op : X^\op \to Y^\op$ is a
  trivial cofibration between cofibrant objects and hence, by the previous
  proposition, that $i^\op \otimes Z^\op : X^\op \otimes Z^\op \to Y^\op
  \otimes Z^\op$ is a folk weak equivalence. This shows that $Z \otimes i$,
  that can be identified with $(i^\op \otimes Z^\op)^\op$, is indeed a folk
  weak equivalence.
\end{proof}

\begin{theorem}\label{thm:Gray_mon_mod}
  The folk model category structure on $\ooCat$ is monoidal for the Gray
  tensor product.
\end{theorem}

\begin{proof}
  This follows from Lemma~\ref{lemma:model_monoidal_criterion}, whose
  non-trivial hypothesis are fulfilled by Theorem~\ref{thm:tens_comp_cof},
  and the previous proposition and its corollary.
\end{proof}

The previous theorem implies that the tensor product of a folk trivial
cofibration and a cofibrant object is a weak equivalence. We will now
prove that this still holds if we remove the cofibrancy hypothesis.

\begin{proposition}\label{prop:triv_cof_retr}
  Any folk trivial cofibration is a strong $\J$-transformation
  retract.
\end{proposition}

\begin{proof}
  Let $i : X \to Y$ be a trivial cofibration. As every \oo-category is
  fibrant, the lifting problem
  \[
    \xymatrix{
      X \ar[d]_i \ar@{=}[r] & X \\
      Y \ar@{.>}[ur] }
  \]
  admits a solution $r : Y \to X$ giving a retraction of
  $i$. Similarly, by Theorem~\ref{thm:Gray_mon_mod}, the lifting
  problem
  \[
    \xymatrix@C=4pc@R=3pc{ \J \otimes X \amalg_{\bDn{1} \otimes X}
      \bDn{1} \otimes Y
      \ar[d]_{k_1 \boxprod i} \ar[r]^-{(ip, (ir, \id{Y}))} & Y \\
      \J \otimes Y \ar@{.>}[ur] & \pbox{,} }
  \]
  where $p$ denotes the ``projection'' \oo-functor
  $p : \J \otimes X \to X$, admits a solution
  $h : \J \otimes Y \to Y$. Such an $h$ is precisely a
  $\J$-transformation as in the definition of a strong
  $\J$-transformation retract.
\end{proof}

\begin{proposition}
  Strong $\J$-transformation retracts are stable under pushouts.
\end{proposition}

\begin{proof}
  The analogous statement for strong reversible transformation
  retracts is \cite[Lemma 17]{Folk}, whose proof applies
  \forlang{mutatis mutandis}.
\end{proof}

\goodbreak

\begin{proposition}\label{prop:tens_mon_ax}\ %
  \begin{enumerate}
    \item Transfinite compositions of pushouts of
      tensor products of an object (on the left) and a folk trivial
      cofibration are folk weak equivalences.

  \item Transfinite compositions of pushouts of tensor products of a
    folk trivial cofibration and an object (on the right) are folk
    weak equivalences.
  \end{enumerate}
\end{proposition}

\begin{proof}
  The second assertion can be deduced from the first one using the
  duality $X \mapsto X^\op$ as in the proof of
  Corollary~\ref{coro:Z_otimes_i}. As for the first one, by
  Proposition~\ref{prop:triv_cof_retr}, trivial cofibrations are
  strong $\J$-transformation retracts. But $\J$-transformation
  retracts are stable by tensoring by an object on the left,
  essentially by definition, and by pushouts by the previous
  proposition. As $\J$-transformation retracts are weak equivalences
  (by Proposition~\ref{prop:J_retr_we}), the result follows from the
  fact that folk weak equivalences are stable under transfinite
  compositions (see \cite[Lemma 4.12]{Folk}).
\end{proof}

\begin{remark}
  In particular, the tensor product of a folk trivial cofibration by
  an object (on the left or on the right) is a folk weak equivalence.
\end{remark}

\begin{remark}\label{rem:tens_mon_ax}
  We proved more precisely that the \oo-functors of the first assertion
  of the proposition are transfinite compositions of $\J$-transformation
  retracts.
\end{remark}

\section{The case of \pdfmn{m}{n}-categories}

In this section, we fix $m$ and $n$ such that $0 \le n \le m \le \omega$.

\begin{paragraph}\label{paragr:def_mnCat}
  Recall that an \ndef{\mn{m}{n}-category} is an \oo-category $X$ such that
  \begin{itemize}
    \item $X$ is an $m$-category, that is, every $k$-cell of $X$ with $k >
      m$ is an identity,
    \item every $k$-cell $x$ of $X$, for $k > n$, is \ndef{invertible},
      meaning that there exists a $k$-cell $y$ such that
      \[ y \comp_{k-1} x = \id{sx} \quadand x \comp_{k-1} y = \id{tx}. \]
  \end{itemize}
  We will denote by $\mnCat{m}{n}$ the full subcategory of $\ooCat$
  consisting of $(m, n)$\=/categories. Note that $(m, m)$-categories are
  nothing but $m$-categories, and $(m, 0)$\=/categories are $m$-groupoids,
  whose category will be denoted by $\nGpd{m}$.

  The category $\mnCat{m}{n}$ is a reflective subcategory of $\ooCat$. In
  other words, the inclusion functor $\mnCat{m}{n} \hookto \ooCat$
  admits a left adjoint $\reflect : \ooCat \to \mnCat{m}{n}$.
\end{paragraph}

The goal of this section is to prove, first, that the Gray tensor product of
\oo-categories induces, using the \oo-functor $r : \ooCat \to \mnCat{m}{n}$,
a monoidal category structure on $\mnCat{m}{n}$ and, second, that this
monoidal category structure is compatible with the folk model category
structure on~$\mnCat{m}{n}$.

\medskip

To prove the first point, we will use Day's reflection theorem:

\begin{proposition}[Day]\label{prop:Day}
  Let $\C$ be a biclosed monoidal category and let $\D \subset \C$ be a
  reflective subcategory of $\C$. Then the following conditions are
  equivalent:
  \begin{enumerate}
    \item\label{item:Day_ideal}
      for every object $X$ of $\C$ and every object $Y$ of $\D$, the objects
      $\Homir_\C(X, Y)$ and $\Homil_\C(X, Y)$ (see
      paragraph~\ref{paragr:def_biclosed}) are in~$\D$,
    \item for every objects $X$ and $Y$ of $\C$, the canonical morphism
      \[ r(X \otimes Y) \xto{\,\sim\,} r(r(X) \otimes r(Y)), \]
      where $r : \C \to \D$ denotes the left adjoint to the inclusion
      functor, is an isomorphism.
  \end{enumerate}
  Moreover, when these conditions are satisfied, the tensor product
    \[ X \otimes_\D Y = r(X \otimes Y) \]
  defines a biclosed monoidal category structure on $\D$, whose unit and
  internal $\Hom$ are those of $\C$.
\end{proposition}

\begin{proof}
  The analogous statement for closed symmetric monoidal categories is a
  particular case of \cite[Theorem 1.2]{DayRefl}. The proof applies
  \forlang{mutatis mutandis} to the case of biclosed monoidal categories.
\end{proof}

\goodbreak

We will prove that condition~\ref{item:Day_ideal} is satisfied in our case
of interest, that is, that if $X$ is an \oo-category and $Y$ is an $(m,
n)$-category, then both $\HomOpLax(X, Y)$ and $\HomLax(X, Y)$ are $(m,
n)$\nbd-categories. We start by some preliminaries on invertible cells.

\begin{paragraph}
  We will denote by $\In{n}$ the free-standing invertible $n$-cell in
  $\ooCat$. In other words, $\In{n}$ is the $n$-category obtained from
  $\Dn{n}$ by formally inverting the principal cell of $\Dn{n}$. We have
  a canonical \oo-functor $\Dn{n} \to \In{n}$. The image of the principal cell
  of $\Dn{n}$ by this \oo-functor is the \ndef{principal cell of $\In{n}$}.
  By definition, an $n$-cell $x$ of an \oo-category~$X$ is invertible if and
  only the corresponding \oo-functor $\funcell{x} : \Dn{n} \to X$ factor
  through $\In{n}$.
\end{paragraph}

\begin{proposition}\label{prop:prod_inv_is_inv}
  Let $x$ be an $m$-cell of an \oo-category $X$ and let $y$ be an $n$-cell
  of an \oo-category $Y$. If either $x$ or $y$ is invertible, then $x
  \otimes y$ is invertible in $X \otimes Y$.
\end{proposition}

\begin{proof}
  The proof is similar to the proof of the analogous fact for reversible
  cells (Proposition~\ref{prop:prod_Rm_Dn}). More precisely, one first proves
  the statement analogous to Lemma~\ref{lem:Gamma_rev_is_rev} using the same
  calculations as in its proof and one then proves the statements analogous
  to Proposition~\ref{prop:prod_rev_by_1cell},
  Corollary~\ref{coro:prod_rev_dual}, Lemma~\ref{lem:prod_Rk_Dn} and,
  finally, Proposition \ref{prop:prod_Rm_Dn}, by a straightforward
  adaptation consisting essentially in replacing the \oo-category $\Rn{n}$
  by $\In{n}$.
\end{proof}

\begin{paragraph}\label{paragr:def_inv_cyl}
  Let $Y$ be an \oo-category. We will say that an $n$-cylinder $c = (x, y,
  \alpha)$ (see paragraph~\ref{paragr:desc_cylinder}) is \ndef{invertible} if
  all the $\alpha_k^\epsilon$, for $0 \le k \le n$ and $\epsilon = \pm$, are
  invertible cells of $Y$. For the same reasons as for reversible cylinders,
  the graded subset $\CylInv(Y)$ of~$\Cyl(Y)$ consisting of invertible
  cylinders forms a sub-\oo-category.

  We will now prove that the \oo-functor
  \[ \HomLax(\In{1}, Y) \to \HomLax(\Dn{1}, Y) = \Cyl(Y), \]
  induced by the canonical \oo-functor $\Dn{1} \to \In{1}$, gives
  an isomorphism of \oo-categories between $\HomLax(\In{1}, Y)$ and
  $\CylInv(Y) \subset \Cyl(Y)$. (This will be achieved in
  Proposition~\ref{prop:Gamma_inv}.)
\end{paragraph}

\begin{paragraph}\label{paragr:def_vert_comp}
  Let $Y$ be an \oo-category. By adjunction, $n$-cells of $\HomLax(\In{1},
  Y)$ correspond to \oo-functors $\In{1} \to \HomOpLax(\Dn{n}, Y)$, that
  is, to invertible $1$-cells of $\HomOpLax(\Dn{n}, Y)$. Note that, again by
  adjunction, $1$-cells of $\HomOpLax(\Dn{n}, Y)$ corresponds to
  $n$-cylinders $c = (x, y, \alpha)$ in $Y$ (and that the source and target
  of such a $c$ correspond to the cells $x$ and $y$ respectively). This
  means that the composition of $1$-cells in
  $\HomOpLax(\Dn{n}, Y)$ defines a composition on $n$-cylinders in $Y$, that
  we will call the \ndef{vertical composition}, and that the $n$-cells of
  $\HomLax(\In{1}, Y)$ correspond to the $n$-cylinders in $Y$ invertible
  for the vertical composition. In particular, this shows that the
  \oo-functor $\HomLax(\In{1}, Y) \to \Cyl(Y)$ identifies $\HomLax(\In{1}, Y)$
  with a sub-\oo-category of $\Cyl(Y)$.

  The vertical composition of $n$-cylinders can be described in the following
  way. Let $c = (x, y, \alpha)$ and let $d = (y, z, \beta)$ be two $n$-cylinders in
  $Y$ (see paragraph~\ref{paragr:desc_cylinder}). We define by induction on~$k$
  such that $0 \le k \le n$ four $(k+1)$-cells $a_k^\epsilon$ and
  $b_k^\epsilon$, with $\epsilon = \pm$, as follows:
  \[
    \begin{split}
      a_k^\epsilon & = b_{k-1}^+ \comp_{k-1} \cdots \comp_1 b_0^+ \comp_0
      \alpha_k^\epsilon, \\
      b_k^\epsilon & = \beta_k^\epsilon \comp_0 a_0^- \comp_1 \cdots
      \comp_{k-1} a_{k-1}^-.
  \end{split}
  \]
  The vertical composition of $d$ and $c$, denoted by $d \compv c$,
  is then given by the triple $(x,z,\gamma)$, where
  \[
    \gamma_k^\epsilon = b^\epsilon_k \comp_k a^\epsilon_k.
  \]
  Note that the unit of an $n$-cell $x$ for the vertical composition is the
  $n$-cylinder $(x, x, \alpha)$, where $\alpha_k^- = \id{s_k(x)}$ and
  $\alpha_k^+ = \id{t_k(x)}$.
\end{paragraph}

\begin{proposition}
  Let $c = (x, y, \alpha)$ be an invertible $n$-cylinder in an \oo-category~$Y$. Then
  $c$ is invertible for the vertical composition with inverse $(y, x, \beta)$ given by:
  \[
  \beta_k^\epsilon = \bar \alpha_0^+ \comp_0 ( \bar \alpha_1^+
  \comp_1 \cdots \comp_{k-2}  (\bar \alpha_{k-1}^+ \comp_{k-1} \bar
  \alpha_{k}^\epsilon \comp_{k-1} \bar \alpha_{k-1}^-) \comp_{k-2} \cdots
  \comp_1 \bar \alpha_1^- ) \comp_0 \bar \alpha_0^-,
  \]
  for $0 \le k \le n$ and $\epsilon = \pm$,
  where $\bar \alpha_l^\epsilon$ denotes the inverse of $\alpha_l^\epsilon$ in
  $Y$.
\end{proposition}

\begin{proof}
  Let us first show that $(y, x, \beta)$ is an $n$-cylinder. Let $1 \le k
  \le n$. We have to show that
  \[ t(\beta^\e_k) = x_k^\epsilon \comp_0 \beta_0^- \comp_1 \cdots
  \comp_{k-1} \beta_{k-1}^-. \]
  For $i \leq k \leq n$, we set
  \[
    \beta^\e_{k, i} =
    \bar \alpha_i^+ \comp_i ( \bar \alpha_{i+1}^+ \comp_{i+1} \cdots
    \comp_{k-2}  (\bar \alpha_{k-1}^+ \comp_{k-1} \bar \alpha_{k}^\epsilon
    \comp_{k-1} \bar \alpha_{k-1}^-) \comp_{k-2} \cdots \comp_{i+1} \bar
  \alpha_{i+1}^- ) \comp_i \bar \alpha_i^- \]
  and
  \[
    u_i = (\alpha_{i-1}^+ \comp_{i-1} \cdots \comp_1 \alpha_0^+ \comp_0
  x_k^\epsilon) \comp_i \beta_{i,i}^- \comp_{i+1} \cdots
    \comp_{k-1} \beta_{k-1,i}^-.
  \]
  In particular, we have
  \[ \beta^\e_{k, 0} = \beta^\e_k, \quad
     \bar\alpha_{i-1}^+ \comp_i \beta^\e_{k,i} \comp_i \bar \alpha_{i-1}^-
       = \beta^\e_{k, i-1}
     \quadand
     \beta_{k,k}^\e = \bar\alpha_k^\e,
  \]
  and
  \[
    \begin{split}
      u_0 & = x_k^\epsilon \comp_0 \beta_0^- \comp_1 \cdots \comp_{k-1}
        \beta_{k-1}^-, \\
      u_k &= \alpha_{k-1}^+ \comp_{k-1} \cdots \comp_1 \alpha_0^+ \comp_0 x_k^\epsilon =
    s(\alpha_k^\epsilon).
    \end{split}
  \]
  We have to show the equality $t(\beta_k^\epsilon) = u_0$. More generally,
  we will prove by descending induction on $i$ such that $0 \leq i \leq
  k$ that we have $t(\beta_{k,i}^\epsilon) = u_i$.
  For $i = k$, we have
  \[ t(\beta_{k, k}^\e) = t(\bar\alpha^\e_k) = s(\alpha^\e_k) = u_k. \]
  For $0 \le i < k$, using the induction hypothesis, we have 
  \[
    \begin{split}
      t(\beta_{k, i}^\e)
      & = t(\alpha^+_i \comp_i \beta^\e_{k, i+1} \comp_i \alpha^-_i) \\
      & = \alpha^+_{i} \comp_i u_{i+1} \comp_i \alpha^-_{i} \\
      & =
      \bar \alpha_{i}^+ \\
      & \quad\quad \comp_{i}
      \big((\alpha_{i}^+ \comp_{i} \cdots \comp_1
        \alpha_0^+ \comp_0 x_k^\epsilon) \comp_{i+1} \beta_{i+1,i+1}^- \comp_{i+2}
      \cdots \comp_{k-1} \beta_{k-1, i+1}^-\big) \\
      & \quad\quad \comp_{i} \bar \alpha^-_{i} \\
      & =
      \big(\bar \alpha_{i}^+ \comp_{i} (\alpha_{i}^+ \comp_{i}
        \cdots \comp_1 \alpha_0^+ \comp_0 x_k^\epsilon) \comp_{i} \bar
        \alpha_{i}^-\big) \\
        & \quad\quad \comp_{i+1} (\bar\alpha_i^+ \comp_i \beta_{i+1,i+1}^-
        \comp_i \bar\alpha_i^-) \comp_{i+2} \cdots \comp_{k-1}
        (\bar\alpha^+_i \comp_i \beta_{k-1,i+1}^- \bar\alpha^-_i) \\
      & =
      \big(\bar \alpha_{i}^+ \comp_{i} (\alpha_{i}^+ \comp_{i}
        \cdots \comp_1 \alpha_0^+ \comp_0 x_k^\epsilon) \comp_{i} \bar
        \alpha_{i}^-\big) \\
        & \quad\quad \comp_{i+1}  \beta_{i+1,i}^- \comp_{i+2} \cdots \comp_{k-1}
        \beta_{k-1,i}^- \\
      & =
      \big(\alpha_{i-1}^+ \comp_{i-1}
        \cdots \comp_1 \alpha_0^+ \comp_0 x_k^\epsilon\big) \comp_{i}
      \beta_{i,i}^- \\
        & \quad\quad \comp_{i+1}  \beta_{i+1,i}^- \comp_{i+2} \cdots \comp_{k-1}
        \beta_{k-1,i}^- \\
     & = u_i.
    \end{split}
  \]
  A similar calculation shows that
  \[
    s(\beta^\e_k) = \beta_{k-1}^+ \comp_{k-1} \cdots \comp_1 \beta_0^+
    \comp_0 y_k^\epsilon.
  \]

  Let us now prove that $(y, x, \beta)$ is an inverse of $(x, y, \alpha)$
  for the vertical composition. Consider $(x, x, \gamma) = (y, x, \beta)
  \compv (x, y, \alpha)$. We have to prove that $\gamma_k^\e$ is an identity
  for every $0 \le k \le n$ and $\e = \pm$. Recall that by definition (see
  paragraph~\ref{paragr:def_vert_comp}), we have
  \[
      \gamma_k^\epsilon = b_k^\epsilon \comp_k a_k^\epsilon,
  \]
  where
  \[
    \begin{split}
      a_k^\epsilon & = b_{k-1}^+ \comp_{k-1} \cdots \comp_1 b_0^+ \comp_0
      \alpha_k^\epsilon, \\
      b_k^\epsilon & = \beta_k^\epsilon \comp_0 a_0^- \comp_1 \cdots
      \comp_{k-1} a_{k-1}^-.
  \end{split}
  \]
  We will start by proving, by induction on $k$, that
  \[
    \begin{split}
      a_k^\epsilon & = \bar \alpha_0^+ \comp_0 \bar \alpha_1^+ \comp_1
      \cdots \comp_{k-2} \bar \alpha_{k-1}^+ \comp_{k-1}
      \alpha_k^\epsilon, \\
      b_k^\epsilon & = \bar \alpha_0^+ \comp_0 \bar \alpha_1^+ \comp_1
      \cdots \comp_{k-2} \bar \alpha_{k-1}^+ \comp_{k-1} \bar
      \alpha_k^\epsilon.
    \end{split}
  \]
  For $k = 0$, the formulas boil down to the equalities $a_0^\epsilon =
  \alpha_0^\epsilon$ and $b_0^\epsilon = \bar \alpha_0^\epsilon$ which hold
  by definition. Suppose $k > 0$. For $0 \le i \leq j \leq n$, we set
  \[
    \begin{split}
      a_{j,i}^\epsilon & = \bar \alpha_i^+ \comp_i \bar \alpha_{i+1}^+
    \comp_{i+1} \cdots \comp_{j-2} \bar \alpha_{j-1}^+ \comp_{j-1}
    \alpha_j^\epsilon, \\
      b_{j,i}^\epsilon & = \bar \alpha_i^+ \comp_i \bar \alpha_{i+1}^+
    \comp_{i+1} \cdots \comp_{j-2} \bar \alpha_{j-1}^+ \comp_{j-1} \bar
    \alpha_j^\epsilon.
    \end{split}
  \]
  In particular, we have
  \[ a^\e_{j,j} = \bar \alpha_j^\epsilon \quadand b^\e_{j,j} = \bar
  \alpha_j^\epsilon. \]
  By induction hypothesis, for $i \leq j < k$, we have
  \[ b_j^\epsilon = b_{j,0}^\epsilon. \]
  We thus have
  \[
    \begin{split}
      a_k^\epsilon
      & = b_{k-1}^+ \comp_{k-1} \cdots \comp_1 b_0^+ \comp_0 \alpha_k^\epsilon \\
      & = b_{k-1,0}^+ \comp_{k-1} \cdots \comp_1 b_{0,0}^+ \comp_0 \alpha_k^\epsilon \\
      & =  (\bar \alpha_0^+ \comp_0 b_{k-1,1}^+) \comp_{k-1} \cdots \comp_2
      (\bar \alpha_0^+ \comp_0 b_{1,1}^+) \comp_1 \bar \alpha_0^+ \comp_0
      \alpha_k^\epsilon \\
      & = \bar \alpha_0^+ \comp_0 (b_{k-1,1}^+ \comp_{k-1} \cdots \comp_2
      b_{1,1}^+ \comp_1 \alpha_k^\epsilon) \\
      & = \cdots \\
      & = \bar \alpha_0^+ \comp_0 \bar \alpha_1^+ \comp_1 \cdots
      \comp_{i-2} \bar \alpha_{i-1}^+ \comp_{i-1} (b_{k-1,i}^+ \comp_{k-1}
      \cdots \comp_{i+1} b_{i,i}^+ \comp_i \alpha_k^\epsilon) \\
      & = \cdots \\
      & = \bar \alpha_0^+ \comp_0  \bar \alpha_1^+ \comp_1 \cdots
      \comp_{k-2} \bar \alpha_{k-1}^+ \comp_{k-1} \alpha_k^\epsilon
    \end{split}
  \]
  and
  \[
    \begin{split}
      b_k^\epsilon
      & = \beta_k^\epsilon \comp_0 a_0^- \comp_1 \cdots \comp_{k-1} a_{k-1}^-  \\
      & = \beta_{k,0}^\epsilon \comp_0 a_{0,0}^- \comp_1 \cdots \comp_{k-1}
      a_{k-1,0}^- \\
      & = \bar\alpha_0^+ \comp_0 \beta_{k,1}^\epsilon \comp_0 \bar \alpha_0^-
        \comp_0 \alpha_0^- \comp_1 (\bar \alpha_0^+ \comp_0 a_{1,1}^-) \comp_2
        \cdots \comp_{k-1} (\bar \alpha_0^+ \comp_0 a_{k-1,1}^-) \\
      & = \bar \alpha_0^+ \comp_0 (\beta_{k,1}^\epsilon \comp_1  a_{1,1}^-
        \comp_2 \cdots \comp_{k-1}  a_{k-1,1}^-) \\
      & = \cdots \\
      & = \bar \alpha_0^+ \comp_0 \bar \alpha_1^+ \comp_1 \cdots
        \comp_{i-2} \bar \alpha_{i-1}^+ \comp_{i-1} (\beta_{k, i}^\epsilon
        \comp_i  a_{i,i}^- \comp_{i+1} \cdots \comp_{k-1}  a_{k-1,i}^-) \\
      & = \cdots \\
      & = \bar \alpha_0^+ \comp_0 \bar \alpha_1^+ \comp_1 \cdots
        \comp_{k-2} \bar \alpha_{k-1}^+ \comp_{k-1} \beta^\e_{k, k} \\
      & = \bar \alpha_0^+ \comp_0 \bar \alpha_1^+ \comp_1 \cdots
        \comp_{k-2} \bar \alpha_{k-1}^+ \comp_{k-1} \bar
        \alpha_k^\epsilon,
    \end{split}
  \]
  which ends the proof of the announced formulas. Finally, we get
  that
  \[
    \begin{split}
      \gamma_k^\epsilon
      & = b_k^\epsilon \comp_k a_k^\epsilon \\
      & = \big(\bar \alpha_0^+ \comp_0 \bar \alpha_1^+ \comp_1 \cdots
      \comp_{k-2} \bar \alpha_{k-1}^+ \comp_{k-1} \bar
      \alpha_k^\epsilon\big) \\
      & \quad\quad
      \comp_k \big(\bar \alpha_0^+ \comp_0 \bar \alpha_1^+ \comp_1 \cdots
      \comp_{k-2} \bar \alpha_{k-1}^+ \comp_{k-1} \alpha_k^\epsilon\big) \\
      & = \bar \alpha_0^+ \comp_0 \bar \alpha_1^+ \comp_1 \cdots
      \comp_{k-2} \bar \alpha_{k-1}^+ \comp_{k-1} (\bar \alpha_k^\epsilon
      \comp_k \alpha_k^\epsilon) \\
      & = \id{\bar \alpha_0^+ \comp_0 \bar \alpha_1^+ \comp_1 \cdots
      \comp_{k-2} \bar \alpha_{k-1}^+}.
    \end{split}
  \]
  This proves that $(y, x, \beta) \compv (x, y, \alpha)$ is indeed the
  identity. Similar calculations show that $(x, y, \alpha) \compv (y, x,
  \beta)$ is the identity as well, thereby ending the proof.
\end{proof}

\begin{proposition}\label{prop:Gamma_inv}
  Let $Y$ be an \oo-category. The \oo-functor
  \[ \HomLax(\In{1}, Y) \to \HomLax(\Dn{1}, Y) = \Cyl(Y), \]
  induces an isomorphism between $\HomLax(\In{1}, Y)$ and
  $\CylInv(Y) \subset \Cyl(Y)$.
\end{proposition}

\begin{proof}
  We saw in paragraph~\ref{paragr:def_vert_comp} that this \oo-functor is
  injective. It thus suffices to prove that its image is precisely
  $\CylInv(Y)$.

  Let us first prove that the \oo-functor lands into $\CylInv(Y)$. Consider an
  $n$-cell of $\HomLax(\In{1}, Y)$, seen as an \oo-functor $c : \In{1}
  \otimes \Dn{n} \to Y$. Let $(x, y, \alpha)$ be the associated cylinder. By
  definition, we have
  \[
    \alpha_k^- = c((01) \otimes s_k(d)) \quadand \alpha_k^+ = c((01)
    \otimes t_k(d)),
  \]
  for $0 \le k \le n$, where $(01)$ and $d$ denote the principal cells of $\In{1}$
  and $\Dn{n}$, respectively. As $(01)$ is invertible in $\In{1}$, so is
  its tensor product with any cell by
  Proposition~\ref{prop:prod_inv_is_inv}, and the $\alpha_k^\epsilon$
  are thus invertible. This proves that $(x, y, \alpha)$ is an invertible
  cylinder.

  Reciprocally, if $c$ is an invertible $n$-cylinder, then, by the previous
  proposition, the cylinder $c$ is invertible for the vertical composition,
  and hence corresponds to an $n$-cell of $\HomLax(\In{1}, Y)$ (see
  paragraph~\ref{paragr:def_vert_comp}), thereby proving the result.
\end{proof}

\begin{remark}
  In particular, an $n$-cylinder is invertible in the sense of
  paragraph~\ref{paragr:def_inv_cyl} if and only if it invertible for the
  vertical composition introduced in paragraph~\ref{paragr:def_vert_comp}.
\end{remark}

\begin{proposition}\label{prop:invertible_transfo}
  An $n$-cell of $\HomOpLax(X, Y)$, seen by adjunction as an \oo-functor $H
  : \Dn{n} \otimes X \to Y$, is invertible if and only if, for every
  $m$-cell $x$ of $X$, the $(n+m)$-cell $H(d \otimes x)$, where $d$ denotes
  the principal cell of $\Dn{n}$, is invertible in $Y$.
\end{proposition}

\begin{proof}
  Suppose that $H : \Dn{n} \otimes X \to Y$ is invertible as an $n$-cell of
  $\HomOpLax(X, Y)$. By universal property of $\In{n}$ and by adjunction,
  this means that $H$ factors through $\In{n} \otimes X$, so that $H(d
  \otimes x) = H'(d' \otimes x)$, where $H' : \In{n} \otimes X
  \to Y$ and $d'$ is the principal cell of $\In{n}$. As $d'$ is invertible
  in $\In{n}$, so is $d' \otimes x$ by
  Proposition~\ref{prop:prod_inv_is_inv}. This implies that $H(d \otimes x)
  = H'(d' \otimes x)$ is invertible, showing one implication.

  Let us show the converse. We will argue by induction on $n \ge 1$.
  Suppose $n = 1$. The hypothesis implies that the associated \oo-functor
  $k : X \to \Cyl(Y)$ factors through $\CylInv(Y) \subset \Cyl(Y)$. The
  result thus follows from the bijections
  \[
    \begin{split}
      \Hom_{\ooCat}(X, \CylInv(Y)) 
      & \simeq \Hom_{\ooCat}(X, \HomLax(\In{1}, Y)) \\
      & \simeq
      \Hom_{\ooCat}(\In{1}, \HomOpLax(X, Y)),
    \end{split}
  \]
  the first bijection being a consequence of the previous proposition.

  Suppose now that $n > 1$. Denote by $s : \Dn{n-1} \otimes \Dn{1} \to
  \Dn{n}$ the \oo-functor of paragraph~\ref{paragr:prod_DnD1_to_Dn} and
  consider the \oo-functor
  \[ H(s \otimes X) : \Dn{n-1} \otimes \Dn{1} \otimes X \to Y. \]
  Denote by $e$ the principal cell of $\Dn{n-1}$. We will prove that,
  for every cell $z$ of $\Dn{1} \otimes X$, the cell $H(s \otimes X)(e
  \otimes z)$ is invertible in $Y$. Since the \oo-category $\Dn{1} \otimes
  X$ is generated by cells of the form $0 \otimes x$, $1 \otimes x$ and
  $(01) \otimes x$, where $(01)$ denotes the principal cell of $\Dn{1}$ and
  $x$ is a cell of $X$, it suffices to show that $H(s \otimes X)(e \otimes
  z)$, where $z$ is one of these generators, is an
  invertible cell. Since $s(e \otimes \epsilon) \otimes x$, for $\epsilon =
  0, 1$, is an identity by definition of $s$ and the fact that $e$ is of
  dimension at least $1$, the cell $H(s \otimes X)(e \otimes \epsilon
  \otimes x)$ is invertible. Furthermore, the cell
  \[
    H(s \otimes X)(e \otimes (01) \otimes x) = H(s(e \otimes (01)) \otimes
    x) = H(d \otimes x)
  \]
  is invertible by hypothesis on $H$. By induction hypothesis, this implies
  that $H(s \otimes X)$ is invertible as an $(n-1)$-cell of
  $\HomOpLax(\Dn{1} \otimes X, Y)$. This means that
  $H(s \otimes X)$ factors trough $\In{n-1} \otimes (\Dn{1} \otimes X)$, so
  that we get an \oo-functor $H' : \In{n-1} \otimes (\Dn{1} \otimes X) \to Y$.
  Denote by $d'$ and $e'$ the principal cells of $\In{n}$ and $\In{n-1}$,
  respectively. Using the fact that, by
  Proposition~\ref{prop:prod_inv_is_inv}, the cell $e' \otimes (01)$ is
  invertible in $\In{n-1} \otimes \Dn{1}$, we get a commutative
  diagram
  \[
    \xymatrix@C=4.5pc{
      \Dn{n} \otimes X \ar[d]_{\funcell{d'} \otimes X}
        \ar[r]^-{\funcell{e \otimes (01)} \otimes X} &
      \Dn{n-1} \otimes \Dn{1} \otimes X \ar[r]^-{s \otimes X}
        \ar[d]_-{\funcell{e \otimes (01)} \otimes \Dn{1} \otimes X} &
      \Dn{n} \otimes X \ar[r] \ar[r]^-H &
      Y \\
      \In{n} \otimes X \ar[r]^-{\funcell{e' \otimes (01)} \otimes X} &
      \In{n-1} \otimes \Dn{1} \otimes X \ar[urr]_{H'}
      & & \pbox{.}
    }
  \]
  Since $s\funcell{e \otimes (01)} = \id{\Dn{n}}$, the composite of the
  three composable horizontal arrows of the diagram is $H : \Dn{n} \otimes X
  \to Y$. This implies that $H$ factors through $\In{n} \otimes X$ and hence
  that it is invertible as a cell of $\HomOpLax(X, Y)$.
\end{proof}

\begin{remark}
  In particular, an oplax transformation is invertible as a $1$-cell of
  $\HomOpLax(X, Y)$ if and only it its components are invertible in $Y$.
\end{remark}

\begin{proposition}\label{prop:Day_mn}
  If $X$ is an \oo-category and $Y$ is an $(m, n)$-category, then both
  $\HomOpLax(X, Y)$ and $\HomLax(X, Y)$ are $(m, n)$-categories.
\end{proposition}

\begin{proof}
  Since
  \[ \HomLax(X,Y) \simeq \HomOpLax(X^\op, Y^\op)^\op \]
  (see paragraph~\ref{paragr:dualities}) and $(m, n)$-categories are stable
  under the duality $Z \mapsto Z^\op$, it suffices to prove that
  $\HomOpLax(X, Y)$ is an $(m, n)$-category.

  The fact that $\HomOpLax(X, Y)$ is an $m$-category is already known (see
  for instance \cite[Proposition A.29]{AraMaltsiJoint}). Let us prove that
  this $m$-category is an $(m, n)$-category. Consider a $k$-cell of
  $\HomOpLax(X, Y)$, with $k > n$, seen as an \oo-functor $H : \Dn{k}
  \otimes X \to Y$. Let $d$ be the principal cell of $\Dn{k}$ and let $x$ be
  a cell of $X$. The cell $H(x \otimes d)$ is a cell of $Y$ of dimension at
  least $k$. It is therefore invertible as $Y$ is an $(m, n)$-category.
  Proposition~\ref{prop:invertible_transfo} thus shows that $H$
  is invertible in $\HomOpLax(X,Y)$.
\end{proof}

\begin{theorem}\label{thm:Day_mn}
  The \ndef{Gray tensor product of \mn{m}{n}-categories}
  \[ X \rotimes{m}{n} Y = \reflect (X \otimes Y), \]
  where $r : \ooCat \to \mnCat{m}{n}$ denotes the left adjoint to
  the inclusion functor $\mnCat{m}{n} \hookto \ooCat$,
  defines a monoidal category structure on $\mnCat{m}{n}$, whose unit is the
  \mn{m}{n}-category~$\Dn{0}$ and whose internal $\Hom$ are $\HomOpLax$ and
  $\HomLax$.
\end{theorem}

\begin{proof}
  This follows from the previous proposition by Day's reflection theorem
  (Proposition~\ref{prop:Day}).
\end{proof}

\begin{remark}
  The case of $m$-categories was already proved in \cite[Appendix
  A]{AraMaltsiJoint}.
\end{remark}

We now focus on the case of $m$-groupoids (that is, the case $n = 0$).

\begin{proposition}
  If $X$ and $Y$ are two \oo-groupoids, then $X \otimes Y$ is an
  \oo-groupoid, so that
  \[ X \otimes_{\omega, 0} Y = X \otimes Y. \]
\end{proposition}

\begin{proof}
  The \oo-category $X \otimes Y$ is generated by $n$-cells of
  the form $x \otimes y$, where $x$ a $k$-cell of $X$ and $Y$ an $l$-cell of
  $Y$ with $n = k + l$. If $n > 0$, then  $k > 0$ or $l>0$, so that $x$ or $y$
  is invertible. It follows from Proposition~\ref{prop:prod_inv_is_inv} that
  $x \otimes y$ is invertible, thereby proving the result.
\end{proof}

\begin{remark}
  Note that it is not true that the tensor product of an \oo-groupoid~$X$
  and an \oo-category~$Y$ is an \oo-groupoid in general, as the tensor
  product of a $0$-cell of~$X$ and a non-invertible $n$-cell of $Y$ is not
  invertible in $X \otimes Y$.
\end{remark}

\begin{lemma}
The functor $X \mapsto X^\op$ is naturally isomorphic to the identity when
restricted to the category $\nGpd{m}$ of $m$-groupoids.
\end{lemma}

\begin{proof}
  Let $X$ be an $m$-groupoid. Recall that for any $n$-cell $x$ of $X$ and
any $i < n$, the cell $x$ is invertible for the composition $\comp_i$ (see
for instance \cite[Proposition 1.3]{AraMetGpd}). We will denote this inverse
by $w_i(x)$. Note that for any $i, j < n$, we have $w_i(w_j(x)) =
w_j(w_i(x))$.
We define an \oo-functor $\delta_X : X^\op \to X$ by setting
$\delta_X(x) = w_1(w_3(\cdots w_k(x)\cdots))$,
for any
$n$-cell $x$ of $X$, where $k$ is the largest odd integer strictly smaller
than $n$. The fact that this defines
an \oo-functor natural in $X$ follows from a straightforward calculation.
Clearly, the \oo-functor $\delta_X$ is an isomorphism with inverse
$\delta_{X^\op}$, thereby proving the result.
\end{proof}

\begin{proposition}
  The Gray tensor product on $\nGpd{m}$ is symmetric.
\end{proposition}

\begin{proof}
  If $X$ and $Y$ are two $m$-groupoids, then, using the natural isomorphism
  of the previous lemma, we get an isomorphism
  \[
    X \rotimes{m}{0} Y = X \otimes Y \simeq (X \otimes
    Y)^\op \simeq Y^\op \otimes X^\op \simeq Y \otimes
    X = Y \rotimes{m}{0} X.
  \]
  One checks that this isomorphism defines a symmetry for the Gray tensor
  product.
\end{proof}

Let us come back to the general case $0 \le n \le m \le \omega$. We will now
show that the tensor product of $(m, n)$-categories is compatible with the
so-called \ndef{folk model category structure on $\mnCat{m}{n}$} that we now
recall:

\begin{theorem}[Lafont--Métayer--Worytkiewicz,
  Ara--Métayer]\label{thm:folk_mn}
  \noindent
  The folk model category structure on $\ooCat$ can be transferred along the
  adjunction
  \[ r : \ooCat \to \mnCat{m}{n}, \qquad \mnCat{m}{n} \hookto \ooCat. \]
  In particular, we get a model category structure on $\mnCat{m}{n}$, whose
  weak equivalences are the folk equivalences between $(m, n)$\nbd-categories
  and which is cofibrantly generated by $r(I)$ and $r(J)$ (see
  paragraphs~\ref{paragr:def_cof} and~\ref{paragr:def_J}).
\end{theorem}

\begin{proof}
  The case of $m$-categories is \cite[Theorem 5]{Folk} and the case of
  \oon{n}-categories is \cite[Theorem 3.19 and Remark 3.20]{AraMetGpd}.
  Combining these two proofs, one easily gets the general case.
\end{proof}

The compatibility between the tensor product of $(m, n)$-categories and the
folk model category structure on $\mnCat{m}{n}$ will follow formally from
the following general statement:

{\renewcommand\N{\mathcal{N}}
\begin{proposition}\label{prop:Day_model}
  Let $\M$ be a biclosed monoidal model category, which is cofibrantly
  generated by sets $I$ and $J$, and whose unit for the tensor product is
  cofibrant, and let $\N \subset \M$ be a reflective subcategory of $\M$.
  Denote by $r : \M \to \N$ the left adjoint to the
  inclusion functor. Suppose that
  \begin{enumerate}
    \item $\N \subset \M$ satisfies the equivalent conditions of Day's
      reflection theorem (Proposition~\ref{prop:Day}), so that
    \[ X \otimes_\N Y = r(X \otimes Y) \]
    defines a biclosed monoidal category structure on $\N$,
    \item $\N$ is endowed with a model category structure
      cofibrantly generated by $r(I)$ and~$r(J)$.
  \end{enumerate}
  Then $\N$ endowed with the tensor product $\otimes_\N$ is a monoidal model
  category.
\end{proposition}

\begin{proof}
  The hypothesis implies that $r$ is a left Quillen functor. In particular,
  the unit of $\otimes_\N$ is cofibrant. By
  Lemma~\ref{lemma:monoidal_model_gen}, it suffices to show that if $i$ is
  in $I$ and $j$ is in $I$, then $r(i) \boxprod_\N r(j)$ is a cofibration of
  $\N$, and that if either $i$ is in $I$ and $j$ is in $J$, or $i$ is
  in $J$ and $j$ is in $I$, then $r(i) \boxprod_\N r(j)$ is a trivial
  cofibration. Using the definition of $\otimes_\N$, the natural isomorphism
  $r(X \otimes Y) \simeq r(r(X) \otimes r(Y))$ and the fact that $r$
  preserves pushouts, we get that $r(i) \boxprod_\N r(j)$ can be identified
  with $r(i \boxprod j)$. The result thus follows from the pushout-product
  axiom in $\M$ and the fact that $r$ is a left Quillen functor.
\end{proof}
}

\begin{theorem}\label{thm:mnCat_mod_mon}
  The folk model category structure on $\mnCat{m}{n}$ is monoidal for the
  Gray tensor product of \mn{m}{n}-categories.
\end{theorem}

\begin{proof}
  This follows from the previous proposition, whose hypothesis are fulfilled
  by Proposition~\ref{prop:Day_mn} and Theorem~\ref{thm:folk_mn}.
\end{proof}

\begin{remark}
  In \cite{LackFolk2}, slightly corrected by \cite{LackFolkBi}, Lack proves
  that the folk model category structure on $\nCat{2}$ is monoidal for the
  \emph{pseudo} Gray tensor product. This is different from the result we
  get from the previous theorem in the case $m = 2$ and $n = 2$, which deals
  with the \emph{oplax} Gray tensor product.
\end{remark}

\begin{proposition}
  The folk model category structure on $\mnCat{m}{n}$ satisfies the two
  following properties:
  \begin{enumerate}
  \item Transfinite compositions of pushouts of tensor products of an
    object (on the left) and a folk trivial cofibration are folk weak
    equivalences.

  \item Transfinite compositions of pushouts of tensor products of a
    folk trivial cofibration and an object (on the right) are folk
    weak equivalences.
  \end{enumerate}
\end{proposition}

\begin{proof}
  The second assertion can be deduced from the first one using the duality
  $X \mapsto X^\op$. As for the first one, using the fact that the tensor
  product of $(m, n)$\nbd-categories is biclosed and hence commutes with
  colimits in each variable, it suffices to consider transfinite
  compositions of pushouts of tensor products of an object and
  an element of~$r(J)$. As the functor~$r$
  commutes with colimits, such a transfinite composition is of the form
  $r(f)$, where $f$ is a transfinite composition of pushouts of tensor
  products of an object and an element of~$J$. By
  Remark~\ref{rem:tens_mon_ax}, such an $f$ is a
  transfinite composition of $\J$-transformation retracts. As folk weak
  equivalences are stable under transfinite compositions, it suffices
  to show that $r$ sends $\J$-transformation retracts to weak equivalences.
  But if $h : \J \otimes X \to Y$ is a $\J$\=/transformation from an
  \oo-functor $u : X \to Y$ to an \oo-functor $v : X \to Y$, then by
  precomposing $r(h) : r(\J \otimes X) \to r(Y)$ by the natural \oo-functor
  \[ \J \otimes r(X) \to r(\J) \otimes r(X) \to r(r(\J) \otimes r(X)) \simeq
  r(\J \otimes X), \]
  one gets an \oo-functor $\J \otimes r(X) \to r(Y)$ defining a
  $\J$-transformation from $r(u)$ to~$r(v)$. This proves that $r$ sends
  $\J$-transformation retracts to $\J$-transformation retracts, hence the
  result by Proposition~\ref{prop:J1_immer_are_we}.
\end{proof}

\begin{remark}
  In the case $n = 0$, in which the tensor product is symmetric, the
  previous proposition asserts that the so-called \ndef{monoid axiom} of Schwede
  and Shipley~\cite{SchwedeShipley} holds in $\nGpd{m}$.
\end{remark}

\begin{corollary}\ %
  Let us endow $\nGpd{m}$ with the Gray tensor product.
\begin{enumerate}
  \item If $\mathcal P$ is a non-symmetric $(\nGpd{m})$-operad, then the
    category of $\mathcal P$\=/algebras in $m$\=/groupoids is endowed with a
    right proper combinatorial model category structure whose weak
    equivalences \resp{fibrations} are the morphisms whose
    underlying $m$-functor is a folk weak equivalence \resp{a folk
    fibration} of $m$-groupoids.
  \item If $\mathcal P$ is a non-symmetric $(\nGpd{m})$-operad such that
    $\mathcal P(n)$ is a cofibrant $m$\=/groupoid for every $n \ge 0$, and
    if $A$ is a cofibrant $\mathcal P$-algebra in $\nGpd{m}$, then the
    underlying $m$-groupoid of $A$ is cofibrant.
\end{enumerate}
\end{corollary}

\begin{proof}
\renewcommand\P{\mathcal{P}}
\newcommand\Pt{\widetilde{\mathcal{P}}}
Using the previous remark, the first point is \cite[Theorem 1.2]{Harper} (or
also \cite[Theorem 1.3]{Muro}, taking $\mathcal V = \mathcal C = \nGpd{m}$).

Let us prove the second point using results and terminology from
\cite{BergerMoerdijk}. Let $\Pt$ be the free symmetric operad on the
non-symmetric operad $\P$. We have $\Pt(n) = \Sigma_n \times \P(n)$, where
$\Sigma_n$ denotes the symmetric group, the action of $\Sigma_n$ on $\Pt(n)$
being the obvious one. Since $\P(n)$ is cofibrant by hypothesis, this
implies that, in the terminology of~\cite{BergerMoerdijk}, the operad $\Pt$
is $\Sigma$-cofibrant. Moreover, the category of $\Pt$-algebras is
isomorphic to the category of $\P$-algebras via a functor constant on
the underlying object and, by the first point, this category is thus
endowed with a model category structure compatible with the forgetful
functor to $\nGpd{m}$. In the terminology of \cite{BergerMoerdijk}, this
means that $\Pt$ is an admissible operad. The result thus follows from
\cite[Corollary 5.5]{BergerMoerdijk}.
\end{proof}

\begin{remark}
  In particular, the first point applied to the operad of monoids (seen as
  an $(\nGpd{m})$-operad by using the inclusion functor of sets into
  $m$\nbd-groupoids) gives a model category structure on the category of monoids
  in the category of $m$\nbd-groupoids endowed with the Gray tensor product.
\end{remark}

\begin{remark}
  The monoid axiom implies many other interesting properties of the homotopy
  theory of operads and their algebras. Another important setting to obtain
  these kinds of results has been introduced by Berger and Moerdijk in
  \cite{BergerMoerdijk}. One easily checks that the folk model category structure on
  $\nGpd{m}$, equipped with the Gray tensor product,
  satisfies the hypothesis of \cite[Theorem 3.1]{BergerMoerdijk}, the ``Hopf
  interval'' being simply the $m$-groupoid $\In{1}$.
\end{remark}

\section{The folk model category structure is monoidal for the join}
\label{sec:join}

In this section, we will recall the definition of the join of
\oo-categories, introduced by the first author and Maltsiniotis in
\cite{AraMaltsiJoint}, and we will prove that the resulting monoidal
category structure is compatible with the folk model category structure.

\medbreak

The strategy to define the join is similar to the one for the Gray tensor
product. In particular, we start by defining the join at the level of
augmented directed complexes.

\begin{paragraph}\label{paragr:def_join_ADC}
  The \ndef{join} $K \join L$ of two augmented directed
  complexes $K$ and $L$ is defined in the following way:
  \begin{itemize}
    \item For $n \ge 0$, we have
      \[ (K \join L)_n = \bigoplus_{\substack{i + 1 + j = n \\ i \ge -1, j
      \ge -1}} K_i \otimes L_j, \]
      where by convention $K_{-1} = \Z$ and $L_{-1} = \Z$. The positive
      generator of $K_{-1}$ or~$L_{-1}$ will be denoted by $\emptyset$. If
      $x$ is in $K_i$ and $y$ is in $K_j$, we will denote by $x \join y$
      the element $x \otimes y$ seen as an $(i+1+j)$-chain of $K \join L$.
    \item For $x$ in $K_i$ and $y$ in $K_j$ with $i + 1 + j \ge 1$, we have
      \[ d(x \join y) = dx \join y + (-1)^{i+1} x \join dy, \]
      where by convention $dz = e(z)\emptyset$ if the degree of $z$ is $0$, and $dz =
      0$ if the degree of $z$ is $-1$.
    \item For $x$ in $K_0$ and $y$ in $L_0$, we have
      \[ e(x \join \emptyset) = e(x) \quadand e(\emptyset \join y) = e(y). \]
  \item The submonoid $(K \join L)^\ast_n$ is defined to be generated
    by the subset
    \[ \bigoplus_{\substack{i + 1 + j = n \\ i \ge -1, j \ge -1}}
      K^\ast_i \otimes L^\ast_j \] of $(K \join L)_n$.
  \end{itemize}
  The join defines a (non-symmetric) monoidal category structure on the
  category of augmented directed complexes. Moreover, the first author
  and Maltsiniotis proved (see \cite[Corollary 6.21]{AraMaltsiJoint}) that
  this monoidal category structure restricts to the full subcategory of
  strong Steiner complexes.
\end{paragraph}

\begin{theorem}[Ara--Maltsiniotis]\label{thm:def_join}
  There exists a unique, up to unique isomorphism, locally biclosed monoidal
  category structure (see paragraph~\ref{paragr:def_loc_biclosed})
  on $\ooCat$ making the functor $\nu_{|\Stf} :
  \Stf \to \ooCat$ a monoidal functor, where $\Stf$ is endowed with the
  monoidal category structure given by the join.
\end{theorem}

\begin{proof}
  This is \cite[Theorem 6.29]{AraMaltsiJoint}. 
\end{proof}

\begin{paragraph}\label{paragr:def_join}
  We define the \ndef{join} of \oo-categories to be the monoidal product given
  by the previous theorem. If $X$ and $Y$ are two \oo-categories, their join
  will be denoted by~$X \join Y$. Explicitly, one has
  \[
    X \join Y =
    \limind_{\substack{\nu(K) \to X, \, K \in \Stf\\
      \nu(L) \to Y, \, L \in \Stf}}
      \nu(K \join L).
  \]
  The unit of the join is the empty \oo-category. As a consequence, if $X$
  and $Y$ are two \oo-categories, we get canonical \oo-functors $\iota_1 : X
  \to X \join Y$ and $\iota_2 : Y \to X \join Y$.

  The fact that the join is locally biclosed means that the functors
  \[
    \begin{split}
      \ooCat & \to \cotr{\ooCat}{X} \\
      Y & \mapsto (X \join Y, \iota_1 : X \to X \join Y)
    \end{split}
  \]
  and
  \[
    \begin{split}
      \ooCat & \to \cotr{\ooCat}{Y} \\
      X & \mapsto (X \join Y, \iota_2 : Y \to X \join Y)
    \end{split}
  \]
  admit right adjoints. We thus get pairs of adjoint functors
  \[
    \begin{split}
      \ooCat & \to \cotr{\ooCat}{X}, \\
      Y & \mapsto (X \join Y, \iota_1)
    \end{split}
    \qquad \qquad
    \begin{split}
      \cotr{\ooCat}{X} & \to \ooCat \\
      (Z, X \xto{u} Z) & \mapsto \cotr{Z}{u} \\
    \end{split}
  \]
  and
  \[
    \begin{split}
      \ooCat & \to \cotr{\ooCat}{Y}, \\
      X & \mapsto (X \join Y, \iota_2)
    \end{split}
    \qquad \qquad
    \begin{split}
      \cotr{\ooCat}{Y} & \to \ooCat, \\
      (Z, Y \xto{v} Z) & \mapsto \trm{Z}{v} \\
    \end{split}
  \]
  so that, if $X$ and $Y$ are \oo-categories and $u : X \to Z$ and $v : Y
  \to Z$ are \oo-functors, we have natural bijections
  \[
    \begin{split}
      \Hom_{\cotr{\ooCat}{X}}((X \join Y, \iota_1), (Z, u))
      & \simeq \Hom_{\ooCat}(Y, \cotr{Z}{u}), \\
      \Hom_{\cotr{\ooCat}{Y}}((X \join Y, \iota_2), (Z, v))
      & \simeq \Hom_{\ooCat}(X, \trm{Z}{v}). \\
    \end{split}
  \]
  (See \cite[Remark 6.37]{AraMaltsiJoint} for the reason for the decoration
  ``co'' in $\trm{Z}{v}$.)

  One important consequence of the existence of these adjoints is that the
  join commutes with connected colimits in each variable.
\end{paragraph}

\begin{examples}
  Here are some examples of joints of \oo-categories:
  \[
      \Dn{0} \join \Dn{0} = \xymatrix{\bullet \ar[r] & \bullet} = \Dn{1}
      \pbox{,}
      \qquad
      \qquad
      \Dn{0} \join \Dn{1} =
      \raisebox{2.0pc}{
      \xymatrix@R=1pc{
        & \bullet \ar[dd] \\
        \bullet \ar[ru] \ar[rd]_(0.50){}="s" \\
        & \bullet \pbox{,}
        \ar@{}"s";[uu]_(0.15){}="ss"_(0.65){}="tt"
        \ar@2"ss";"tt"
      }}
  \]
  \[
      \Dn{1} \join \Dn{1} =
    \raisebox{1.5pc}{
    \xymatrix{
      \bullet \ar[r]_(0.60){}="03" \ar[d] \ar[dr]_{}="02"_(0.60){}="02'" &
      \bullet
      &
      \bullet \ar[r]_(0.40){}="03'" \ar[d]_{}="t3"
        &
      \bullet
      \\
      \bullet \ar[r] & \bullet \ar[u]_{}="s3"
      &
      \bullet \ar[r] \ar[ur]_{}="13"_(0.40){}="13'" & \bullet \ar[u] \pbox{.}
      \ar@{}"s3";"t3"_(0.20){}="ss3"_(0.80){}="tt3"
      \ar@3"ss3";"tt3"
      \ar@{}"13";[]_(0.05){}="s123"_(0.85){}="t123"
      \ar@2"s123";"t123"
      \ar@{}"02";[lll]_(0.05){}="s012"_(0.85){}="t012"
      \ar@2"s012";"t012"
      \ar@{}"03'";"13'"_(0.05){}="s013"_(0.85){}="t013"
      \ar@2"s013";"t013"
      \ar@{}"03";"02'"_(0.05){}="s023"_(0.85){}="t023"
      \ar@2"s023";"t023"
      \\
    }}
  \]
\end{examples}

We now begin to prove that the join is compatible with the folk model
category structure.

\begin{proposition}\label{prop:join_comp_cof}
  If
  \[ i : X \to Y \quadand j : Z \to T \]
  are two folk cofibrations, then the \oo-functor
  \[
    i \boxjoin j :
    Y \join Z \amalg_{X \join Z} X \join T \to Y \join T
  \]
  is also a folk cofibration.
\end{proposition}

\begin{proof}
  The proof is essentially the same as the one of
  Theorem~\ref{thm:tens_comp_cof}. It is immediate that if $K$ and $L$ are
  two augmented directed complexes with basis (that we denote
  by~$B_K$ and $B_L$, respectively, following the notation introduced in
  paragraph~\ref{paragr:conv_B_K}), then $K \join L$ (see
  paragraph~\ref{paragr:def_join_ADC}) is an augmented directed complex with
  basis
  \[
    B_{K \join L} =
      \{ x \join y \mid x \in B_K, \, y \in B_L\}
        \cup \{x \join \emptyset \mid x \in B_K\}
        \cup \{\emptyset \join y \mid y \in B_L\}.
  \]
  From this, one deduces as in Proposition~\ref{prop:boxprod_basis} that
  if $i : K \to L$ and $j : M \to N$ are two rigid monomorphisms between
  augmented directed complexes with basis, then the morphism
  \[
    i \boxjoin j : L \join M \amalg_{K \join M} K \join N \to L \join N
  \]
  is a rigid monomorphism between augmented directed complexes with basis
  which identifies $L \join M \amalg_{K \join M} K \join N$ with the
  subcomplex generated by
  \[ B_L \join B_M \cup B_K \join B_N \cup B_L \join \{\emptyset\} \cup
  \{\emptyset\} \join B_N. \]
  One then deduces, as in Proposition~\ref{prop:boxprod_rigid_mono}, that if
  $K$, $L$, $M$ and $N$ are assumed to be strong Steiner complexes, then the
  \oo-functor
  \[
    \nu(i) \boxjoin \nu(j) : \nu(L) \join \nu(M) \amalg_{\nu(K) \join
    \nu(M)} \nu(K) \join \nu(N) \to \nu(L) \join \nu(N)
  \]
  is a folk cofibration. To do so, one needs the fact that rigid
  monomorphisms and strong Steiner complexes are stable under join (see
  \cite[Proposition 6.17 and Corollary~6.21]{AraMaltsiJoint}) and that
  the functor $\nu_{|\Stf} : \Stf \to \ooCat$ is monoidal for the join
  (see \cite[Theorem~6.29]{AraMaltsiJoint}). The result then follows from
  the fact that generating cofibrations are of the form $\nu(i)$ for $i$ a
  rigid monomorphism between strong Steiner complexes
  (see Proposition~\ref{lemma:gen_cof_rig}).
\end{proof}

\begin{corollary}
  The join of two cofibrant \oo-categories is a cofibrant
  \oo-category.
\end{corollary}

\begin{proof}
  This follows immediately from the previous proposition.
\end{proof}

To end the proof of the compatibility of the join with the folk model
category structure, we will apply
Lemma~\ref{lemma:model_monoidal_criterion}. The only non-trivial remaining
hypothesis to be checked is \ref{hyp:trivial_cof}. To do so, we will prove
that the class of $\J$-transformations is stable by taking the join by an object
on the right (a fact which was trivial for the tensor product) by showing
that, for every \oo-category $T$, the functor ${-} \join T$ can be endowed
with what is called a tensorial strength for the Gray tensor product. We
start by defining this tensorial strength at the level of augmented directed
complexes.

\begin{paragraph}
  Let $K$, $L$ and $M$ be three augmented directed complexes. We define a
  morphism
  \[ \sigma : K \otimes (L \join M) \to (K \otimes L) \join M, \]
  natural in $K$, $L$ and $M$, in the following way. For $n \ge 0$, we set
  \[
    \sigma_n(x \otimes (y \join z))
    =
    \begin{cases}
      e(x)(y \join z) & \text{if $|y| = -1$,} \\
      (x \otimes y) \join z & \text{if $|y| \ge 0$,}
    \end{cases}
  \]
  where by convention $e(x) = 0$ if $x$ is not of degree $0$.
\end{paragraph}

\begin{proposition}\label{prop:def_strength_ADC}
  The morphisms $\sigma_n$ define a morphism of augmented directed
  complexes $\sigma : K \otimes (L \join M) \to (K \otimes L) \join M$.
\end{proposition}

\begin{proof}
  It is immediate that $\sigma_n$ respects positive elements. Let us prove
  that $\sigma_0$ is compatible with the augmentations. If the degree of $x
  \otimes (y \join z)$ is $0$ then $|x| = 0$ and
  \begin{itemize}
    \item either $|y| = -1$ and $|z| = 0$, in which
      case we have
      \[
        e(\sigma_0(x \otimes (y \join z))) = e(e(x)(y \join
        z)) = e(x)e(y \join z) = e(x \otimes (y \join z)),
      \]
    \item or $|y| = 0$ and $|z| = -1$, in which case we can assume that $z =
    \emptyset$ and we have
      \[
        \begin{split}
          e(\sigma_0(x \otimes (y \join \emptyset))) & =
          e((x \otimes y) \join \emptyset) = e(x \otimes y) \\ & = e(x)e(y)
          = e(x)e(y \join \emptyset) = e(x \otimes (y \join \emptyset)).
        \end{split}
      \]
  \end{itemize}

  Suppose now that the degree $n$ of $x \otimes (y \join z)$ is at least
  $1$ and let us prove that
  \[
    \sigma_{n-1}d(x \otimes (y \join z)) = d\sigma_n(x \otimes (y \join z)).
  \]
  We will freely use the conventions for the differentials of tensors and
  joins introduced in paragraphs~\ref{paragr:def_tens_ADC}
  and~\ref{paragr:def_join_ADC}. We distinguish four cases:
  \begin{itemize}
    \item If $|y| = -1$, then we have
      {
        \allowdisplaybreaks
        \begin{align*}
          \sigma_{n-1}d(x \otimes (y \join z))
          & =
          \sigma_{n-1}\big(dx \otimes (y \join z) + (-1)^{|x|} x
          \otimes d(y \join z)\big) \\*
          & =
          e(dx) (y \join z) + (-1)^{|x|} e(x)
          d(y \join z) \\*
         & \phantom{\simeq 1} \text{(being careful with the case $|z| = 0$)} \\ 
          & =
          e(x) d(y \join z) \\*
          & \phantom{\simeq 1} \text{(as $ed = 0$ and $e(x) = 0$ if $|x|
              \neq 0$)} \\
          & = d(e(x) (y \join z)) \\*
          & = d\sigma_n(x \otimes (y \join z)).
        \end{align*}
      }%
    \item If $|y| = 0$ and $|z| = -1$, then we have
      \[
        \begin{split}
          \sigma_{n-1}d(x \otimes (y \join z))
          & =
          \sigma_{n-1}(dx \otimes (y \join z)) \\
          & =
          (dx \otimes y) \join z \\
          & =
          d(x \otimes y) \join z \\
          & =
          d((x \otimes y) \join z) \\
          & =
          d\sigma_n(x \otimes (y \join z)).
        \end{split}
      \]
    \item If $|y| = 0$ and $|z| \ge 0$, then we have
    { \allowdisplaybreaks
        \begin{align*}
          \MoveEqLeft
          \sigma_{n-1}d(x \otimes (y \join z)) \\*
          & =
          \sigma_{n-1}\big(dx \otimes (y \join z) + (-1)^{|x|}x \otimes d(y
          \join z)\big) \\*
          & =
          \sigma_{n-1}\big(dx \otimes (y \join z) + (-1)^{|x|}x \otimes
            (e(y)\emptyset
          \join z) + (-1)^{|x|+|y|+1} x \otimes (y \join dz)\big) \\
          & =
          (dx \otimes y) \join z + (-1)^{|x|}e(x)e(y)(\emptyset \join z)
          + (-1)^{|x|+|y|+1} (x \otimes y) \join dz \\*
          & \phantom{\simeq 1} \text{(being careful with the case $|z| =
          0$)} \\
          & =
          (dx \otimes y) \join z + (e(x \otimes y)\emptyset \join z)
          + (-1)^{|x \otimes y|+1} (x \otimes y) \join dz \\*
          & \phantom{\simeq 1} \text{(as the second term is null if $|x|
        \neq 0$)} \\
        & = d(x \otimes y) \join z + (-1)^{|x \otimes y|+1} (x \otimes y)
        \join dz \\
          & \phantom{\simeq 1} \text{(distinguishing the cases
          $|x| = 0$ and $|x| \neq 0$)} \\
          & =
          d((x \otimes y) \join z) \\*
          & =
          d\sigma_n(x \otimes (y \join z)).
        \end{align*}
      }%
    \item If $|y| \ge 1$, then we have
    {
    \allowdisplaybreaks
        \begin{align*}
          \MoveEqLeft
          \sigma_{n-1}d(x \otimes (y \join z)) \\*
          & =
          \sigma_{n-1}\big(dx \otimes (y \join z) + (-1)^{|x|}x \otimes d(y
          \join z)\big) \\*
          & =
          \sigma_{n-1}\big(dx \otimes (y \join z) + (-1)^{|x|}x \otimes (dy
          \join z) + (-1)^{|x|+|y|+1} x \otimes (y \join dz)\big) \\
          & =
          (dx \otimes y) \join z + (-1)^{|x|}(x \otimes dy)
          \join z + (-1)^{|x|+|y|+1} (x \otimes y) \join dz \\*
          & \phantom{\simeq 1} \text{(being careful with the cases $|z| =
        -1$ and $|z| = 0$)} \\
          & =
          d(x \otimes y) \join z + (-1)^{|x \otimes y| + 1} (x \otimes y) \join dz \\
          & =
          d((x \otimes y) \join z) \\*
          & =
          d\sigma_n(x \otimes (y \join z)),
        \end{align*}
      }%
      thereby proving the result. \qedhere
  \end{itemize}
\end{proof}

\begin{proposition}\label{prop:strenth_ADC}
  For any augmented directed complex $T$, the morphism $\sigma$ of the
  previous proposition defines a tensorial strength for the tensor product
  on the functor
  \[
    \begin{split}
      \ADC & \to \ADC \\
      K & \mapsto K \join T,
    \end{split}
  \]
  meaning that, for every augmented directed complexes $K$, $L$ and $M$, the
  triangles
  \[
    \xymatrix@C=0pc{
      K \otimes L \otimes (M \join T)
      \ar[rr]^{K \otimes s} \ar[dr]_s
      &&
      K \otimes ((L \otimes M) \join T)
      \ar[dl]^s
      \\
      &
      (K \otimes L \otimes M) \join T
    }
  \]
  and
  \[
    \xymatrix@C=1pc{
      \Z \otimes (K \join T)
      \ar[rr]^s \ar[dr]_\lambda
      &&
      (\Z \otimes K) \join T
      \ar[dl]^{\lambda \join T}
      \\
      &
      K \join T
      & \pbox{,}
    }
  \]
  where we have neglected the associativity constraints and $\lambda$
  denotes the left unit constraint, are commutative.
\end{proposition}

\begin{proof}
  This follows from direct calculations.
\end{proof}

\begin{paragraph}
  Let $X$, $Y$ and $Z$ be three \oo-categories. We define an \oo-functor
  \[
    s : X \otimes (Y \join Z) \to (X \otimes Y) \join Z,
  \]
  natural in $X$, $Y$ and $Z$, in the following way. First, recall that any
  \oo-category $T$ is a canonical colimit of \oo-categories associated to
  strong Steiner complexes in the sense that we have a canonical isomorphism
  \[ T \simeq \limind_{\substack{\nu(K) \to T, \, T \in \Stf}} \nu(K). \]
  (This follows from the fact that $\nu(\Stf)$ contains Joyal's category
  $\Theta$ which is dense in~$\ooCat$, see also \cite[Theorem
  7.1]{Steiner}.) This colimit is connected as the null augmented directed
  complex is an initial object of the category of strong Steiner complexes. As by
  Theorems~\ref{thm:def_tens} and~\ref{thm:def_join}, both the Gray tensor
  product and the join commute with connected colimits in each variable and
  are compatible with the functor $\nu_{|\Stf} : \Stf \to \ooCat$, we get
  canonical isomorphisms
  \[
    \begin{split}
    X \otimes (Y \join Z)
    & \simeq
    \limind_{\substack{\nu(K) \to X, \, K \in \Stf\\
      \nu(L) \to Y, \, L \in \Stf\\
      \nu(M) \to Z, \, M \in \Stf}}
      \nu(K \otimes (L \join M)), \\
    (X \otimes Y) \join Z
    & \simeq
    \limind_{\substack{\nu(K) \to X, \, K \in \Stf\\
      \nu(L) \to Y, \, L \in \Stf\\
      \nu(M) \to Z, \, M \in \Stf}}
      \nu((K \otimes L) \join M).
    \end{split}
  \]
  We thus obtain our \oo-functor $s :  X \otimes (Y \join Z)\to (X \otimes
  Y) \join Z$ by taking the colimit over $K$, $L$ and $M$ of the
  \oo-functors
  \[
    \nu(\sigma) : \nu(K \otimes (L \join M)) \to \nu((K \otimes L) \join M),
  \]
  where $\sigma : K \otimes (L \join M) \to (K \otimes L) \join M$ is the
  morphism of Proposition~\ref{prop:def_strength_ADC}.
\end{paragraph}

\begin{proposition}
  For any \oo-category $T$, the \oo-functor $s$ of the previous paragraph
  defines a tensorial strength for the Gray tensor product on the functor
  \[
    \begin{split}
      \ooCat & \to \ooCat \\
      X & \mapsto X \join T,
    \end{split}
  \]
  meaning that, for every \oo-categories $X$, $Y$ and $Z$, the triangles
  \[
    \xymatrix@C=0pc{
      X \otimes Y \otimes (Z \join T)
      \ar[rr]^{X \otimes s} \ar[dr]_s
      &&
      X \otimes ((Y \otimes Z) \join T)
      \ar[dl]^s
      \\
      &
      (X \otimes Y \otimes Z) \join T
    }
  \]
  and
  \[
    \xymatrix@C=1pc{
      \Dn{0} \otimes (X \join T)
      \ar[rr]^s \ar[dr]_\lambda
      &&
      (\Dn{0} \otimes X) \join T
      \ar[dl]^{\lambda \join T}
      \\
      &
      X \join T
      & \pbox{,}
    }
  \]
  where we have neglected the associativity constraints and $\lambda$
  denotes the left unit constraint, are commutative.
\end{proposition}

\begin{proof}
  Using the same arguments as for the definition of $s$ in terms of
  $\sigma$, we get that these two triangles are colimits of the image by
  $\nu$ of triangles as in Proposition~\ref{prop:strenth_ADC}. The result
  thus follows from this proposition, which asserts that these triangles are
  commutative.
\end{proof}

\begin{proposition}
  Let $f, g : X \to Y$ be two \oo-functors. If $h : \J \otimes X \to Y$
  is a $\J$-transformation from $f$ to $g$, then, for any \oo-category $Z$,
  the \oo-functor
  \[
    \J \otimes (X \join Z) \xto{s} (\J \otimes X) \join Z \xto{h
    \join Z} Y \join Z
  \]
  defines a $\J$-transformation from $f \join Z$ to $g \join Z$.
\end{proposition}

\begin{proof}
  This follows from the commutativity of the diagram
  \[
    \xymatrix@C=2.5pc{
      X \join Z \ar[d]_{\simeq} \ar[dr]^{\simeq}
      \ar@/^7ex/[ddrr]^{f \join Z}
      \\
      \Dn{0} \otimes (X \join Z)
      \ar[r]^s \ar[d]_{\funcell{0} \otimes (X \join Z)}
      &
      (\Dn{0} \otimes X) \join Z
      \ar[d]^{(\funcell{0} \otimes X) \join Z}
      \\
      \J \otimes (X \join Z)
      \ar[r]^s
      &
      (\J \otimes X) \join Z
      \ar[r]^-{h \join Z}
      &
      Y \join Z
      \\
      \Dn{0} \otimes (X \join Z)
      \ar[r]^s \ar[u]^{\funcell{1} \otimes (X \join Z)}
      &
      (\Dn{0} \otimes X) \join Z
      \ar[u]_{(\funcell{1} \otimes X) \join Z}
      \\
      X \join Z \ar[u]^{\simeq} \ar[ur]_{\simeq}
      \ar@/_7ex/[uurr]_{g \join Z}
      && \pbox{,}
    }
  \]
  the two squares in the middle of the diagram being commutative by
  naturality of $s$, the two triangles by the previous proposition, and the
  two other small diagrams by hypothesis on $h$.
\end{proof}

\begin{remark}
  The analogous statement for oplax transformations, obtained by replacing
  $\J$ by $\Dn{1}$, is true as well, the proof applying \forlang{mutatis
  mutandis}.
\end{remark}

\begin{theorem}\label{thm:join_mon}
  The folk model category structure on $\ooCat$ is monoidal for the join.
\end{theorem}

\begin{proof}
  We apply Lemma~\ref{lemma:model_monoidal_criterion}. The hypothesis
  \ref{hyp:unit} and \ref{hyp:target_cof} are true for the folk model
  category structure and the hypothesis \ref{hyp:tens_I} is true for any
  locally biclosed monoidal category, as the initial object is the tensor
  unit. The hypothesis \ref{hyp:cof} follows from
  Proposition~\ref{prop:join_comp_cof}.

  It remains to prove \ref{hyp:trivial_cof}. Let $i$ be a generating trivial
  cofibration and let $Z$ be an \oo-category. As seen in the proof of
  Proposition~\ref{prop:i_otimes_Z}, the \oo-functor $i$ is a
  $J_1$\=/deformation retract. The previous proposition thus implies that $i
  \join Z$ is also a $J_1$-deformation retract, and therefore a weak
  equivalence by Proposition~\ref{prop:J1_immer_are_we}. The fact that $Z
  \join i$ is also a weak equivalence follows from a duality argument, as in
  the proof of Corollary~\ref{coro:Z_otimes_i}, using the canonical natural
  isomorphism
  \[ (X \join Y)^\op \simeq Y^\op \join X^\op \]
  (see \cite[Proposition 6.35]{AraMaltsiJoint}). This ends the proof of
  \ref{hyp:trivial_cof} and hence of the theorem.
\end{proof}

\goodbreak

The statement analogous to Proposition~\ref{prop:tens_mon_ax} holds as well,
the proof being a direct adaption:

\begin{proposition}\ %
  \begin{enumerate}
  \item 
  Transfinite compositions of pushouts of join of an object (on
  the left) and a folk trivial cofibration are folk weak equivalences.

  \item
  Transfinite compositions of pushouts of join of a
  folk trivial cofibration and an object (on the right) are folk weak
  equivalences.
  \end{enumerate}
\end{proposition}

\begin{remark}
  In particular, the join of a folk trivial cofibration by an object
  (on the left or on the right) is a folk weak equivalence.
\end{remark}

Finally, as for the Gray tensor product, the join of \oo-categories induces a
join of $m$\nbd-categories, which is compatible with the folk model category
structure on $\nCat{m}$:

\begin{theorem}[Ara--Maltsiniotis]
  Let $m \ge 0$. The \ndef{join of $m$-categories}
  \[ X \rjoin{m} Y = \reflect (X \join Y), \]
  where $r : \ooCat \to \nCat{m}$ denotes the left adjoint to the inclusion
  functor $\nCat{m} \hookto \ooCat$, defines a locally biclosed monoidal
  category structure on $\nCat{m}$.
\end{theorem}

\begin{proof}
  This is the main result of \cite[Chapter 8]{AraMaltsiJoint}.
\end{proof}

\begin{theorem}
  The folk model category structure on $\nCat{m}$ is monoidal for
  the join of $m$-categories.
\end{theorem}

\begin{proof}
  This follows from a straightforward adaption of the proof of
  Proposition~\ref{prop:Day_model}, which only uses connected colimits, with
  whom any locally biclosed monoidal tensor commutes.
\end{proof}

\appendix

\section{Monoidal model categories and derived tensor products}

In this appendix, we recall classical results on biclosed monoidal
model categories and extend them to locally biclosed monoidal model
categories. In particular, we will get that the ``local internal $\Homi$''
of the join, the so-called generalized slices, can be right-derived as
functors of two variables.

\begin{paragraph}\label{paragr:def_monoidal_model}
  A \ndef{monoidal model category} is a model category whose underlying
  category is endowed with a monoidal category structure satisfying
  the following compatibility axioms:
  \begin{enumerate}[label=M\arabic*)]
  \item the tensor product
    $\otimes : \M \times \M \to \M$ satisfies the
    \ndef{pushout-product axiom}: if $i : A \to B$ and $j : C \to D$
    are two cofibrations, then the morphism
    \[
      i \boxprod j : B \otimes C \amalg_{A \otimes C} A \otimes D \to
      B \otimes D,
    \]
    induced by the commutative square
    \[
      \xymatrix{ A \otimes C \ar[d]_{i \otimes C} \ar[r]^{A \otimes j}
        &
        A \otimes D \ar[d]^{i \otimes D} \\
        B \otimes C \ar[r]_{B \otimes j} & B \otimes D \pbox{,} }
    \]
    is a cofibration. Moreover, if either $i$ or $j$ is a trivial
    cofibration, then so is $i \boxprod j$.
  \item the tensor product satisfies the \ndef{unit axiom}:
    for every cofibrant replacement $p : QI \trivfib I$ of the tensor unit
    and every cofibrant object $A$, both $p \otimes A : QI \otimes A \to I
    \otimes A$ and $A \otimes p : A \otimes QI \to A \otimes I$ are weak
    equivalences.
  \end{enumerate}
\end{paragraph}

\begin{remark}\label{rem:M1}
  The pushout-product axiom implies that, if the cofibrations \resp{the
  trivial cofibrations} are stable by tensoring by the initial object
  $\emptyset$, then they are stable by tensoring by any cofibrant object
  $X$. In particular, if this condition for trivial cofibrations is
  satisfied and the tensor unit $I$ is cofibrant, then, by Ken Brown's
  lemma, the pushout-product axiom implies the unit axiom.
\end{remark}

\begin{lemma}\label{lemma:monoidal_model_gen}
  Let $\M$ be a cofibrantly generated model category with sets of
  generating cofibrations $I$ and of generating trivial cofibrations $J$
  and let $\otimes : \M \times \M \to \M$ be a functor which commutes
  with pushouts and transfinite compositions in each variable. Then
  for $\otimes$ to satisfy the pushout-product axioms, it suffices
  that it holds for cofibrations in $I$ and trivial cofibrations in
  $J$.
\end{lemma}

\begin{proof}
  See for instance (the proof of) \cite[Lemma 4.1.4]{Morel}.
\end{proof}

\begin{remark}\label{rem:monoidal_model_gen}
  More precisely, in the situation of the previous lemma, each of the three
  conditions appearing in the pushout-product axiom can be checked on
  ``generators''.
\end{remark}

\begin{proposition}\label{prop:derived_tens}
  If $\M$ is a monoidal model category having the additional property that
  trivial cofibrations are stable by tensoring by the initial object, then
  the tensor product admits a total left derived functor
  $\Lotimes : \Ho(\M) \times \Ho(\M) \to \Ho(\M)$ and this derived
  tensor product defines a monoidal category structure on $\Ho(\M)$.
\end{proposition}

\begin{proof}
  It follows from Remark~\ref{rem:M1} that the tensor
  product of two trivial cofibrations between cofibrant objects is a weak
  equivalence. By Ken Brown's lemma, this implies that the tensor product
  preserves weak equivalences between cofibrant objects and hence, by a
  classical result of Quillen \cite[I.4, Proposition 1]{Quillen},
  that its total left derived functor $\Lotimes$ exists. Checking that
  $\Lotimes$ indeed defines a monoidal category structure on $\Ho(\M)$ is
  not difficult (see the proof of \cite[Theorem~4.3.2]{Hovey}).
\end{proof}

\begin{paragraph}\label{paragr:def_biclosed}
  Recall that a monoidal category $\C$ is said to be \ndef{biclosed}
  if, for every object~$X$ of $\C$, the functor
  \[
    \begin{split}
      \C & \to \C \\
      Y & \mapsto X \otimes Y
    \end{split}
  \]
  and, for every object $Y$ of $\C$, the functor
  \[
    \begin{split}
      \C & \to \C \\
      X & \mapsto X \otimes Y
    \end{split}
  \]
  admit right adjoints. In this case, we get pairs of adjoint
  functors
  \[
    \begin{split}
      \C & \to \C \\
      Y & \mapsto X \otimes Y
    \end{split}
    \qquad \qquad
    \begin{split}
      \C & \to \C \\
      Z & \mapsto \Homil_\C(X, Z)
    \end{split}
  \]
  and
  \[
    \begin{split}
      \C & \to \C \\
      X & \mapsto X \otimes Y
    \end{split}
    \qquad \qquad
    \begin{split}
      \C & \to \C \\
      Z & \mapsto \Homir_\C(Y, Z) .
    \end{split}
  \]
  Moreover, $\Homil_\C$ and $\Homir_\C$ extend to
  functors
  \[
    \Homil_\C, \Homir_\C : \C^\op \times \C \to \C
  \]
  and, if $X$, $Y$ and $Z$ are three objects, we get natural
  bijections
  \[
    \Hom_\C(X, \Homir_\C(Y, Z)) \simeq \Hom_\C(X \otimes Y, Z) \simeq
    \Hom_\C(Y, \Homil_\C(X, Z)).
  \]
\end{paragraph}

\begin{remark}
  Let $\C$ be a monoidal category. If $\C$ is biclosed, then its
  tensor product preserves colimits in each variable. By a classical
  adjoint theorem, the converse holds provided that the category $\C$
  is locally presentable.
\end{remark}

\begin{proposition}\label{prop:def_closed_mon_model}
  Let $\M$ be a model category endowed with a biclosed monoidal
  category structure. Then the following conditions are equivalent:
  \begin{enumerate}[label=\roman*)]
  \item the tensor product satisfies the pushout-product axiom,
  \item for every cofibration $i : A \to B$ and every fibration
    $p : X \to Y$, the induced map
    \[
      \Homil_\M(B, X) \to \Homil_\M(A, X) \times_{\Homil_\M(A, Y)}
      \Homil_\M(B, Y)
    \]
    is a fibration that is trivial if either $i$ or $p$ is,
  \item for every cofibration $j : C \to D$ and every fibration
    $p : X \to Y$, the induced map
    \[
      \Homir_\M(D, X) \to \Homir_\M(C, X) \times_{\Homir_\M(C, Y)}
      \Homir_\M(D, Y)
    \]
    is a fibration that is trivial if either $j$ or $p$ is.
  \end{enumerate}
\end{proposition}

\begin{proof}
  See for instance \cite[Lemma 4.2.2]{Hovey}.
\end{proof}

\begin{paragraph}
  A \ndef{biclosed monoidal model category} is a monoidal model
  category whose underlying monoidal category is biclosed.
  Our example of interest in this paper is the folk model category structure
  on $\ooCat$ (or more generally $\mnCat{m}{n}$) endowed with the Gray
  tensor product (see Theorems~\ref{thm:Gray_mon_mod} and
  \ref{thm:def_tens}). Since tensoring any object by the initial object gives the
  initial object in a biclosed monoidal model category, the hypothesis of
  Proposition~\ref{prop:derived_tens} are satisfied in such a model category
  and the monoidal tensor $\otimes$ thus admits a total right derived
  functor $\Lotimes$.
\end{paragraph}

\begin{proposition}
  Let $\M$ be a biclosed monoidal model category. Then, if $X$ is a
  cofibrant object of $\M$, the adjoint pair
  \[
    \begin{split}
      \M & \to \M \\
      Y & \mapsto X \otimes Y
    \end{split}
    \qquad \qquad
    \begin{split}
      \M & \to \M \\
      Z & \mapsto \Homil_\M(X, Z)
    \end{split}
  \]
  is a Quillen pair and, likewise, if $Y$ is a cofibrant object of
  $\M$, the adjoint pair
  \[
    \begin{split}
      \M & \to \M \\
      X & \mapsto X \otimes Y
    \end{split}
    \qquad \qquad
    \begin{split}
      \M & \to \M \\
      Z & \mapsto \Homir_\M(Y, Z)
    \end{split}
  \]
  is a Quillen pair.
\end{proposition}

\begin{proof}
  It suffices to show that the left adjoints respect cofibrations and
  trivial cofibrations. This follows from the pushout-product axiom (see
  Remark~\ref{rem:M1}).
\end{proof}

\begin{theorem}[Hovey]
  Let $\M$ be a biclosed monoidal model category. Then the monoidal
  category structure on $\Ho(\M)$ defined by the derived tensor
  product is biclosed.  Moreover, the functors
  \[ \Homil_\M, \Homir_\M : \M^\op \times \M \to \M \] admit total
  right derived functors and we have
  \[ \RHomil_\M = \Homil_{\Ho(\M)} \quadand \RHomir_\M =
    \Homir_{\Ho(\M)}. \]
\end{theorem}

\begin{proof}
  This is \cite[Theorem 4.3.2]{Hovey}. Let us just briefly recall why
  these functors admit total right derived functors: one deduces from
  Proposition~\ref{prop:def_closed_mon_model} that these functors
  preserve trivial fibrations between fibrant objects; by Ken Brown's
  lemma, this implies that they preserve weak equivalences between
  fibrant objects and hence that they admit total right derived
  functors.
\end{proof}

We now move on to locally biclosed monoidal categories, as introduced in
\cite{AraMaltsiJoint}.

\begin{paragraph}\label{paragr:def_loc_biclosed}
  Let $\C$ be a monoidal category. Let us denote by $\join$ the tensor
  product of $\C$ and suppose that its tensor unit is an initial
  object $\emptyset$. If $X$ and $Y$ are two objects, we get morphisms
  \[ X \xto{\iota_1} X \join Y \xot{\iota_2} Y \]
  by precomposing the morphisms
  \[ X \join \emptyset \xto{X \join \emptyset_Y} X \join Y
    \xot{\emptyset_X \join Y} \emptyset \join Y, \]
  where $\emptyset_Z$ denotes the unique morphism from $\emptyset$ to $Z$,
  with the unit constraints. Using these morphisms, we obtain, for every
  object $X$ of $\C$, a functor
  \[
    \begin{split}
      \C & \to \cotr{\C}{X} \\
      Y & \mapsto (X \join Y, \iota_1)
    \end{split}
  \]
  and, for every object $Y$ of $\C$, a functor
  \[
    \begin{split}
      \C & \to \cotr{\C}{Y} \\
      X & \mapsto (X \join Y, \iota_2).
    \end{split}
  \]

  We say that $\C$ is \ndef{locally biclosed} if its tensor unit is an
  initial object and if the two above functors admit right
  adjoints. In this case, we thus get pairs of adjoint functors
  \[
    \begin{split}
      \C & \to \cotr{\C}{X} \\
      Y & \mapsto (X \join Y, \iota_1)
    \end{split}
    \qquad \qquad
    \begin{split}
      \cotr{\C}{X} & \to \C \\
      (Z, X \xto{u} Z) & \mapsto \cotr{Z}{u}
    \end{split}
  \]
  and
  \[
    \begin{split}
      \C & \to \cotr{\C}{Y} \\
      X & \mapsto (X \join Y, \iota_2)
    \end{split}
    \qquad \qquad
    \begin{split}
      \cotr{\C}{Y} & \to \C \\
      (Z, Y \xto{v} Z) & \mapsto \tr{Z}{v}.
    \end{split}
  \]
  By abuse of notation, we will often denote $\cotr{Z}{u}$ by
  $\cotr{Z}{X}$ and, similarly, $\tr{Z}{v}$ by~$\tr{Z}{Y}$. These
  functors are called the \ndef{slice functors}. By definition, we
  have natural bijections
  \[
    \Hom_{\cotr{\C}{X}}((X \join Y, \iota_1), (Z, u)) \simeq \Hom_\C(Y,
    \cotr{Z}{X})
  \]
  and
  \[
    \Hom_{\cotr{\C}{Y}}((X \join Y, \iota_2), (Z, v)) \simeq \Hom_\C(X,
    \tr{Z}{Y}).
  \]
  Similarly to what happened in the case of biclosed monoidal
  categories, the functors $(Z, u : X \to Z) \mapsto \cotr{Z}{X}$ and
  $(Z, v : Y \to Z) \mapsto \tr{Z}{Y}$ can be made functorial in $X$
  and~$Y$, respectively. More precisely, they canonically extend to
  functors
  \[
    \begin{split}
      \twAr(\C) & \to \C \\
      X \to Z & \mapsto \cotr{Z}{X}
    \end{split}
    \qquad \qquad
    \begin{split}
      \twAr(\C) & \to \C \\
      Y \to Z & \mapsto \tr{Z}{Y},
    \end{split}
  \]
  where $\twAr(\C)$ denotes the twisted arrow category of
  $\C$. (Recall that the objects of $\twAr(\C)$ are arrows
  $f : X \to Y$ of $\C$ and that a morphism of $\twAr(\C)$ from an
  object $f : X \to Y$ to an object $f' : X' \to Y'$ is a pair of
  morphisms $g : X' \to X$ and $h : Y \to Y'$ of $\C$ making the
  square
  \[
    \xymatrix{
      X \ar[d]_f & X' \ar[l]_g \ar[d]^{f'} \\
      Y \ar[r]_h & Y' }
  \]
  commute.)  Indeed, if
  \[
    \xymatrix{
      X \ar[d]_u & X' \ar[l]_g \ar[d]^{u'} \\
      Z \ar[r]_h & Z' }
  \]
  is a commutative square, one defines a morphism
  \[ (g^\ast, h_\ast) : \cotr{Z}{X} \to \cotr{Z'}{X'} \] using the
  Yoneda lemma. If $T$ is any object of $\C$, then
  \[
    \begin{split}
      \Hom_{\C}(T, \cotr{Z}{X}) & \simeq \Hom_{\cotr{\C}{X}}((X \join
      T,
      \iota_1), (Z, u)) \\
      \Hom_{\C}(T, \cotr{Z'}{X'}) & \simeq \Hom_{\cotr{\C}{X'}}((X'
      \join T, \iota_1), (Z', u'))
    \end{split}
  \]
  and the natural map
  \[
    \Hom_{\C}(g \join T, h) : \Hom_{\C}(X \join T, Z) \to \Hom_{\C}(X'
    \join T, Z')
  \]
  induces the desired morphism. Note that this morphism
  $(g^\ast, h_\ast)$ is the diagonal of the commutative square
  \[
    \xymatrix{ \cotr{Z}{X} \ar[r]^{h_\ast} \ar[d]_{g^\ast} &
      \cotr{Z'}{X}
      \ar[d]^{g^\ast} \\
      \cotr{Z}{X'} \ar[r]_{h_\ast} & \cotr{Z'}{X'} \pbox{,} }
  \]
  where $g^\ast = (g^\ast, {\id{}}_\ast)$ and
  $h_\ast = ({\id{}}^\ast, h_\ast)$. A similar construction applies to
  the other slice functor.
\end{paragraph}

\begin{remark}
  Let $\C$ be a monoidal category. If $\C$ is locally biclosed, then its
  tensor product preserves connected colimits in each variable (as
  the forgetful functor $\C \to \cotr{\C}{Z}$ preserves these
  colimits). By a classical adjoint theorem, the converse holds provided
  that the category $\C$ is locally presentable (as colimits in
  $\cotr{\C}{Z}$ can be computed as connected colimits in $\C$).
\end{remark}

\begin{paragraph}
  A \ndef{locally biclosed monoidal model category} is a monoidal
  model category whose underlying monoidal category is locally
  biclosed. Our example of interest in this paper is the folk model category
  structure on $\ooCat$ (or more generally $\nCat{n}$) endowed with the join
  (see Theorems~\ref{thm:join_mon} and \ref{thm:def_join}).
  Note that in the locally biclosed setting, the unit axiom is a consequence
  of the pushout-product axiom. This follows from Remark~\ref{rem:M1} as the
  tensor unit is the initial object and is thus cofibrant. Moreover,
  Proposition~\ref{prop:derived_tens} shows that the monoidal tensor $\join$
  of such a model category admits a total right derived functor $\Ljoin$.
\end{paragraph}

\begin{proposition}\label{prop:loc_closed_Quillen_adj}
  Let $\M$ be a locally biclosed monoidal model category. For every
  cofibrant object $X$ of $\M$, the adjoint pair
  \[
    \begin{split}
      \M & \to \cotr{\M}{X} \\
      Y & \mapsto (X \join Y, \iota_1)
    \end{split}
    \qquad \qquad
    \begin{split}
      \cotr{\M}{X} & \to \M \\
      (Z, X \to Z) & \mapsto \cotr{Z}{X}
    \end{split}
  \]
  is a Quillen pair and, likewise, for every cofibrant object $Y$ of
  $\M$, the adjoint pair
  \[
    \begin{split}
      \M & \to \cotr{\M}{Y} \\
      X & \mapsto (X \join Y, \iota_2)
    \end{split}
    \qquad \qquad
    \begin{split}
      \cotr{\M}{Y} & \to \M \\
      (Z, Y \to Z) & \mapsto \tr{Z}{Y}
    \end{split}
  \]
  is a Quillen pair.
\end{proposition}

\begin{proof}
  It suffices to show that the left adjoints respect cofibrations and
  trivial cofibrations. As the cofibrations and trivial cofibrations
  of $\cotr{\M}{Z}$ are defined using the forgetful functor
  $\cotr{\M}{Z} \to \M$, this follows from the pushout-product axiom by
  Remark~\ref{rem:M1}.
\end{proof}

The goal of the rest of this appendix is to derive the slice functors of a
locally biclosed monoidal model category $\M$, as functors of source
$\twAr(\M)$, where the weak equivalences of $\twAr(\M)$ are the level-wise
weak equivalences. Unfortunately, the category $\twAr(\M)$ is neither
finitely cocomplete (it does not even have an initial object) nor finitely
complete in general, and therefore cannot be endowed with a model category
structure.  We will see that it can be endowed with a right simplicial
derivability structure in the sense of Kahn and Maltsiniotis (see
\cite[Definition 6.7]{KahnMaltsi}) and that this is enough to derive the
slice functors.

\medbreak

We start by recalling this notion of right simplicial derivability
structure and the corresponding result of derivation.

\begin{paragraph}\label{paragr:localizer}
  A \ndef{localizer}, or \ndef{relative category}, is a pair
  $(\C, \W)$, where $\C$ is a category and $\W$ is a class of morphisms of
  $\C$ called \ndef{weak equivalences}. Such a localizer is said to be
  \ndef{multiplicative} if $\W$ contains all the identities and is stable
  under composition. In this case, the class $\W$ can be identified with a
  subcategory of $\C$ with same objects as $\C$.

  A \ndef{morphism} from a localizer $(\C, \W)$ to a localizer $(\C', \W')$
  is a functor $F : \C \to \C'$ such that $F(\W) \subset \W'$.

  If $(\C, \W)$ is a localizer and $I$ is a small category, then we get a
  localizer $(\C_I, \W_I)$, where $\C_I$ denotes the category $\Homi(I, \C)$
  of functors from $I$ to $\C$ and $\W_I$ the class of natural
  transformations between these functors which are object-wise weak
  equivalences. This construction is functorial in $I$ in an obvious way: if
  $F : (\C, \W) \to (\C', \W')$ is a morphism of localizers, we get a
  morphism $F_I : (\C_I, \W_I) \to (C'_I, \W'_I)$.
\end{paragraph}

\begin{paragraph}
  Fix $K : (\C_0, \W_0) \to (\C, \W)$ a morphism of multiplicative
  localizers and denote by $K^\flat$ the induced functor $K^\flat : \W_0 \to \W$.
  If $X$ is an object of $\C$, the \ndef{category of right $K\!$-resolutions}
  of $X$ is the comma category $X \comma K^\flat$, that is, the category whose
  objects are pairs $(Y, X \xto{w} KY)$, where $Y$ is an object of $\C_0$
  and $w$ is a weak equivalence of $\C$, and whose morphisms from an object
  $(Y, w)$ to an object $(Y', w')$ are the weak equivalences $w_0 : Y \to Y'$ of
  $\C_0$ such that $K(w_0)w = w'$.

  If $I$ is a small category and $F : I \to \C$ is a functor, then,
  by considering the induced morphism of localizers $K_I$, we get a notion of
  category of right $K_I$-resolutions for $F$. In particular, considering
  $I = \{0 < 1\}$, we get a notion of \ndef{category of right
  $K\!$-resolutions of an arrow} of $\C$, and taking $I = \{0 < 1 < 2\}$, we
  get a notion of \ndef{category of right $K\!$-resolutions of a pair of
  composable arrows} of $\C$.
\end{paragraph}

\begin{paragraph}
  Let $(\C, \W)$ be a multiplicative localizer. A \ndef{right simplicial
  derivability structure} on $(\C, \W)$ consists of a multiplicative
  localizer $(\C_0, \W_0)$ and a morphism of localizers $K : (\C_0, \W_0)
  \to (\C, \W)$ satisfying the following conditions:
  \begin{enumerate}
    \item for every object $X$ of $\C$, the category of right $K$-resolutions of
      $X$ is $1$-connected (that is, simply connected and non-empty),
    \item for every arrow $f$ of $\C$, the category of right $K$-resolutions of
      $f$ is $0$-connected (that is, connected and non-empty),
    \item for every pair $(g, f)$ of composable arrows of $\C$, the category
      of right $K$\=/resolutions of $(g, f)$ is $-1$-connected (that is, non-empty).
  \end{enumerate}
\end{paragraph}

\begin{example}\label{ex:model_deriv}
  If $\M$ is a model category, then $(\M, \W)$, where $\W$ is the class of
  weak equivalences of $\M$, is naturally endowed with a right simplicial
  derivability structure $K : (\M_0, \W_0) \to (\M, \W)$, where $\M_0$
  denotes the full subcategory of $\M$ consisting of fibrant objects and
  $\W_0$ the class of weak equivalences between fibrant objects (see the
  ``table of implications'' at the very end of \cite{KahnMaltsi}).
\end{example}

\begin{proposition}[Kahn--Maltsiniotis]\label{prop:Kahn-Maltsi}
  Let $F : (\C, \W) \to (\C', \W')$ be a morphism of localizers. If there
  exists a right simplicial derivability structure $K : (\C_0, \W_0) \to
  (\C, \W)$ on $(\C, \W)$ such that $FK(\W_0) \subset \W'$, then $F$ admits
  a total right derived functor $\mathbb{R}F : \C[\W^{-1}] \to
  \C'[\W'^{-1}]$.
\end{proposition}

\begin{proof}
  See \cite[Proposition 5.9 and paragraph 6.8]{KahnMaltsi}.
\end{proof}

We will now prove a general result allowing to lift a right simplicial
derivability structure along a discrete opfibration, result that we will
then apply to the discrete opfibration $\twAr(\M) \to \M^\op \times \M$.

\begin{proposition}\label{prop:disc_opfib_simpl}
  Let $(\C, \W)$ be a multiplicative localizer endowed with a right
  simplicial derivability structure $K : (\C_0, \W_0) \to (\C, \W)$ and let
  $p : \Ct \to \C$ be a discrete opfibration. Set $\Wt = p^{-1}(\W)$. Then
  $(\Ct, \Wt)$ is endowed with a natural simplicial derivability structure
  $\Kt : (\Ct_0, \Wt_0) \to (\Ct, \Wt)$, obtained by pulling back $K$ along
  $p$.
\end{proposition}

\begin{proof}
  Let $I$ be a small category and let $\Ft : I \to \Ct$ be a functor. We are
  going to show that the categories of right $\Kt_I$-resolutions of $\Ft$
  and of right $K_I$-resolutions of~$p\Ft$ are isomorphic. This will
  immediately imply the result.

  By definition, we have a pullback square
  \[
    \xymatrix{
      (\Ct_0, \Wt_0) \ar[r]^{\Kt} \ar[d] & (\Ct, \Wt) \ar[d]^{p} \\
      (\C_0, \W_0) \ar[r]_{K} & (\C, \W)
    }
  \]
  in the category of localizers. Note that pullbacks in this category are
  computed component-wise. By applying the $\Homi(I, {-})$ functor, we get
  a commutative square
  \[
    \xymatrix{
      ((\Ct_0)_I, (\Wt_0)_I) \ar[r]^-{\Kt_I} \ar[d] & (\Ct_I, \Wt_I) \ar[d]^{p_I} \\
      ((\C_0)_I, (\W_0)_I) \ar[r]_-{K_I} & (\C_I, \W_I)
      \pbox{,}
    }
  \]
  that is easily seen to still be a pullback square. As the $\Homi(I, {-})$
  functor preserves discrete opfibrations, the functor $p_I : \Ct_I \to
  \C_I$ and therefore its restriction $(p_I)^\flat : \Wt_I \to \W_I$ are
  still discrete opfibrations.

  By definition, the category of right $\Kt_I$-resolutions of $\Ft : I \to
  \Ct$ is the comma category $\Ft \comma (\Kt_I)^\flat$, while the category of
  right $K_I$-resolutions of $p_I(\Ft) = p\Ft : I \to \C$ is the comma
  category $p_I(\Ft) \comma (K_I)^\flat$. The result thus follows from the
  following lemma, applied to the pullback square
  \[
    \xymatrix@C=2.5pc{
      (\Wt_0)_I \ar[r]^-{(\Kt_I)^\flat} \ar[d] & \Wt_I
      \ar[d]^{(p_I)^\flat} \\
      (\W_0)_I \ar[r]_-{(K_I)^\flat} & \W_I \pbox{,}
    }
  \]
  lemma which is probably well known and whose proof is left as an easy exercise
  to the reader.
\end{proof}

\begin{lemma}
  \newcommand\X{\mathcal{X}}
  \newcommand\B{\mathcal{B}}
  Let
  \[
    \xymatrix{
      \X' \ar[r]^G \ar[d] & \X \ar[d]^p \\
      \B' \ar[r]_F & \B
    }
  \]
  be a pullback square of categories, where $p$ is a discrete opfibration.
  Then, for every object $X$ of $\X$, the functor $p$ induces an
  isomorphism between the comma categories~$X \comma G$ and $p(X) \comma
  F$.
\end{lemma}

\begin{paragraph}
  Let $\M$ be a model category. We will say that a morphism
  \[
    \xymatrix{
      X \ar[d]_f & X' \ar[l]_g \ar[d]^{f'} \\
      Y \ar[r]_h & Y'
    }
  \]
  of $\twAr(\M)$ from $f$ to $f'$ is
  \begin{itemize}
    \item a \ndef{weak equivalence} if $g$ and $h$ are,
    \item a \ndef{fibration} if $g$ is a cofibration and $h$ is a
     fibration.
  \end{itemize}
  The category $\twAr(\M)$ admits as a terminal object the unique arrow
  $\emptyset \to \ast$ from the initial object of $\M$ to the terminal object
  of $\M$, and we will say that an object $X \to Y$ of $\twAr(\M)$ is
  \ndef{fibrant} if the unique morphism from this object to the terminal
  object is a fibration. This amounts to saying that $X$ is cofibrant and
  $Y$ is fibrant.

  We will denote by $(\twAr(\M), \Wt)$ the resulting localizer and by
  $(\twAr(\M)_0, \Wt_0)$ the induced localizer on the full subcategory of
  $\twAr(\M)$ consisting of fibrant objects.
\end{paragraph}

\begin{proposition}\label{prop:der_struct_tw}
  If $\M$ is a model category, then the inclusion morphism
  \[ (\twAr(\M)_0, \Wt_0) \hookto (\twAr(\M), \Wt) \]
  is a right simplicial derivability structure.
\end{proposition}

\begin{proof}
  It is immediate that the functor
  \[
    \begin{split}
      \twAr(\M) & \to \M^\op \times \M \\
      X \to Y & \mapsto (X, Y)
    \end{split}
  \]
  is a discrete opfibration. We thus get the result by applying
  Proposition~\ref{prop:disc_opfib_simpl} to this functor and to the right
  simplicial derivability structure associated to the model category~$\M^\op
  \times \M$ (see Example~\ref{ex:model_deriv}).
\end{proof}

To use Proposition~\ref{prop:Kahn-Maltsi} to derive the slice functors, we
now need to prove that these functors preserve weak equivalences between fibrant
objects. To do so, we will generalize
Proposition~\ref{prop:def_closed_mon_model} to the locally biclosed setting.

\begin{paragraph}\label{paragr:def_trp}
  Let $\C$ be a locally biclosed monoidal category. If $i : A \to B$
  and $j : C \to D$ are two morphisms of $\C$, note that the morphism
  \[ i \boxjoin j : B \join C \amalg_{A \join C} A \join D \to B \join
    D
  \]
  is naturally above both $B$ and $D$. If now $p : (X, f) \to (Y, g)$
  is a morphisms of $\cotr{\C}{B}$, using $i$ and $p$ we get a morphism
  \[
    \cotrp{p}{i} : \cotr{X}{B} \to \cotr{X}{A} \times_{\cotr{Y}{A}}
    \cotr{Y}{B}
  \]
  induced by the commutative square
  \[
    \xymatrix{ \cotr{X}{B} \ar[r]^{p_\ast} \ar[d]_{i^\ast} &
      \cotr{Y}{B}
      \ar[d]^{i^\ast} \\
      \cotr{X}{A} \ar[r]_{p_\ast} & \cotr{Y}{A} \pbox{.}  }
  \]
  Similarly, from $j$ and a morphism $p : (X, f) \to (Y, g)$ of
  $\cotr{\C}{D}$, we get a morphism
  \[
    \trp{p}{j} : \tr{X}{D} \to \tr{X}{C} \times_{\tr{Y}{C}} \tr{Y}{D}.
  \]
\end{paragraph}

\begin{lemma}[Joyal]
  Let $\C$ be a locally biclosed monoidal category. If $i : A \to B$
  and $j : C \to D$ are two morphisms of $\C$ and $p : X \to Y$ is a
  morphism of $\C$ above $D$, then we have
  \[ i \join' j \perp_{\cotr{\C}{D}} p \quadiff j \perp_{\C}
    \cotrp{p}{i} \quadiff i \perp_{\C} \trp{p}{j},
  \]
  where $\perp_{\mathcal{D}}$ denotes the relation of weak
  orthogonality in the category $\mathcal{D}$.
\end{lemma}

\begin{proof}
  The lemma is inspired by \cite[Lemma 3.6]{JoyalQuasiKan}, whose
  proof applies \forlang{mutatis mutandis}.
\end{proof}

\begin{proposition}\label{prop:def_loc_closed_model_cat}
  Let $\M$ be a model category endowed with a locally biclosed
  monoidal category structure. Then the following conditions are
  equivalent:
  \begin{enumerate}[label=\roman*)]
  \item the tensor product $\join$ satisfies the pushout-product
    axiom,
  \item for every cofibration $i : A \to B$, every fibration
    $p : X \to Y$ and every map $f : B \to X$, the induced map
    \[
      \cotrp{p}{i} : \cotr{X}{B} \to \cotr{X}{A} \times_{\cotr{Y}{A}}
      \cotr{Y}{B}
    \]
    is a fibration that is trivial if either $i$ or $p$ is,
  \item for every cofibration $j : C \to D$, every fibration
    $p : X \to Y$ and every map $f : D \to X$, the induced map
    \[
      \trp{p}{j} : \tr{X}{D} \to \tr{X}{C} \times_{\tr{Y}{C}}
      \tr{Y}{D}
    \]
    is a fibration that is trivial if either $j$ or $p$ is.
  \end{enumerate}
\end{proposition}

\begin{proof}
  This follows directly from the previous lemma and the fact that
  \[
    \begin{split}
      i \boxjoin j \perp_\M p & \quadiff \text{for every
        $f : B \to X$, we have
        $i \boxjoin j \perp_{\cotr{\M}{B}} p$,}  \\
      & \quadiff \text{for every $f : D \to X$, we have
        $i \boxjoin j \perp_{\cotr{\M}{D}} p$.} \qedhere
    \end{split}
  \]
\end{proof}

\begin{proposition}\label{prop:twisted_tr_Quillen}
  If $\M$ is a locally biclosed monoidal model category, then the
  functors
  \[
    \begin{split}
      \twAr(\M) & \to \M \\
      X \to Z & \mapsto \cotr{Z}{X}
    \end{split}
    \qquad \qquad
    \begin{split}
      \twAr(\M) & \to \M \\
      Y \to Z & \mapsto \tr{Z}{Y}
    \end{split}
  \]
  both send fibrations \resp{trivial fibrations} between fibrant
  objects to fibrations \resp{trivial fibrations}. Moreover, they preserve
  weak equivalences between fibrant objects.
\end{proposition}

\begin{proof}
  Let us prove the result for the first functor, the proof for the
  second one being similar. Let
  \[
    \xymatrix{
      X \ar[d]_u & X' \ar[l]_g \ar[d]^{u'} \\
      Z \ar[r]_h & Z'
    }
  \]
  be a morphism of $\twAr(\M)$ between fibrant objects $u$ and $u'$. The
  morphism
  \[ (g^\ast, h_\ast) : \cotr{Z}{X} \to \cotr{Z'}{X'} \]
  factors as
  \[ \cotr{Z}{X} \xto{g^\ast} \cotr{Z}{X'} \xto{h_\ast} \cotr{Z'}{X'}. \]
  These morphisms $g^\ast$ and $h_\ast$ are the images of $g$ and
  $h$ by the functors
  \[
    \begin{split}
      (\tr{\M}{Z})^\op & \to \M \\
      (X, X \to Z) & \mapsto \cotr{Z}{X}
    \end{split}
    \qquad \qquad
    \begin{split}
      \cotr{\M}{X'} & \to \M \\
      (Z, X' \to Z) & \mapsto \cotr{Z}{X'},
    \end{split}
  \]
  and it therefore suffices to show that the first of these functors sends
  cofibrations \resp{trivial cofibrations} of $\tr{\M}{Z}$ to fibrations
  \resp{trivial fibrations} of~$\M$ and that the second one preserves
  fibrations and trivial fibrations. Note that by Ken Brown's lemma (which
  cannot be applied directly to $\twAr(\M)$), this will imply that the first
  functor preserves weak equivalences between cofibrant objects of
  $\tr{\M}{Z}$ (an object of $\tr{\M}{Z}$ being cofibrant if its underlying
  object in $\M$ is) and that the second functor preserves weak equivalences
  between fibrant objects in $\cotr{\M}{X'}$ (an object of $\cotr{\M}{X'}$
  being fibrant if its underlying object in $\M$ is), thereby proving the
  second assertion.

  For the first functor, observe that the morphism $g^\ast$ can be
  identified with the morphism
  \[
    \cotrp{p}{g} : \cotr{Z}{X} \to \cotr{Z}{X'}
    \times_{\cotr{\ast}{X'}} \cotr{\ast}{X}
  \]
  of paragraph~\ref{paragr:def_trp}, where $p : Z \to \ast$ denotes
  the unique morphism from $Z$ to the terminal object $\ast$.
  As $Z$ is fibrant, Proposition~\ref{prop:def_loc_closed_model_cat} implies
  that this first functor sends cofibrations \resp{trivial cofibrations} of
  $\tr{\M}{Z}$ to fibrations \resp{trivial fibrations} of $\M$. As for the
  second functor, since $X'$ is cofibrant, it preserves fibrations and trivial
  fibrations by Proposition~\ref{prop:loc_closed_Quillen_adj}.
\end{proof}

\begin{theorem}
  If $\M$ is a locally biclosed monoidal model category, then the
  functors
  \[
    \begin{split}
      \twAr(\M) & \to \M \\
      X \to Z & \mapsto \cotr{Z}{X}
    \end{split}
    \qquad \qquad
    \begin{split}
      \twAr(\M) & \to \M \\
      Y \to Z & \mapsto \tr{Z}{Y}
    \end{split}
  \]
  both admit total right derived functors.
\end{theorem}

\begin{proof}
  Since by the previous proposition these functors preserve weak equivalences
  between fibrant objects, the result follows from the derivability
  condition of Kahn and Maltsiniotis (Proposition~\ref{prop:Kahn-Maltsi})
  applied to the right simplicial derivability structure of
  Proposition~\ref{prop:der_struct_tw}.
\end{proof}

\begin{corollary}
  The functors
  \[
    \begin{split}
      \twAr(\ooCat) & \to \ooCat \\
      X \xto{u} Z & \mapsto \cotr{Z}{u}
    \end{split}
    \qquad \qquad
    \begin{split}
      \twAr(\ooCat) & \to \ooCat \\
      Y \xto{v} Z & \mapsto \trm{Z}{v}
    \end{split}
  \]
  (see paragraph~\ref{paragr:def_join} for the notation), where $\ooCat$ is
  endowed with the folk model category structure, admit total right derived
  functors.
\end{corollary}

\begin{proof}
  This follows from the previous theorem applied to the folk model category
  structure on $\ooCat$ endowed with the join (see
  Theorems~\ref{thm:join_mon} and \ref{thm:def_join}).
\end{proof}

\bibliography{biblio}

\providecommand{\bysame}{\leavevmode ---\ }
\providecommand{\og}{``}
\providecommand{\fg}{''}
\providecommand{\smfandname}{\&}
\providecommand{\smfedsname}{\'eds.}
\providecommand{\smfedname}{\'ed.}
\providecommand{\smfmastersthesisname}{M\'emoire}
\providecommand{\smfphdthesisname}{Th\`ese}
\begin{thebibliography}{10}

\bibitem{AlAglSteiner}
{\scshape F.~A. Al-Agl {\normalfont \smfandname} R.~Steiner} -- {\og Nerves of
  multiple categories\fg}, \emph{Proc. London Math. Soc. (3)} \textbf{66}
  (1993), no.~1, p.~92--128.

\bibitem{AraMaltsiJoint}
{\scshape D.~Ara {\normalfont \smfandname} G.~Maltsiniotis} -- {\og Joint et
  tranches pour les $\infty$-cat{\'e}gories strictes\fg}, \emph{M\'{e}m. Soc.
  Math. Fr. (N.S.)} \textbf{165} (2020).

\bibitem{AraMetGpd}
{\scshape D.~Ara {\normalfont \smfandname} F.~M{\'e}tayer} -- {\og The
  {B}rown-{G}olasi\'nski model structure on strict {$\infty$}\nbd-groupoids
  revisited\fg}, \emph{Homology Homotopy Appl.} \textbf{13} (2011), no.~1,
  p.~121--142.

\bibitem{BergerMoerdijk}
{\scshape C.~Berger {\normalfont \smfandname} I.~Moerdijk} -- {\og Axiomatic
  homotopy theory for operads\fg}, \emph{Comment. Math. Helv.} \textbf{78}
  (2003), no.~4, p.~805--831.

\bibitem{CransThese}
{\scshape S.~Crans} -- {\og On combinatorial models for higher dimensional
  homotopies\fg}, \smfphdthesisname, Utrecht University, 1995, under the
  supervision of I. Moerdijk et D. van Dalen.

\bibitem{DayRefl}
{\scshape B.~Day} -- {\og A reflection theorem for closed categories\fg},
  \emph{J. Pure Appl. Algebra} \textbf{2} (1972), no.~1, p.~1--11.

\bibitem{GrayFCT}
{\scshape J.~W. Gray} -- \emph{{\selectlanguage{english}Formal category theory:
  adjointness for {$2$}-categories}}, Lecture Notes in Mathematics, vol. 391,
  Springer-Verlag, 1974.

\bibitem{HadziThesis}
{\scshape A.~Hadzihasanovic} -- {\og The algebra of entanglement and the
  geometry of composition\fg}, \smfphdthesisname, University of Oxford, 2017,
  under the supervision of B. Coecke.

\bibitem{Harper}
{\scshape J.~E. Harper} -- {\og Homotopy theory of modules over operads and
  non-{$\Sigma$} operads in monoidal model categories\fg}, \emph{J. Pure Appl.
  Algebra} \textbf{214} (2010), no.~8, p.~1407--1434.

\bibitem{Hovey}
{\scshape M.~Hovey} -- \emph{Model categories}, Mathematical Surveys and
  Monographs, vol.~63, American Mathematical Society, 1999.

\bibitem{JoyalQuasiKan}
{\scshape A.~Joyal} -- {\og Quasi-categories and {K}an complexes\fg}, \emph{J.
  Pure Appl. Algebra} \textbf{175} (2002), no.~1-3, p.~207--222, Special volume
  celebrating the 70th birthday of Professor Max Kelly.

\bibitem{KahnMaltsi}
{\scshape B.~Kahn {\normalfont \smfandname} G.~Maltsiniotis} -- {\og Structures
  de d\'{e}rivabilit\'{e}\fg}, \emph{Adv. Math.} \textbf{218} (2008), no.~4,
  p.~1286--1318.

\bibitem{LackFolk2}
{\scshape S.~Lack} -- {\og A {Q}uillen model structure for 2-categories\fg},
  \emph{$K$-Theory} \textbf{26} (2002), no.~2, p.~171--205.

\bibitem{LackFolkBi}
\bysame , {\og A {Q}uillen model structure for bicategories\fg},
  \emph{$K$-Theory} \textbf{33} (2004), no.~3, p.~185--197.

\bibitem{Folk}
{\scshape Y.~Lafont, F.~M{\'e}tayer {\normalfont \smfandname} K.~Worytkiewicz}
  -- {\og A folk model structure on omega-cat\fg}, \emph{Adv. Math.}
  \textbf{224} (2010), no.~3, p.~1183--1231.

\bibitem{LucasThesis}
{\scshape M.~Lucas} -- {\og Cubical categories for homotopy and rewriting\fg},
  \smfphdthesisname, University Paris Diderot -- Paris 7, 2018, under the
  supervision of Y. Guiraud and P.-L. Curien.

\bibitem{MetCof}
{\scshape F.~M{\'e}tayer} -- {\og Cofibrant objects among higher-dimensional
  categories\fg}, \emph{Homology, Homotopy Appl.} \textbf{10} (2008), no.~1,
  p.~181--203.

\bibitem{Morel}
{\scshape F.~Morel} -- {\og Th\'{e}orie homotopique des sch\'{e}mas\fg},
  \emph{Ast\'{e}risque} (1999), no.~256, p.~vi+119.

\bibitem{Muro}
{\scshape F.~Muro} -- {\og Homotopy theory of nonsymmetric operads\fg},
  \emph{Algebr. Geom. Topol.} \textbf{11} (2011), no.~3, p.~1541--1599.

\bibitem{Quillen}
{\scshape D.~G. Quillen} -- \emph{Homotopical algebra}, Lecture Notes in
  Mathematics, vol.~43, Springer-Verlag, 1967.

\bibitem{SchwedeShipley}
{\scshape S.~Schwede {\normalfont \smfandname} B.~E. Shipley} -- {\og Algebras
  and modules in monoidal model categories\fg}, \emph{Proc. London Math. Soc.
  (3)} \textbf{80} (2000), no.~2, p.~491--511.

\bibitem{Steiner}
{\scshape R.~Steiner} -- {\og Omega-categories and chain complexes\fg},
  \emph{Homology Homotopy Appl.} \textbf{6} (2004), no.~1, p.~175--200.

\bibitem{StreetOrient}
{\scshape R.~Street} -- {\og The algebra of oriented simplexes\fg}, \emph{J.
  Pure Appl. Algebra} \textbf{49} (1987), no.~3, p.~283--335.

\end{thebibliography}
\bibliographystyle{mysmfplain}

\end{document}